\documentclass[12pt,leqno]{amsart}
\usepackage{amsmath}
\usepackage{amssymb}
\usepackage{amsthm}

\usepackage[dvipsnames]{xcolor}
\definecolor{PTgreen}{HTML}{228833}
\definecolor{PTred}{HTML}{CC3311}

\usepackage[pagebackref]{hyperref}
\hypersetup{
	colorlinks=true, % make hyperlinks different color
	linkcolor=blue,  % color of \ref
	citecolor=PTgreen, % color of \cite
}

\usepackage[letterpaper,margin=1in]{geometry}

\usepackage{enumerate}
\usepackage[shortlabels]{enumitem}
\usepackage{mathtools}
\usepackage{etoolbox}
\apptocmd{\sloppy}{\hbadness 10000\relax}{}{}
\apptocmd{\sloppy}{\vbadness 10000\relax}{}{}
\usepackage[english]{babel}
\numberwithin{equation}{section}
\numberwithin{figure}{section}

\newtheorem{lemma}{Lemma}[section]
\newtheorem{prop}[lemma]{Proposition}
\newtheorem{thm}[lemma]{Theorem}
\newtheorem{cor}[lemma]{Corollary}

\theoremstyle{definition}
\newtheorem{rmk}[lemma]{Remark}
\newtheorem{defn}[lemma]{Definition}
\newtheorem{eg}[lemma]{Example}

\newcommand{\N}{\mathbb{N}}

\newcommand{\R}{\mathbb{R}}

\newcommand{\sphere}{\mathbb{S}}
\newcommand{\ra}{\rightarrow}
\newcommand{\da}{\downarrow}

\newcommand{\dist}{\mathop\mathrm{dist}\nolimits}
\newcommand{\norm}[1]{\left \lVert #1 \right \rVert}
\newcommand{\abs}[1]{\left \lvert #1 \right \rvert}
\newcommand{\divv}{\mathrm{div}}
\newcommand{\ol}{\overline}
\renewcommand{\Im}{\mathrm{Im}}

\newcommand{\HCP}[1]{\mathfrak{H}_{#1}}
\newcommand{\nin}{n^+_{\mathrm{in}}}
\newcommand{\nout}{n^+_{\mathrm{out}}}
\newcommand{\nodal}{\mathcal{N}}

\renewcommand{\epsilon}{\varepsilon}

\allowdisplaybreaks

\begin{document}

\title[Nodal domains of homogeneous caloric polynomials]{On the number of nodal domains of homogeneous caloric polynomials}
\author{Matthew Badger}
\author{Cole Jeznach}
\thanks{M.~Badger was partially supported by NSF DMS grant 2154047. C.~Jeznach was partially supported by Simons Collaborations in MPS grant 563916 and NSF DMS grant 2000288.}
\address{Department of Mathematics, University of Connecticut, Storrs, CT 06269-1009.}
\email{matthew.badger@uconn.edu}
\address{Departament de Matem\`{a}tiques, Universitat Aut\`{o}noma de Barcelona, Bellaterra, 08193, Spain}
\email{colejeffrey.jeznach@uab.cat} 
\date{October 21, 2025}
\addtocontents{toc}{\protect\setcounter{tocdepth}{2}}

\begin{abstract}
We investigate the minimum and maximum number of nodal domains across all time-dependent homogeneous caloric polynomials of degree $d$ in $\mathbb{R}^{n}\times\mathbb{R}$ (space $\times$ time), i.e., polynomial solutions of the heat equation satisfying $\partial_t p\not\equiv 0$ and $$p(\lambda x, \lambda^2 t) = \lambda^d p(x,t)\quad\text{for all $x \in \mathbb{R}^n$, $t \in \mathbb{R}$, and $\lambda > 0$.}$$ When $n=1$, it is classically known that the number of nodal domains is precisely $2\lceil d/2\rceil$. When $n=2$, we prove that the minimum number of nodal domains is 2 if $d\not \equiv 0\pmod 4$ and is 3 if $d\equiv 0\pmod 4$. When $n\geq 3$, we prove that the minimum number of nodal domains is $2$ for all $d$. Finally, we show that the maximum number of nodal domains is $\Theta(d^n)$ as $d\rightarrow\infty$ and lies between $\lfloor \frac{d}{n}\rfloor^n$ and $\binom{n+d}{n}$ for all $n$ and $d$. As an application and motivation for counting nodal domains, we confirm existence of the singular strata in Mourgoglou and Puliatti's two-phase free boundary regularity theorem for caloric measure.
\end{abstract}

\maketitle

\vspace{-.5in}

\tableofcontents

\section{Introduction}\label{s:intro}

A general motivation for studying caloric polynomials comes from their ubiquity in the theory of parabolic PDEs as finite order solutions of the heat equation. After showing that tangent functions of solutions to a parabolic PDE are homogeneous caloric polynomials (hereafter abbreviated \emph{hcps}), one can deduce strong unique continuation principles and estimate the dimension of nodal and singular sets of solutions \cite{Chen98}. In a similar vein, it was recently shown that zero sets of hcps appear as the supports of tangent measures in non-variational free boundary problems for caloric measure \cite{MP21}. Hcps also arise in geometric contexts such as understanding the dimension of ancient caloric functions on manifolds with polynomial growth \cite{CM21}.

In this paper, with a goal of confirming existence of singular strata in the aforementioned free boundary regularity problem for caloric measure (see \S\ref{ss:2-phase}), we investigate basic topology of nodal domains of hcps in $\R^{n+1}\equiv \R^n\times \R$ (space $\times$ time). In particular, for each ambient dimension and for each parabolic degree (see \S\ref{ss:defs}), we would like to determine the minimum and maximum number of possible nodal domains. This type of question has been studied extensively for spherical harmonics (see \S\ref{ss:harmonic} for a brief survey). Also, up to scaling by a constant, for each degree $d$, there is a unique hcp in $\R^{1+1}$ of degree $d$ and the number of nodal domains is precisely $2\lceil d/2\rceil$ (see \S\ref{sec:prop}). Thus, the problem is to determine what happens for time-dependent hcps in at least two space variables. We fully determine the minimum number of domains (with detailed constructions) and establish asymptotic bounds for the maximum number of domains. Notice that by a straight-forward application of the mean-value property for caloric functions, we know that if a non-constant solution to the heat equation has a zero, then the number of nodal domains is at least two (see Remark \ref{r:meanvalue}).
\begin{thm}[minimum number of nodal domains]
\label{thm:main}
When $n\geq 2$,  the minimum number $m_{n,d}$ of nodal domains of time-dependent homogeneous caloric polynomials in $\R^{n+1}$ of degree $d\geq 2$ satisfies (see Figure \ref{fig:mathematica_graphics})
\begin{equation}m_{2,d} = \begin{cases}
2, &  \text{when }d \not \equiv 0 \pmod 4, \\
3, &  \text{when }d \equiv 0 \pmod 4.
\end{cases}\end{equation}
When $n\geq 3$, we have $m_{n,d}=2$ for all $d\geq 2$.
\end{thm}

 \begin{figure}
\begin{center}\includegraphics[width=.3\textwidth]{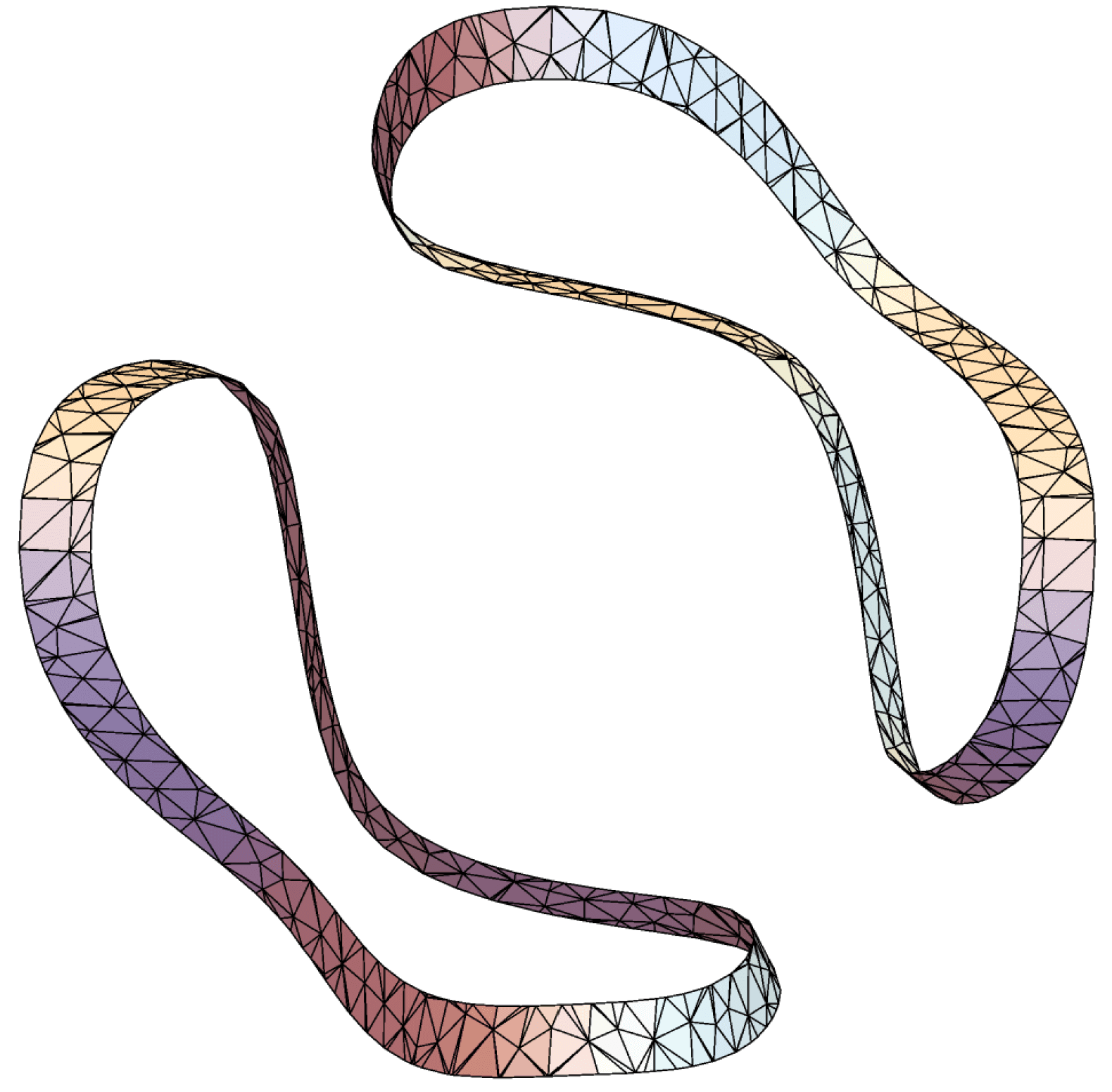}\hfill \includegraphics[width=.3\textwidth]{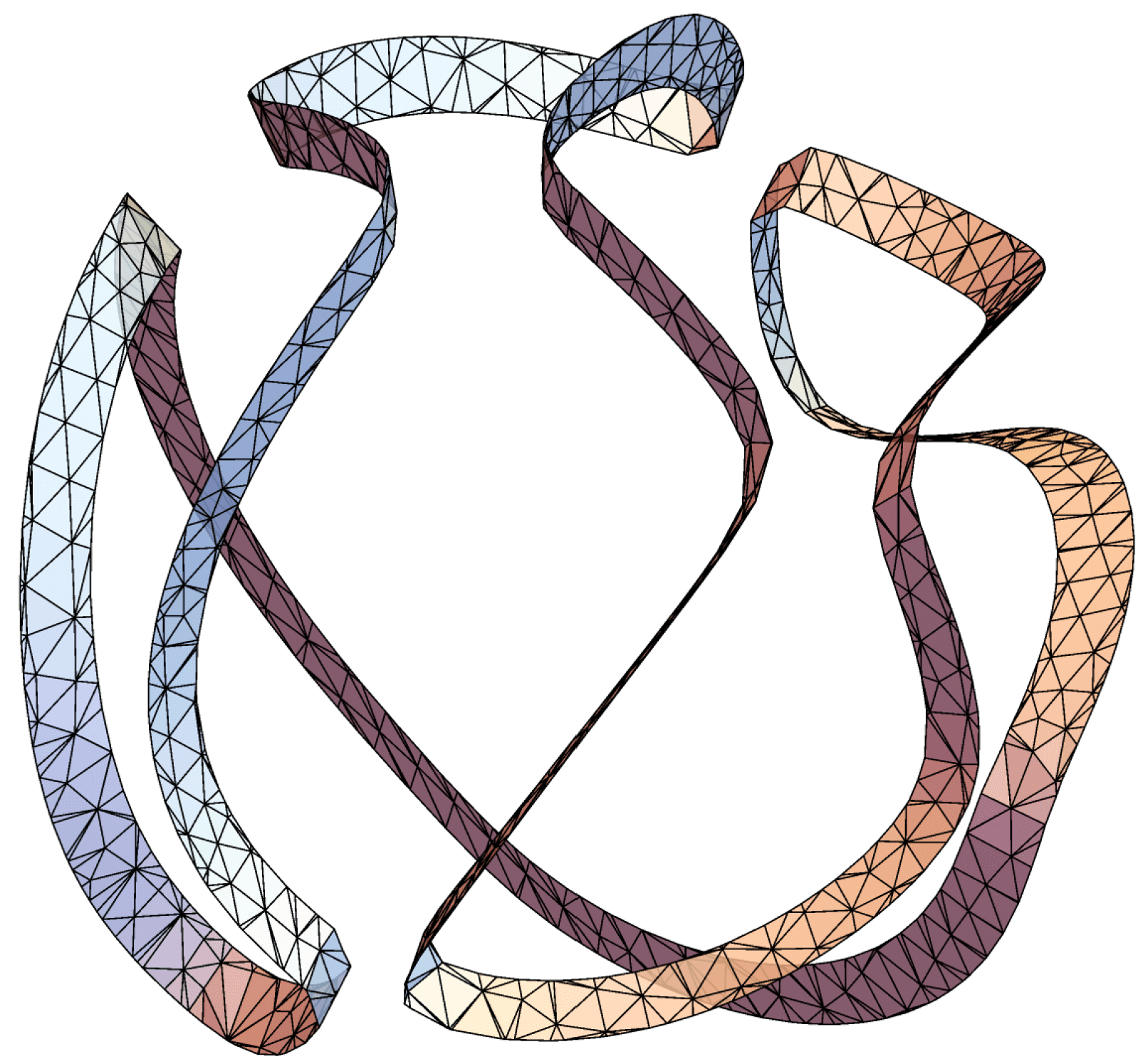}
\hfill\includegraphics[width=.3\textwidth]{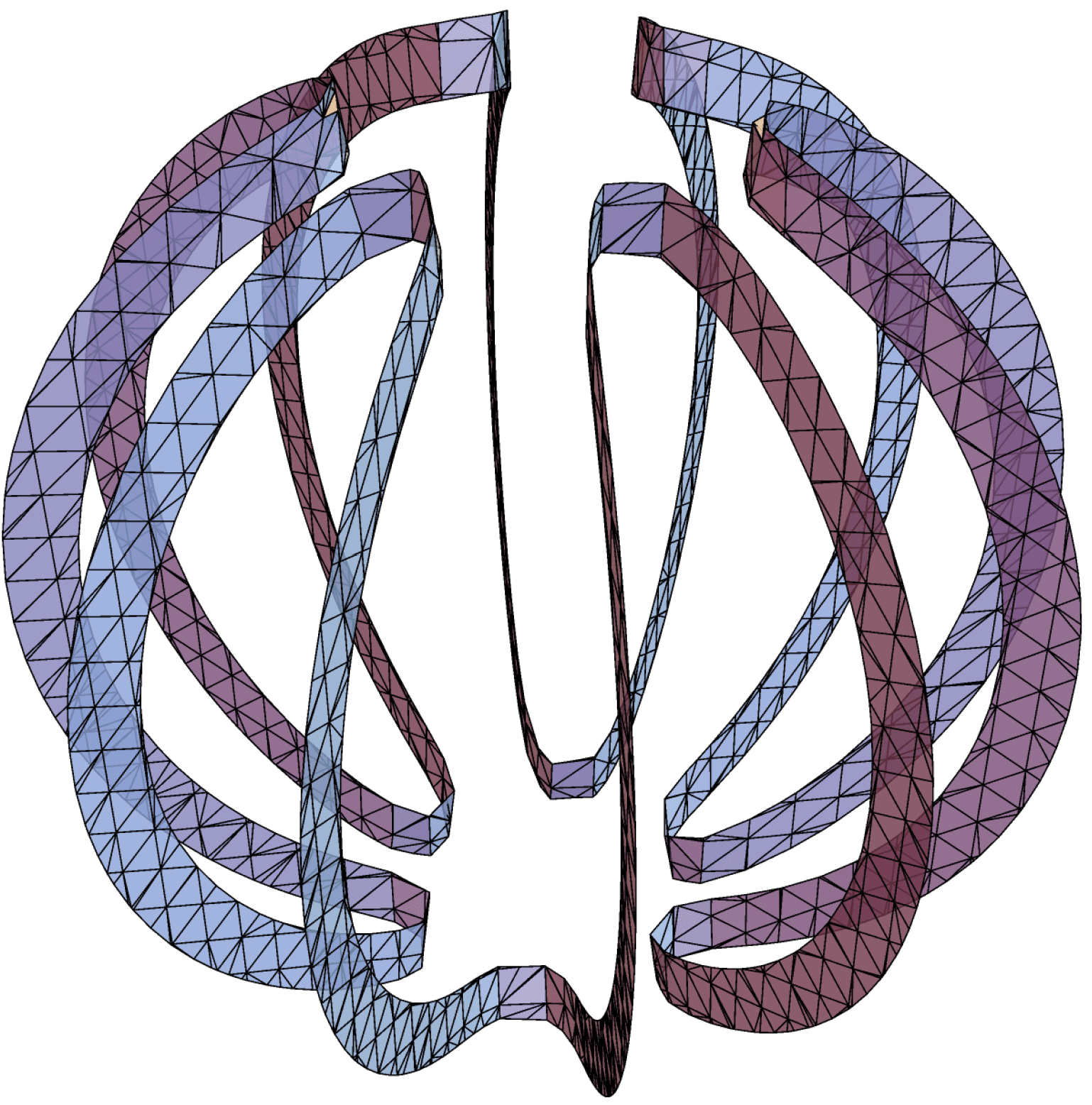}\end{center}\caption{Gallery of nodal sets of homogeneous caloric polynomials $u$ in \texorpdfstring{$\R^{2+1}$}{2+1 dimensions} achieving the minimum number $m_{2,d}$ of nodal domains. For increased visibility, we show the intersection of the full nodal set with an annulus: $\{ (x,y,t) \, : \, u(x,y,t) = 0  \} \cap (B_1 \setminus B_{1-\delta})$. Left: \eqref{eqn:u_0_mod_4} with $d=4$, $\epsilon = 0.2$, $\alpha = \pi/10$. Middle: \eqref{eqn:d_odd_u} with $d = 5$, $\epsilon =0.3$,  $\alpha = \pi/10$. Right: \eqref{eqn:lewy_u} with $d = 6$, $\epsilon = 0.05$.}\label{fig:mathematica_graphics}\end{figure}

We prove the case $n\geq 3$ of Theorem \ref{thm:main} in \S\ref{sec:high_dim} by modifying a construction from \cite{BET17} for homogeneous harmonic polynomials (hereafter abbreviated \emph{hhps}). The case $n=2$ is established in \S\S\ref{sec:dim_3_low} and \ref{sec:dim_3_constr}, partly by employing the perturbation technique from \cite{Lewy77}. Notably, the proof that $m_{2,4k}=3$ is the most difficult argument in the paper, and requires first proving a lower bound that exploits the structure of such polynomials (Corollary \ref{cor:0_mod_4}), and constructing an example that attains this lower bound (Theorem \ref{thm:0mod4}).

\begin{thm}[maximum number of nodal domains] \label{thm:max-asymptotics} For all $n\geq 2$, the maximum number $M_{n,d}$ of nodal domains of time-dependent homogeneous caloric polynomials in $\R^{n+1}$ of degree $d$ is $\Theta(d^n)$ as $d\rightarrow\infty$, i.e., there is a constant $C >0 $ so that $C^{-1} d^n \le M_{n,d} \le C d^n$ for all $d$ sufficiently large. More precisely, for all $n\geq 2$ and $d\geq 2$,
\begin{equation}
 \left\lfloor \frac{d}{n}\right\rfloor^n \le M_{n,d} \le \binom{n+d}{n}. \label{thebounds}
\end{equation}
\end{thm}

We prove Theorem \ref{thm:max-asymptotics} in \S\ref{sec:high_dim}. The lower bound on $M_{n,d}$ in \eqref{thebounds} is based on an elementary construction using  products of hcps in $\R^{1+1}$. The upper bound follows from an indirect application of Courant's nodal domain theorem, exploiting a connection between negative time-slices of caloric polynomials in $\R^{1+1}$ and Hermite orthogonal polynomials.

\begin{rmk} Classical theorems in algebraic geometry imply that the maximal {possible} number of nodal domains {for a} polynomial $p:\R^{n+1}\rightarrow\R$ of degree $d$ is $\Theta(d^{n+1})$ as $d\rightarrow\infty$ and lies between $\binom{d}{0}+\cdots+\binom{d}{n+1}$ and $(d+1)^n(d+2)$; see \cite[Proposition 2.4]{stanley-parkcity} for the lower bound (from hyperplane arrangements) and \cite[Theorem 3]{M64} for the upper bound. Thus, perhaps not unsurprisingly, the {asymptotics for the number of nodal domains of }hcps is distinct from those of general polynomials. Finding the exact value of $M_{n,d}$ appears to be a difficult problem. Except for some low degrees ($d\leq 6$), the corresponding problem for spherical harmonics in $\mathbb{R}^3$ is also open \cite{Leydold96}. \end{rmk}

We discuss proof strategies for the main theorems and outline the paper in \S\ref{ss:organize}.

\subsection{Definitions and examples} \label{ss:defs}

Let us decode the terminology in the statements of the main theorems.
We study the heat equation $(\partial_t - \Delta) u = 0$ in $\R^{n+1} = \{ (x,t) : x \in \R^n, t \in \R \}$, where $\Delta\equiv \Delta_x=\sum_1^n \partial_{x_i}\partial_{x_i}$. The natural notion of homogeneity in this setting is anisotropic.

\begin{defn}
A function $f(x,t)$ is \textbf{parabolically homogeneous} of degree $d \in \R $ if
\begin{align}
f(\lambda x, \lambda^2 t) & = \lambda^d f(x,t)\quad\text{for all $\lambda > 0$ and $(x,t)$.}
\end{align}
\end{defn}

\begin{defn}
A polynomial $p(x,t)$ is a \textbf{homogeneous caloric polynomial (hcp) of degree $d$} if $p$ satisfies the heat equation and is parabolically homogeneous{} of degree $d \in \N=\{0,1,2,\dotsc\}$. We say that $p$ is \textbf{time-dependent} if $\partial_t p \not\equiv 0$.
\end{defn}

\begin{eg}\label{eg:hcp} For any exponent $k\in \N$ and multi-index
$\alpha\in\N^n$, the monomial $t^k x^\alpha$ is parabolically homogeneous{} of degree $2k + \abs{\alpha}=2k+\alpha_1+\cdots+\alpha_n$. In particular, if $p(x,t)$ is parabolically homogeneous{} of degree $d \in \N$, then each monomial in $p(x,t)$ has the form $t^k x^\alpha$ where $2k + \abs{\alpha} = d$. Parabolic and algebraic homogeneity are distinct notions; e.g.,
\begin{equation}
p(x,t)=  t^2 +tx^2 + x^4/12
\end{equation}
is an hcp of degree $4$ in $\R^{1+1}$, but $p$ is not \emph{algebraically} homogeneous. Nevertheless, the parabolic degree and the algebraic degree of an hcp always coincide.
\end{eg}

\begin{rmk} If $p$ is a time-dependent hcp, then the degree of $p$ is at least 2.
\end{rmk}

\begin{defn}
The \textbf{nodal domains} of a continuous function $u(x,t)$ on $\R^{n+1}$ are the connected components of the set $\{(x,t)  :  u(x,t) \ne 0\}$. We let $\nodal(u) \in \N \cup \{+\infty\}$ denote the number of nodal domains of $u$. The \textbf{nodal set} of $u$ is $\{(x,t):u(x,t)=0\}$.
\end{defn}

\begin{rmk}\label{r:sphere} Let $\sphere^n=\{(x,t)\in\R^{n+1}:x_1^2+\cdots+x_n^2+t^2=1\}$. By parabolic homogeneity, if $u$ is an hcp of degree $d\geq 1$ in $\R^{n+1}$, then $\mathcal{N}(u)$ is the number of connected components of $\{(x,t)\in\sphere^n:u(x,t)\neq 0\}$ in $\sphere^n$.\end{rmk}

\begin{rmk}\label{r:meanvalue} For any $n\geq 1$, the mean value property (e.g., see \cite[p.~50]{Evans-PDE}) implies that a non-constant solution $u:\R^{n+1}\rightarrow \R$ of the heat equation takes positive and negative values in any neighborhood of a zero of $u$. In particular, $m_{n,d}\geq 2$ for all $n\geq 1$ and $d\geq 2$.\end{rmk}

\begin{eg}\label{eq:deg_2}
For any $1\leq j\leq n$, the polynomial $p(x,t) = 2t + x_j^2$ in $\R^{n+1}$ is an hcp of degree $2$. Moreover, $p(x,t)$ has exactly two nodal domains: $\{ t > x_j^2\}$ and $\{t < x_j^2 \}$. Thus, for all $n\geq 1$, the minimum number of nodal domains of time-dependent hcps in $\R^{n+1}$ of degree 2 is $m_{n,2}=2$.
\end{eg}

\begin{eg}\label{eq:1d} Up to scaling by a constant multiple, for every $d\geq 1$, there exists a unique hcp $p_d(x,t)$ of degree $d$ in $\R^{1+1}$ and $m_{1,d}=M_{1,d}=\nodal(p_d)=2\lceil d/2\rceil$. See \S\ref{sec:prop} for the details.
\end{eg}

\begin{eg}\label{eq:n2d3} The polynomial \begin{equation}p(x,y,t)=150t(3x+y)+27x^3+267x^2y+144xy^2-64y^3\end{equation} is an hcp of degree 3 in $\R^{2+1}$ and $\nodal(p)=2$. This example can be found by evaluating \eqref{eqn:d_odd_u} with $d=3$, $\epsilon=1$, and $(\cos\alpha,\sin\alpha)=(3/5,4/5)$ and multiplying by a constant to obtain a polynomial with integer coefficients. It can be checked that $\nabla p(x,y,t)=0$ if and only if $(x,y,t)=(0,0,0)$.
\end{eg}

\begin{eg}\label{eq:n2d4} The polynomial \begin{equation}\begin{split} p(x,y,t)=7500 t^2 &+ 150t(37x^2-7xy+13y^2) \\ &+ 192 x^4 + 176 x^3 y + 1623 x^2 y^2 - 351 x y^3 - 108 y^4\end{split}\end{equation} is an hcp of degree 4 in $\R^{2+1}$ and $\nodal(p)=3$. This example can be found by evaluating \eqref{eqn:u_0_mod_4} with $d=4$, $\epsilon=1/2$, and $(\cos\alpha,\sin\alpha)=(3/5,4/5)$ and multiplying by a constant to obtain a polynomial with integer coefficients.
\end{eg}

\begin{eg} \label{eq:n3d4} The polynomial \begin{equation} p(x,y,z,t)=12t^2+12tx^2+x^4+y^4-6y^2z^2+z^4\end{equation} is an hcp of degree 4 in $\R^{3+1}$ and $\nodal(p)=2$. See Proposition \ref{prop:high_dim}. It can be checked that $\nabla p(x,y,z,t)=0$ if and only if $(x,y,z,t)=(0,0,0,0)$.\end{eg}

\subsection{Comparison with spherical harmonics and Grushin spherical harmonics}\label{ss:harmonic} Steady-state solutions of the heat equation on $\R^{n+1}$ correspond to harmonic functions on $\R^n$. Nodal geometry of homogeneous harmonic polynomials $p$ in $\R^n$ (also called \emph{solid harmonics}) and of the so-called \emph{spherical harmonics} $p|_{\mathbb{S}^{n-1}}$ is well-studied. See \cite{EJN07, NS09,Logunov-annals1,Logunov-annals2} for a short sample, including results for Laplace-Beltrami eigenfunctions on closed Riemannian manifolds beyond the sphere.

Parallel to the quantities $m_{n,d}$ and $M_{n,d}$ defined in Theorems \ref{thm:main} and \ref{thm:max-asymptotics}, we let $\tilde m_{n,d}$ and $\tilde M_{n,d}$ denote the minimum and maximum number of nodal domains of hhps in $\R^n$ of degree $d$, respectively. In the line ($n=1$), the only harmonic functions are affine and $\tilde m_{1,1}=\tilde M_{1,1}=2$ trivially. In the plane ($n=2$), since harmonic functions can be written as the real part of a complex-analytic function, the nodal sets of hhp degree $d$ are rotations of $\{(x,y):\mathrm{Re}(x+iy)^d=0\}$ and $\tilde m_{2,d}=\tilde M_{2,d}=2d$ for all $d\geq 1$. The situation becomes more interesting when $n\geq 3$.
Lewy \cite{Lewy77} proved that $\tilde m_{3,d}=2$ whenever $d\geq 1$ is odd and $\tilde m_{3,d}=3$ whenever $d\geq 2$ is even. An explicit example of a degree 3 hhp with exactly two nodal domains, \begin{equation}p(x,y,z) = x^3 - 3xy^2 + z^3 - (3/2)(x^2 + y^2)z,\end{equation} was found independently by Szulkin \cite{Szulkin78}. Badger, Engelstein, and Toro \cite{BET17} gave a simple construction (utilizing the explicit description of hhps in $\R^2$) that shows $\tilde m_{n,d}=2$ for all $n\geq 4$ and $d\geq 1$, independent of the parity of $d$.

When $n\geq 3$, any hhp $p$ in $\R^n$ of degree $d$ satisfies the equation $$-\Delta_{\sphere^{n-1}} p|_{\sphere^{n-1}} = d(d+n-2)p|_{\sphere^{n-1}},$$ where $\Delta_{\sphere^{n-1}}$ denotes the Laplace-Beltrami operator on the sphere. Courant's nodal domain theorem asserts that when listed with multiplicity, the $m$-th eigenfunction of the Laplace-Beltrami operator on a closed $C^1$ Riemannian manifold has at most $m$ nodal domains \cite{CH,local-courant}. Since the dimension of the vector space of hhps of degree $d \ge 2$ in $\R^n$ is exactly
$\binom{n+d-1}{n-1} - \binom{n+d-3}{n-1}$
(see \cite[Proposition 5.8]{ABR}), the maximal number of linearly independent hhps of degree at most $d$ is exactly $\binom{n+d-1}{n-1} + \binom{n+d-2}{n-1}=O(d^{n-1})$ as $d\rightarrow\infty$. Thus, $\tilde M_{n,d}=O(d^{n-1})$ as $d\rightarrow\infty$. It is known that the upper bound on $M_{n,d}$ provided by Courant's theorem is not sharp. See \cite{Leydold96} for further discussion and the (still to this day) state-of-the-art bounds on $M_{2,d}$.

In \cite{LTY15}, Liu, Tian, and Yang study the minimum number $\tilde m^G_{2,d}$ of nodal domains of \emph{Grushin spherical harmonics}, i.e.~parabolically homogenenous polynomial solutions $p(x,y,t)$ of the operator $L_G=\partial_x^2 + \partial_y^2 + (x^2 + y^2) \partial_t^2$ on $\R^{2+1}$. In particular, they prove that $\tilde m^G_{2,d}=2$ when $d\equiv 0\pmod{4}$, whereas $\tilde m^G_{2,d}\geq 3$ when $d\equiv 0\pmod{4}$; moreover, they provide examples that show $m^G_{2,4}=m^G_{2,8}=m^G_{2,12}=3$. The method of proof is the perturbation technique of Lewy \emph{op.~cit.} Other than the fact that the parabolic scaling is the natural scaling for solutions of the Grushin operator and the heat operator, there does not seem to be any immediate connection between Grushin spherical harmonics and hcps. To wit, when $p(x,y,t)$ is a Grushin spherical harmonic, so is $p(x,y,-t)$, whereas this strong symmetry property is not enjoyed by time-dependent solutions of the heat equation. Thus, the main results in \cite{LTY15} cannot be used to establish Theorem \ref{thm:main} or vice-versa.

\subsection{Free boundary regularity for caloric measure} \label{ss:2-phase} The phrase \emph{caloric measure} refers to a family of probability measures $\omega^{X,t}_\Omega$ that are supported on a subset of the boundary $\partial\Omega$ of a space-time domain $\Omega\subset \R^{n+1}=\R^{n}\times\R$ and indexed by the points $(X,t)\in\Omega$. They arise in connection with the Dirichlet problem for the heat equation. Stochastically, $\omega^{X,t}_\Omega(E)$ is the probability that the trace $(B(t+s),t-s)_{s\geq 0}$ of a Brownian traveler $B(s)$ starting at $B(t)=X$ and sent into the past first intersects $\partial\Omega$ inside the set $E\subset\R^{n+1}$. For a consolidated introduction to caloric measure, see \cite[\S3]{dim-caloric}, and for extensive background, see \cite{Watson} or \cite{Doob}. Recent progress on free boundary regularity for caloric measure was made by Mourgoglou and Puliatti \cite{MP21}, propelling the time-dependent theory for caloric measure closer to the better developed, time-independent theory for harmonic and elliptic measure (see e.g.~\cite{KPT,AM19,BETnu}). Among other results---and setting aside certain technical assumptions related to the heat potential theory---their work leads to the following description of the asymptotic shape of the free boundary in the two-phase setting.

\begin{thm}[Mourgoglou-Puliatti] \label{t:mp} Assume that $\Omega^+=\R^{n+1}\setminus\overline{\Omega^-}$ and $\Omega^-=\R^{n+1}\setminus\overline{\Omega^+}$ are complementary domains in $\R^{n+1}$ with a sufficiently regular (for heat potential theory), common boundary $\partial\Omega=\partial\Omega^+=\partial\Omega^-$. Let $\omega^\pm=\omega^{X_\pm,t_0}_{\Omega^\pm}$ be caloric measures for $\Omega^\pm$ with poles at $(X_\pm,t_0)\in\Omega^\pm$ or poles at infinity. If $\omega^\pm$ are doubling measures, $\omega^+\ll\omega^-\ll\omega^+$, and the Radon-Nikodym derivatives $d\omega^-/d\omega^+$ and $d\omega^+/d\omega^-$ are bounded continuous functions on $\partial\Omega\cap\{t\leq t_0\}$, then $\partial\Omega=\Gamma_1\cup\Gamma_2\cup\cdots\cup\Gamma_{d_0},$ where geometric blow-ups (tangent sets) $\Sigma=\lim_{i\rightarrow\infty} r_{i}^{-1}(\partial\Omega-x)$ of $\partial\Omega$ at $x\in\Gamma_d$ along sequence of scales $r_i\rightarrow 0$ are zero sets of homogeneous caloric polynomials $p$ of degree $d$ such that $\{p>0\}$ and $\{p<0\}$ are connected. Cf.~\cite[Theorems III, IV]{MP21}.
\end{thm}

The main results in this paper validate Mourgoglou and Puliatti's theory by {demonstrating that there exist situations where time-dependent geometric blow-ups exist}. Furthermore, we obtain a refined description of the free boundary in low dimensions.

\begin{thm} When $n=1$, {$\partial\Omega=\Gamma_1 \cup \Gamma_2$}. When $n=2$, \begin{equation}\partial\Omega=\bigcup_{k\geq 0} \Gamma_{4k+1}\cup\Gamma_{4k+2}\cup\Gamma_{4k+3};\end{equation} for every $d\not\equiv 0\pmod 4$, the stratum $\Gamma_d$ is nonempty for some pair of {complementary domains $\Omega_{\pm}$} {satisfying} the free boundary condition. When $n=3$, the stratum $\Gamma_d$ can be nonempty for every $d\geq 1$.\end{thm}

\begin{proof} When $n=1$, $\Gamma_d=\emptyset$ for all $d\geq 2$ by Example \ref{eq:1d}. When $n=2$, $\Gamma_d=\emptyset$ for all $d=4k$ by Theorem \ref{thm:main}. For the remaining pairs of $n$ and $d$, the examples of hcps $p$ in $\R^{n+1}$ of degree $d$ with $\nodal(p)=2$ constructed in the proof of Theorem \ref{thm:main} (see the proofs of Proposition \ref{prop:high_dim} and Theorems \ref{thm:2mod4} and \ref{thm:d_odd} for details) have smooth zero sets outside any neighborhood of the origin. This fact is enough to ensure that the domains $\Omega_p^+=\{p>0\}$ and $\Omega_p^-=\{p<0\}$ associated to $p$ satisfy the background regularity hypothesis in Theorem \ref{t:mp}. If $\omega_p^\pm$ denote the caloric measures on $\Omega_p^\pm$ with poles at infinity, then it is known that $\omega_p^+=\omega_p^-$ and $d\omega_p^-/d\omega_p^+\equiv 1$ (see \cite[\S6]{MP21}). Finally, by parabolic homogeneity, $\{p=0\}$ is the unique blow-up of $\partial\Omega_p^\pm=\{p=0\}$ at the origin. Therefore, $\Gamma_d$ is nonempty in $\partial\Omega^\pm_p$. \end{proof}

\subsection{Proof strategies and outline of the paper} \label{ss:organize}

In Section \ref{sec:prop}, we recall classical facts about hcps, including their connection with Hermite polynomials. We also introduce a basis of hcps of degree $d$ in $\R^{n+1}$, which we use in the constructions in Sections \ref{sec:high_dim} and \ref{sec:dim_3_constr}.

In Section \ref{sec:high_dim}, we first prove that $m_{n,d} = 2$ for $n \ge 3$ and any $d \ge 1$ by modifying a construction in \cite{BET17}. Next, we build time-dependent hcps in $\R^{n+1}$ with a large number of nodal domains by taking products of hcps in $\R^{1+1}$. Finally, after showing that any nodal domain of an hcp necessarily intersects $\{t=-1\}$, we employ the proof of Courant's nodal domain theorem on negative time slices of hcps to establish the upper bound on the number of nodal domains in Theorem \ref{thm:max-asymptotics}. This leaves us to determine the value of $m_{2,d}$ for $d \ge 1$.

In Section \ref{sec:dim_3_low}, we show that $m_{2,d} \ge 3$ whenever $d \ge 1$ satisfies $d \equiv 0 \pmod 4$. The main argument leverages the fact that if $p$ is an hcp of degree $d$ in $\R^{2+1}$, then the nodal set of $p|_{\sphere^2}$ near the north and south poles is asymptotic to the zero set of a homogeneous harmonic polynomial in two variables, which are easy to describe and are perfectly understood.

In Section \ref{sec:dim_3_constr}, we construct examples of time-dependent hcps of degree $d \ge 1$ in $\R^{2+1}$ with exactly $m_{2,d}$ nodal domains. The basic strategy dates back to \cite{Lewy77}: starting with an hcp of degree $d$ in $\R^{2+1}$ whose zero set we can explicitly describe, we perturb the polynomial to produce a time-dependent hcp with the desired number of nodal domains. This style of argument requires a careful analysis of how perturbation affects the topology of nodal sets; see Lemma \ref{l:graph} for a precise statement and Section \ref{s:appendix} for the proof of the lemma.

\section{Basic properties of homogeneous caloric polynomials}\label{sec:prop}

Given an hcp $p(x,t)$ of degree $d$, we typically shall choose to write $p(x,t)$ in the form
\begin{equation}
p(x,t) = t^m p_m(x) + t^{m-1} p_{m-1}(x) + \cdots  + p_0(x)\quad \text{for all $x\in\R^n$, $t\in\R$}. \label{eqn:std_form}
\end{equation}
Each coefficient $p_{m-j}=p_{m-j}(x)$ is necessarily an algebraically homogeneous polynomial of degree $d - 2(m-j)$ (see Example \ref{eg:hcp}). Moreover, applying the heat operator to $p$ and collecting like powers of $t$, we obtain the relations
\begin{equation}\label{cond:soln}
0 = \Delta p_m,\quad m p_m = \Delta p_{m-1},\quad \cdots, \quad
(m-j) p_{m-j} = \Delta p_{m-j-1},\quad\cdots,\quad p_1 = \Delta p_0.
\end{equation}
As such, we refer to $p_m$ as the \emph{harmonic coefficient} of $p$ and obtain that the other coefficients $p_{m-j}$ are polyharmonic: $\Delta^{j+1} p_{m-j} = 0$.  In fact, the relations \eqref{cond:soln} and the requirement that each coefficient $p_i(x)$ be homogeneous gives a characterization of $p(x,t)$ being an hcp. Thus, we arrive at the following elementary method of generating hcps: starting with any choice of $m\geq 0$ and hhp $p_m(x)$, use \eqref{cond:soln} to inductively solve for homogeneous coefficients $p_{m-1}(x), \dots, p_0(x)$; then $p(x,t)$ defined by \eqref{eqn:std_form} is an hcp.

Now, given any homogeneous polynomial $q(x)=bx^d$ with $x\in\R^1$, the polynomial $r(x)=\frac{1}{(d+2)(d+1)}cx^2q(x)$ is the unique homogeneous function such that $r''(x)=cq(x)$. Since the only hhps in $\R^1$ are of the form $q(x)=b$ or $q(x)=bx$, it follows that, up to scaling by a constant, there exists a unique hcp $p_d(x,t)$ in $\R^{1+1}$ for each degree $d\geq 0$. We adopt the following normalization for the hcps $p_d(x,t)$, emphasizing the time variable.

\begin{defn}\label{def:basic-p} For each $d\geq 0$, define $p_d:\R^{1+1}\rightarrow\R$ as follows. When $d=2k$  is even, \begin{equation*}
p_{d}(x,t) := t^k + \tfrac{k}{2!} t^{k-1} x^2 + \tfrac{k(k-1)}{4!} t^{k-2} x^4 + \cdots + \tfrac{k!}{(2k)!} x^{2k} = \sum_{j=0}^k \tfrac{k!}{(k-j)! (2j)!} t^{k-j} x^{2j}.
\end{equation*} When $d=2k+1$ is odd, \begin{equation*}
p_{d}(x,t) := t^k x + \tfrac{k}{3!} t^{k-1} x^3 + \tfrac{k(k-1)}{5!} t^{k-2} x^5 + \cdots + \tfrac{k!}{(2k+1)!} x^{2k+1} = \sum_{j=0}^k \tfrac{k!}{(k-j)! (2j+1)!} t^{k-j} x^{2j+1}.\end{equation*}
\end{defn}
{We shall refer to each $p_d(x,t)$ in the form above as a \textit{basic} hcp, since we shall show shortly that hcps in higher spatial dimensions can be represented through these one-dimensional hcps (see Lemma \ref{cor:orth}).}

\begin{defn}[see {\cite[\S5.5]{Szego75}}] The \emph{Hermite polynomials} $H_0(x)=1$, $H_1(x)=2x$, $H_2(x)=4x^2-2$, $H_3(x)=8x^3-12x$, $H_4(x)=16x^4-48x^2+12$, $H_5(x)=32x^5-160x^3+120x$, \emph{etc.}~are the family of orthogonal polynomials for the weighted space $L^2(\R,e^{-x^2}dx)$ defined by requiring that $\deg H_d=d$, the coefficient of $x^d$ in $H_d(x)$ is positive, and \begin{equation} \label{H-def} \int_{\R} H_{d}(x)H_{d'}(x)\,e^{-x^2}dx=\pi^{1/2}2^d d!\,\delta_{dd'}\quad\text{for all }d,d'\in\N.\end{equation}\end{defn}

\begin{lemma}[see {\cite[\S5.5]{Szego75}}] Equivalently, for all $d\geq 0$, \begin{equation}\label{H-expansion} H_d(x)=\sum_{j=0}^{\lfloor d/2\rfloor} \frac{d!}{j!(d-2j)!}(-1)^j(2x)^{d-2j}.\end{equation}\end{lemma}

After establishing the connection between the basic hcps and the Hermite polynomials (see e.g.~\cite{RW59}, \cite{Chen98}, \cite{KPS21}), one can use facts about the zeros of $H_d(x)$ and parabolic scaling to derive the following description of $p_d(x,t)$ and its nodal set.

\begin{prop}\label{lem:hcp_dim_1} For all $d\geq 2$, the basic hcp $p_d(x,t)$ assumes the form
\begin{equation}\label{pd-factor}
p_{d}(x,t) =\left\{\begin{array}{rl}
  (t+a_{d,1}x^2)\cdots (t+ a_{d,k}x^2) &\text{when $d=2k$ is even},\\
  x(t+a_{d,1}x^2)\cdots(t+a_{d,k}x^2) & \text{when $d=2k+1$ is odd},\end{array}\right.
\end{equation} for some distinct numbers $0<a_{d,1}<\dots<a_{d,k}$. Moreover, if we write \begin{align*}
p_{2k-1}(x,t)&=x(t+a_1x^2)\cdots(t+a_{k-1}x^2), &p_{2k+1}(x,t)&=x(t+c_1x^2)\cdots(t+c_kx^2),\end{align*}\begin{equation*}
p_{2k}(x,t)=(t+b_1x^2)\cdots(t+b_kx^2),\end{equation*}
 with the $a_i$\!'s, $b_i$\!'s, and $c_i$\!'s each listed in increasing order, then the coefficients associated with consecutive polynomials are interlaced: \begin{equation}\label{pd-interlace} \left\{\begin{array}{l}
 b_1<a_1<b_2<a_2<\cdots<a_{k-1}<b_k,\\
 \,c_1<b_1<c_2<b_2<\cdots<b_{k-1}<c_k<b_k.
 \end{array}\right.
 \end{equation}
\end{prop}
\begin{proof} Replacing $j$ by $k-j$ in the summations in Definition \ref{def:basic-p} yield that for all $d\geq 0$, \begin{equation}\label{basic-p-expansion} p_d(x,t)=\sum_{j=0}^{\lfloor d/2\rfloor} \frac{\lfloor d/2\rfloor!}{j!(d-2j)!}t^jx^{d-2j}.\end{equation} Comparing \eqref{H-expansion} and \eqref{basic-p-expansion}, we see that \begin{equation} \label{basic-b-vs-H} p_d(x,-1)=\frac{\lfloor d/2\rfloor!}{d!}H_d(x/2).\end{equation} Because the Hermite polynomial $H_d$ is even, when $d$ is even, $H_d$ is odd, when $d$ is odd, and orthogonal polynomials have a full number of distinct real roots (see \cite[Theorem 3.3.1]{Szego75}), we can factor $H_d(x)=(x^2-r_1^2)\cdots(x-r_k^2)$ for some $0<r_k<\cdots<r_1$, when $d=2k$ is even, and $H_d(x)=x(x^2-r_1^2)\cdots(x^2-r_k^2)$, when $d=2k+1$ is odd. Hence $$p_d(x,-1)=\frac{\lfloor d/2\rfloor!}{d!}r_1^2\cdots r_k^2x^{d-2\lfloor d/2\rfloor}\left(\frac{x^2}{4r_1^2}-1\right)\cdots\left(\frac{x^2}{4r_k^2}-1\right).$$ Together with parabolic homogeneity and the fact that $p_d(x,t)$ was normalized to have leading term $t^k$ or $t^kx$, this yields \eqref{pd-factor} with $a_{d,i}=1/(4r_i^2)$. Thus, \eqref{pd-interlace} follows from the interlacing of roots of consecutive Hermite polynomials \cite[Theorem 3.3.2]{Szego75}.\end{proof}

\begin{figure}\begin{center}\includegraphics[width=\textwidth]{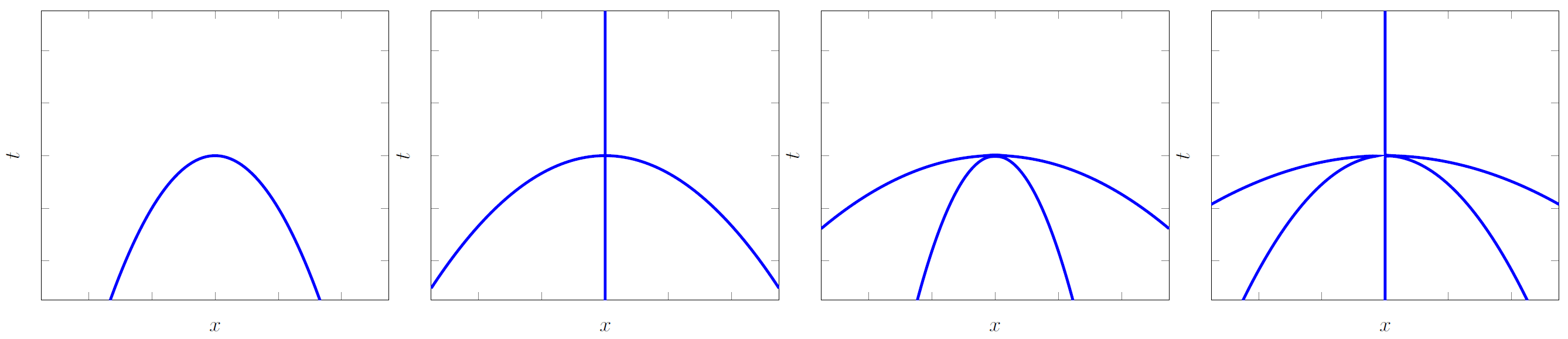}\end{center}
\caption{The nodal set of a degree $d$ hcp in $\R^{1+1}$ is a union of $\lfloor d/2 \rfloor$ nested, downward-opening parabolas with a common turning point at the origin, and when $d$ is odd, an additional vertical line (the $t$-axis). Thus, the number of nodal domains is precisely $2\lceil d/2\rceil$. From left to right, we illustrate the cases $d=2$, \dots, $d=5$. Inside the nodal set of $p_dp_{d+1}$, the ``nodal parabolas'' of consecutive hcps $p_d$ and $p_{d+1}$ are intertwined: the ``widest'' parabola of $p_{d+1}$ sits above the ``widest'' parabola of $p_d$; the ``widest'' parabola of $p_d$ sits above the ``second widest'' parabola of $p_{d+1}$; \emph{etc.}} \label{fig:hcp_1d}
\end{figure}

\begin{cor} \label{cor:dim_1}
Any hcp $p(x,t)$ in $\R^{1+1}$ of degree $d\geq 1$ has exactly $2\lceil d/2\rceil$ nodal domains. See Figure \ref{fig:hcp_1d}.
\end{cor}

In contrast to the case $n=1$, given a homogeneous polynomial $q(x)$ with $x\in\R^n$ for some $n\geq 2$, there is more than one way to produce a homogeneous polynomial $r(x)$ such that $\Delta r(x) = cq(x)$. We can build an explicit basis for the vector space $\HCP{d}(\R^{n+1})$ of all hcps of degree $d$ in $\R^{n+1}$ (and the zero function) using products of basic hcps in $\R^{1+1}$.

\begin{defn} Let $n\geq 2$. For each multi-index $\alpha=(\alpha_1,\dots,\alpha_n)$, define \begin{equation} p_\alpha(x,t):=p_{\alpha_1}(x_1,t)\cdots p_{\alpha_n}(x_n,t). \label{eqn:p_alpha}\end{equation}\end{defn}

{Recall that for a separable Hilbert space $H$, a countable set $\{e_i\}_{i \in I} \subset H$ is called an orthogonal basis for $H$ if its elements are pairwise orthgonal, and its span is dense in $H$.}

\begin{lemma}\label{cor:orth} For all $t<0$, the set $\{p_\alpha(\cdot,t):\alpha\text{ is a multi-index in }\N^n\}$ is an orthogonal basis for the weighted space $L^2(\R^n,e^{-|x|^2/4t}\,dx)$.
\end{lemma}

\begin{proof} Let $t<0$ and let $\alpha$ and $\beta$ be multi-indices in $\N^n$ with $\alpha\neq\beta$. By Fubini's theorem,
\begin{align*}
\int_{\R^n}  p_\alpha(x,t) p_\beta(x,t) e^{-\abs{x}^2/ 4\abs{t}} \, dx & = \prod_{i=1}^n \int_{\R} p_{\alpha_i}(x_i, t) p_{\beta_i}(x_i,t) e^{-x_i^2 / 4 \abs{t}} \, dx_i.
\end{align*} Thus, the left hand side vanishes if and only if at least one of the terms on the right hand side vanish. Since $\alpha\neq\beta$, there exists $1\leq i\leq n$ such that $\alpha_i\neq \beta_i$. Write $\gamma_i=(\alpha_i+\beta_i)/2$. By a simple change of variables, with $x_i=2|t|^{1/2}y$, parabolic homogeneity,  \eqref{basic-b-vs-H}, and \eqref{H-def}, \begin{equation*}\begin{split}&|t|^{-1/2}\int_{\R} p_{\alpha_i}(x_i, t) p_{\beta_i}(x_i,t) e^{-x_i^2 / 4 \abs{t}} \, dx_i= 2\int_\R p_{\alpha_i}(2|t|^{1/2}y,t) p_{\beta_i}(2|t|^{1/2}y,t)\,e^{-y^2}dy\\ &\quad =2|t|^{\gamma_i}\int_R p_{\alpha_i}(2y,-1) p_{\beta_i}(2y,-1)\,e^{-y^2}dy=2|t|^{\gamma_i}C(\alpha',\beta')\int_{\R}H_{\alpha_i}(y) H_{\beta_i}(y)\,e^{-y^2}dy=0.\end{split}\end{equation*} By a similar argument, $\{p_\alpha(\cdot,t):\alpha\text{ is a multi-index in }\N^n\}$ is an orthogonal basis for \\ $L^2(\R^n,e^{-|x|^2/4t}\,dx)$, because $\{H_d(x):d\geq 0\}$ is an orthogonal basis for $L^2(\R, e^{-x^2}dx)$.
\end{proof}

\begin{lemma}\label{lem:gen} For all $n\geq 1$ and $d\geq 0$, the set $\mathfrak{P}_{d}(\R^{n+1})=\{p_\alpha:|\alpha|=d\}$ is a basis for $\HCP{d}(\R^{n+1})$ and $\dim \HCP{d}(\R^{n+1})=\binom{n-1+d}{n-1}$. \end{lemma}

\begin{proof} For a proof that $\dim \HCP{d}(\R^{n+1})=\binom{n-1+d}{n-1}$, see \cite[Lemma 2.2]{CM21}. The fact that $\#\mathfrak{P}_d(\R^{n+1})=\binom{n-1+d}{n-1}$ is a standard exercise; see e.g.~\cite[p.~12]{Evans-PDE}. Now, the set $\mathfrak{P}_{d}(\R^{n+1})$ is linearly independent, because $\{p_{\alpha}(\cdot,-1):|\alpha|=d\}$ is linearly independent by Lemma \ref{cor:orth}. {Indeed, if we suppose that $\sum_{\abs{\alpha} = d} c_\alpha p_\alpha(x,t) = 0$ for some coefficients $c_\alpha$, then since the $p_\alpha$ take the form \eqref{eqn:p_alpha}, we see that $p_\alpha(x,-1) \not \equiv 0$ for each $\abs{\alpha} = d$. Lemma \ref{cor:orth} implies $c_\alpha =0$, and thus $\mathfrak{P}_d(\R^{n+1})$ is linearly independent.} Therefore, $\mathfrak{P}_d(\R^{n+1})$ is a basis for $\HCP{d}(\R^{n+1})$.
\end{proof}

{Since hhps are hcps, they can using the basis provided by Lemma \ref{lem:gen}. One may initially worry whether this is actually possible, because hhps are $t$-independent, whereas the non-constant hcps in the basis are $t$-dependent. The following example illustrates how an hhp can indeed be represented using $t$-dependent hcps.}
\begin{eg} The hhp $x^2-y^2$ in $\R^2$ can be expressed as a linear combination of the basic hcps $p_1(x,t)$ and $p_1(y,t)$: $x^2-y^2=2(t+\frac{1}{2}x^2)-2(t+\frac{1}{2}y^2)$.\end{eg}

We will also need the following fact in the next section, {which essentially says that the time-slices of hcps are eigenvalues of the operator $L = -\Delta - \nabla (\log(\phi)) \cdot \nabla  $, where $\phi(x) = e^{-\abs{x}^2/4}$. The lemma speaks to the relationship between hcps and Hermite polynomials given in \eqref{basic-b-vs-H}. We include the proof for completeness.}

\begin{lemma}\label{lem:eig} If $p$ is an hcp of degree $d$ in $\R^{n+1}$, then the function $v(x):= p(x,-1)$ satisfies \begin{align}\label{eqn:eig}
-\divv( e^{-\abs{x}^2/4} \nabla v(x)) = (d/2) e^{-\abs{x}^2/4} v(x).
\end{align} %Thus, for all $f\in C^1(\R^n)\cap L^2(\R^n,e^{-x^2/4}dx)$, \begin{equation}\label{eqn:parts} \int_{\R^n}\nabla f(x)\cdot \nabla v(x)\,e^{-x^2/4}dx = \frac{d}{2}\int_{\R^n} f(x)v(x)\,e^{-x^2/4}dx.\end{equation}
\end{lemma}

\begin{proof} Indeed, for $t < 0$, we have $p(x,t)  = (-t)^{d/2} p((-t)^{-1/2}x, -1) = (-t)^{d/2} v((-t)^{1/2}x)$.
Applying the heat operator $\partial_t - \Delta$ and simplifying yields
\begin{align*}
0 & = (-t)^{\frac{d}{2}-1}\left[-(d/2)v((-t)^{-1/2}x)+(1/2)\nabla v((-t)^{-1/2}x)\cdot (-t)^{-1/2}x - \Delta v((-t)^{-1/2}x)\right].
\end{align*}
Thus, $(1/2) \nabla v(x) \cdot x -\Delta v(x) = (d/2)v(x)$. Therefore, \begin{equation*}-\divv(e^{-|x|^2/4}\nabla v(x)) =-e^{-|x|^2/4}\left[(1/2)\nabla v\cdot x-\Delta v(x)\right]=(d/2)e^{-|x|^2/4}v(x).\qedhere\end{equation*}
\end{proof}

\section{HCP in high spatial dimensions}\label{sec:high_dim}

In this section, we  establish the case $n\geq 3$ of Theorem \ref{thm:main} and prove Theorem \ref{thm:max-asymptotics}.

\begin{prop}[Theorem \ref{thm:main} when $n\geq 3$] \label{prop:high_dim}  If $\phi(x,y)$ is an hhp of degree $d\geq 1$ in $\R^{2}$ and $\psi(z,t)$ is an hcp of degree $d$ in $\R^{1+1}$, then $u(x,y,z,t):=\phi(x,y)+\psi(z,t)$ is an hcp in $\R^{3+1}$ that has exactly two nodal domains. In particular, $m_{n,d}=2$ for all $n\geq 3$ and $d\geq 2$.
\end{prop}

\begin{proof} We modify the proof of \cite[Lemma 1.7]{BET17}, which gives a construction of hhps in $\R^n$ with two nodal domains when $n\geq 4$. The essential change is to show how to incorporate parabolic scaling. A trivial (but important!) observation is that for any $(x,y)$ and $(z,t)$, the expressions $|\phi(\lambda x,\lambda y)|=\lambda^d|\phi(x,y)|$ and $|\psi(\lambda z,\lambda^2t)|=\lambda^d|\psi(z,t)|$ are (weakly) increasing as functions of $\lambda\in[0,1]$. For any $(x,y)\neq (0,0)$, let ``the line segment from $(0,0)$ to $(x,y)$'' have its usual meaning. For any $(z,t)\neq (0,0)$, let ``the line segment from $(0,0)$ to $(z,t)$'' mean the curve described by $\{(\lambda z,\lambda^2 t): \lambda\in[0,1]\}$. Then $|\phi|$ and $|\psi|$ are increasing along line segments started at the origin and $|\phi|$ and $|\psi|$ are decreasing along line segments terminating at the origin. Also, the sets $\{\phi>0\}$, $\{\phi<0\}$, $\{\psi>0\}$, and $\{\psi<0\}$ are each nonempty.  Keeping these preliminaries in mind, we will now show that $\{u>0\}$ is path-connected. (Applying the same argument to $-u$ shows that $\{u<0\}$ is path-connected.)

Suppose that $(x_1,y_1,z_1,t_1)$ and $(x_2,y_2,z_2,t_2)$ are points at which $\phi(x_i,y_i)+\psi(z_i,t_i) >0$. Then for each $i=1,2$, at least one of the terms $\phi(x_i,y_i)$ and $\psi(z_i,t_i)$ is positive. Because the argument that follows only involves  ``line segments'' with an endpoint at the origin and the monotonicity of $|\phi|$ and $|\psi|$ along those line segments, we may suppose without loss of generality that $\psi(z_1,t_1)>0$ and $\psi(z_2,t_2)>0$. We emphasize that $\phi(x_1,y_1)$ and $\phi(x_2,y_2)$ may have any sign (positive, negative, zero). Choose any auxiliary point $(x_+,y_+)$ at which $\phi(x_+,y_+)>0$. We can build a ``piecewise linear'' path in $\{u>0\}$ from $(x_1,y_1,z_1,t_1)$ to $(x_2,y_2,z_2,t_2)$ as follows. The path is a concatenation of six ``line segments'' (see Figure \ref{fig:high_dim_paths}):%
\begin{figure}
 \begin{center}\includegraphics[width=.8\textwidth]{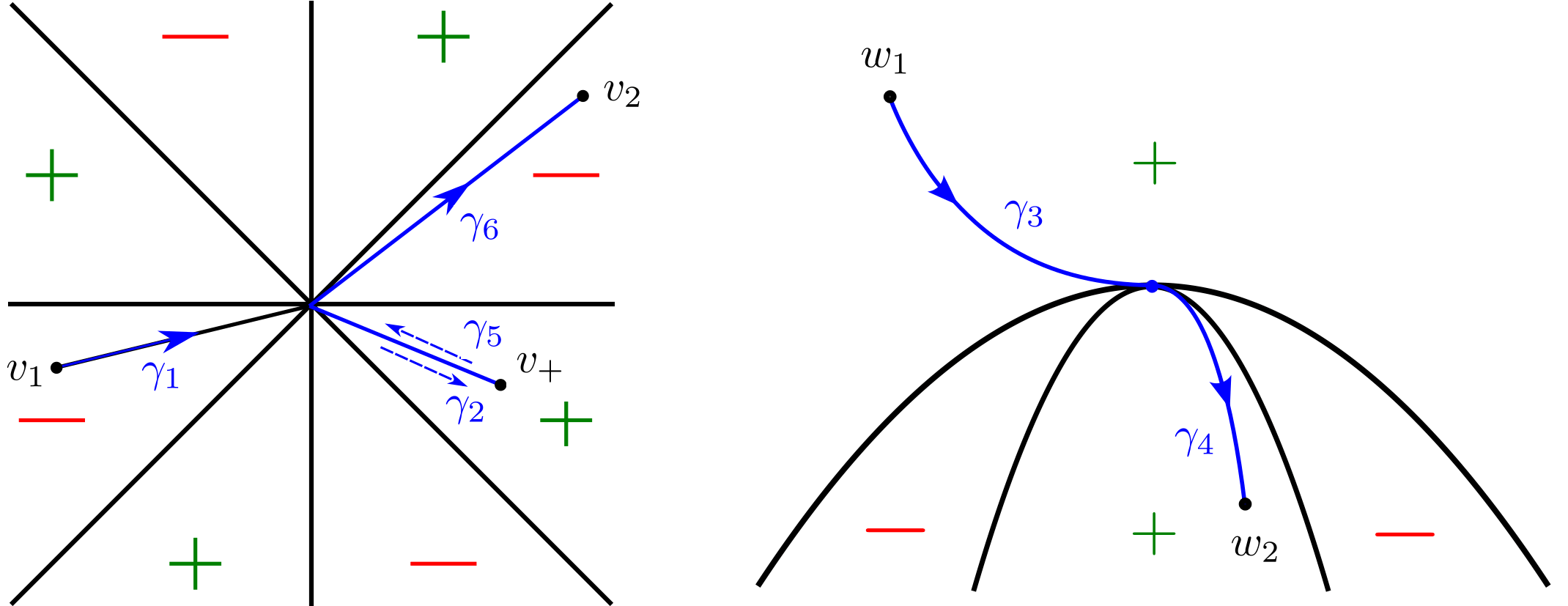}\end{center}
 \caption{Connecting $(v_1,w_1)=(x_1,y_1,z_1,t_1)$ to $(v_2,w_2)=(x_2,y_2,z_2,t_2)$ in $\R^{3+1}\cap\{u>0\}$, where $\phi(x,y)$ is a degree 4 hhp in $\R^{2}$ (left) and $\psi(z,t)$ is a degree 4 hcp in $\R^{1+1}$ (right).}\label{fig:high_dim_paths}
\end{figure}%
\begin{enumerate}
 \item First follow the line segment from $(x_1,y_1,z_1,t_1)$ to $(0,0,z_1,t_1)$.
 \item Second follow the line segment from $(0,0,z_1,t_1)$ to $(x_+,y_+,z_1,t_1)$.
 \item Third follow the line segment from $(x_+,y_+,z_1,z_2)$ to $(x_+,y_+,0,0)$.
 \item Fourth follow the line segment from $(x_+,y_+,0,0)$ to $(x_+,y_+,z_2,t_2)$.
 \item Fifth follow the line segment from $(x_+,y_+,z_2,t_2)$ to $(0,0,z_2,t_2)$.
 \item Sixth follow the line segment from $(0,0,z_2,t_2)$ to $(x_2,y_2,z_2,t_2)$.
\end{enumerate} It remains to confirm that the six line segments lie in $\{u>0\}$. The first segment lies in the positivity set, because $u(x_1,y_1,z_1,t_1)>0$ and $\psi(z_1,t_1)>0$. (If $\phi(x_1,y_1)\leq 0$, then $u$ increases from the initial point to the terminal point. If $\phi(x_1,y_1)\geq 0$, then $u$ decreases from the initial point to the terminal point, but $u$ never falls below $\psi(z_1,t_1)$.) Along the second segment, $\phi\geq 0$ and $\psi>0$, so $u>0$. Along the third segment, $\phi>0$ and $\psi\geq 0$, so $u>0$. After reversing the orientation, identical arguments show that the fourth, fifth, and sixth segments lie in $\{u>0\}$, as well.

{Altogether, we have thus shown that $m_{n,d} \le 2$ for $n \ge 3$ and $ d \ge 2$. Recalling Remark \ref{r:meanvalue}, we conclude that indeed $m_{n,d} = 2$ for such $n$ and $d$.}
\end{proof}

We now move on to the lower and upper bounds on $M_{n,d}$ for any spatial dimension.

\begin{prop}\label{prop:courant_hcp} For all $n\geq 2$ and $d\geq 2$, there exists a time-dependent hcp of degree $d$ in $\R^{n+1}$ with at least $\lfloor d/n\rfloor^n$ nodal domains.
\end{prop}
\begin{proof}
Let $c =  \lfloor d/n \rfloor$. If $c\leq 1$, then the conclusion is trivial. Thus, we may assume that $c\geq 2$. Let $p$ be an hcp in $\R^{1+1}$ of degree $c$ and let $q$ be an hcp in $\R^{1+1}$ of degree $b=d - (n-1)c \ge c$. A direct computation shows that $u(x,t):= q(x_n,t)\prod_{i=1}^{n-1} p(x_i, t)$ is an hcp of degree $d$. Moreover, $u$ is time-dependent, since $c\geq 2$ implies $p$ and $q$ are time-dependent. By Corollary \ref{cor:dim_1},
\begin{equation} \label{eqn:nprod}
\nodal(u) \ge \nodal(p)^{n-1}\nodal(q)\geq c^{n-1}b\geq c^n.
\end{equation}
We can verify the first inequality in \eqref{eqn:nprod} as follows. Suppose that $(x,t)$ and $(y,s)$ are points in $\R^{n+1}$ such that $u(x,t)\neq 0$, $u(y,s)\neq 0$, and \begin{enumerate}
\item[(a)] $(x_i,t)$ and $(y_i,s)$ belong to different nodal domains of $p$ for some $1\leq i\leq n-1$, or
\item[(b)] $(x_n,t)$ and $(y_n,s)$ belong to different nodal domains of $q$.
\end{enumerate} If (a) holds, assign $j=i$ and $r=p$; otherwise, if (b) holds, assign $j=n$ and $r=q$. Let $\gamma:[0,1]\rightarrow\R^{n+1}$ be any curve such that $\gamma(0)=(x,t)$ and $\gamma(1)=(y,s)$ and let $\pi$ be the orthogonal projection onto the $e_je_{n+1}$-plane. Then $\pi\circ \gamma:[0,1]\rightarrow\R^{1+1}$ is a curve connecting $(x_j,t)$ to $(y_j,s)$. Since $(x_j,t)$ and $(y_j,s)$ lie in different nodal domains of $r$, it follows that $r(\pi(\gamma(c)))=0$ for some $c\in(0,1)$. Thus, $u(\gamma(c))=0$, as well, by definition of $u$. Since $\gamma$ was an arbitrary curve connecting $(x,t)$ to $(y,s)$, we conclude that $(x,t)$ and $(y,s)$ lie in different nodal domains of $u$.

{We conclude as follows. Suppose that $v:\R^{n+1} \ra \R$ is a continuous function, and define $\mathfrak{N}(v)$ to be the set of nodal sets of the function $v$. Then we can identify each $V \in \mathfrak{N}(v)$ by an equivalence class of points in $V$ as follows. We say $(x,t) \sim_v (y,s)$ if there is a continuous path $\gamma:[0,1] \ra \R^{n+1}$ so that $\gamma(0) = (x,t)$, $\gamma(1) = (y,s)$, and $\abs{v(\gamma(p))} \ne 0$ for $p \in [0,1]$ (so, $\mathrm{Im}(\gamma)$ must be contained in exactly one nodal domain of $v$). For the functions $q$ and $p$, we have even more; each nodal domain of $p, q$ intersects the hyperplane $\{t = -1\}$. See Lemma \ref{negative-times} that follows, or for the simple one-dimensional case, refer to Figure \ref{fig:hcp_1d}. In particular each element of $\mathfrak{N}(p(x_i, -1))$ (resp. $\mathfrak{N}(q(x_n, t)))$ is identified by a point $(y_i, -1)$ (resp. $(y_n, -1)$). Then we define the mapping $I: \mathfrak{N}(q(x_n,t)) \times \prod_{i=1}^{n-1} \mathfrak{N}(p(x_i, t)) \ra \mathfrak{N}(u)$ by $((y_1, -1), \dotsc, (y_n, -1)) \ra (y_1, \dotsc, y_n, -1)$. The previous paragraph showed that this mapping $I$ is injective, and so we conclude the first inequality in \eqref{eqn:nprod}, and thus the proof of the proposition.}
\end{proof}

\begin{eg} Let $u(x,y,t)=(2t+x^2)(2t+y^2)$. Then $\nodal(u)=6>4=\nodal(2t+x^2)\nodal(2t+y^2)$. Thus, the first inequality in \eqref{eqn:nprod} can be strict. \end{eg}

The next lemma will help us bound the number of nodal domains of an hcp from above.

\begin{lemma}\label{negative-times} If $u$ is an hcp of degree $d\geq 1$ in $\R^{n+1}$, then every nodal domain of $u$ intersects the hyperplane $\{t=-1\}$. Thus, $\mathcal{N}(u)$ is at most the number of connected components of $\{x\in\R^n:u(x,-1)\neq 0\}$ in $\R^n$.\end{lemma}

\begin{proof} By parabolic scaling, it suffices to prove that every nodal domain of $u$ intersects the lower half-space $\{t<0\}$. Suppose to get a contradiction that there exists a nodal domain $U$ of $u$ such that $U\subset\{t\geq 0\}$. In fact, observe that $U\subset\{t>0\}$, since $U$ is open. Replacing $u$ by $-u$, if necessary, we may further assume that $U\subset\{u>0\}$. Write $$\mathbb{E}^n_+:=\{(x,t)\in\R^{n+1}:x_1^2+\cdots+x_n^2+|t|=1\text{ and }t\geq 0\}.$$ Define an auxiliary hcp of degree $2d$ in $\R^{n+1}$  by $f(x,t):=\sum_{i=1}^n p_{2d}(x_i,t)$. By Theorem \ref{lem:hcp_dim_1}, there exist constants $a_1,\dots,a_d>0$ such that $f(x,t)=\sum_{i=1}^n (t+a_1x_i^2)\cdots (t+a_dx_i^2).$ Hence $f(x,t)>0$ for all $(x,t)\in \mathbb{E}^n_+$. Thus, because $f$ is continuous and $\mathbb{E}^n_+$ is compact, we can find $c>0$ such that $f(x,t)>c$ for all $(x,t)\in \mathbb{E}^n_+$.

To proceed, choose $\epsilon>0$ small enough so that $v_\epsilon := u - \epsilon f$ is positive at some point in $U$. Let $V$ be any connected component of $U\cap\{v_\epsilon>0\}$. If $(x,t)\in V$ and $x_1^2+\cdots+x_n^2+|t|=\lambda^2$, then $(x/\lambda,t/\lambda^2)\in \mathbb{E}^n_+$ {(since $(x,t) \in V \subset U$ implies that $t > 0$)} and we see that
\begin{align*}
v_\epsilon(x,t) & = \lambda^d u(x/\lambda, t/\lambda^2) - \epsilon \lambda^{2d} f(x/\lambda, t/\lambda^2) \le \lambda^d(\norm{u}_{L^\infty(\mathbb{E}^n_+)} - \lambda^{d} \epsilon c) < 0
\end{align*} whenever $\lambda\gg 1$ (depending only on $d$, $c$, $\epsilon$, and $\|u\|_{L^\infty(\mathbb{E}^n_+)}$). This shows that $V$ is bounded. Furthermore, since $\partial V\subset \partial U\cup \{v_\epsilon=0\}\subset\{u=0\}\cup\{v_\epsilon=0\}$, we have $v_\epsilon\leq 0$ on $\partial V$. Therefore, $v_\epsilon\leq 0$ throughout $V$ by the maximum principle for solutions of the heat equation. This contradicts our assertion that $v_\epsilon$ is positive at some point of $V$. \end{proof}

\begin{eg} For all $k\geq 1$, the basic hcp $p_{2k}(x,t)$ of degree $2k$ in $\R^{1+1}$ has $\mathcal{N}(p_{2k})=2k$, but $\mathcal{N}(p_{2k}|\{t=-1\})=2k+1$. See Figure \ref{fig:hcp_1d}. \end{eg}

We now are ready to present an analogue of Courant's nodal domain theorem for the negative time-slices of hcps, which gives the upper bound on $M_{n,d}$. For the original version of Courant's theorem, see \cite[Chapter VI, \S6]{CH}; {for a version associated to the so-called Witten-Laplacian $L = -\Delta - \nabla(\log(\phi)) \cdot  \nabla$,  introduced right before Lemma \ref{lem:eig}, see \cite{CM25}. Unfortunately the result in the latter reference is only stated for bounded domains, so for completeness we find it easiest to provide a short proof in our specific case.}
% However the reader should note that with Lemma \ref{lem:eig} and \cite{CM25}, one morally obtains the following result.

%With the orthonormal basis of time slices of hcps as above, we may now prove our upper bound on $M_{n,d}$ as in Theorem \ref{thm:main}. Notice that for $d$ large, $\binom{n+d}{n} \le C_n d^n$, so that the upper bound on $M_{n,d}$ obtained in the following Theorem has the same order of growth (in $d$) as the lower bound obtained in Proposition \ref{prop:courant_hcp}.

\begin{thm}\label{thm:courant}
If $u$ is an hcp of degree $d$ in $\R^{n+1}$, then $\nodal(u) \le \binom{n+d}{n}$.
\end{thm}

\begin{proof} Let $H$ denote the Hilbert space $L^2(\R^n,e^{-x^2/4}dx)$. By Lemma \ref{cor:orth},  as $\alpha$ ranges over all multi-indices in $\N^n$, the polynomials $p_\alpha(x):=p_\alpha(x,-1)$ form an orthogonal basis for $H$. Assign $q(x):=u(x,-1)$. By Lemma \ref{negative-times}, $\nodal(u)\leq\nodal(q)$. Thus, to establish the theorem, it suffices to prove that $k:=\nodal(q)\leq \binom{n+d}{n}=\#\{p_\alpha(x):|\alpha|\leq d\}$. Enumerate the nodal domains of $q$ by $V_1, \dotsc, V_k \subset \R^n$. We suppose for the sake of contradiction that $k>\binom{n+d}{n}$. Consider the homogeneous system of $\binom{n+d}{n}$ linear equations $\sum_{i=1}^k c_{\alpha,i}a_i=0$, where $$c_{\alpha,i}:=\int_{\R^n} q(x)\chi_{V_i(x)} p_\alpha(x)\,e^{-x^2/4}dx\quad\text{for all $|\alpha|\leq d$ and $1\leq i\leq k$}.$$ Since $k>\binom{n+d}{n}$, we can find a non-zero solution vector $(a_1,\dots,a_k)$. By definition of the linear system, the function $g$ defined by $g(x):=\sum_{i=1}^k a_iq(x)\chi_{V_i}(x)$ is orthogonal to $p_\alpha$ in $H$ for all $|\alpha|\leq d$. On the one hand, by \eqref{eqn:eig}, {the fact that $q|_{\partial V_i} = 0$}, and integration by parts,
\begin{equation}\label{eqn:eigp}
\begin{split}
&\int_{\R^n} \abs{\nabla g(x)}^2 e^{-\abs{x}^2/4} \; dx  = \sum_{i=1}^k a_i^2 \int_{V_i} \abs{\nabla q(x)}^2 e^{-\abs{x}^2/4} \; dx\\
&\qquad= \sum_{i=1}^k a_i^2  \int_{V_i} \frac{d}{2} q(x)^2 e^{-\abs{x}^2/4} \; dx = \frac{d}{2} \int_{\R^n} g(x)^2 e^{-\abs{x}^2/4} \; dx.
\end{split}
\end{equation}
On the other hand, we can expand $g$ with respect to the orthogonal basis $\{p_\alpha\}$ for $H$, say $g(x) = \sum_{\abs{\alpha} > d} b_\alpha p_\alpha(x)$
for some coefficients $b_\alpha$. Now, by \eqref{eqn:eig} and integration by parts
\begin{align*}
\int_{\R^n} \nabla_x p_\alpha(x)\cdot \nabla_x p_\beta(x) e^{-\abs{x}^2/4} \, dx & = \int_{\R^n} \dfrac{\abs{\beta}}{2} p_\alpha(x) p_\beta(x) e^{-\abs{x}^2/4} \, dx = 0\quad\text{for all }\alpha\neq\beta.
\end{align*} Thus, using \eqref{eqn:eig} and integration by parts once more,
\begin{equation}\begin{split} \label{eqn:eigp2}
&\int_{\R^n} \abs{\nabla g(x)}^2 e^{-\abs{x}^2/4} \; dx  = \sum_{\abs{\alpha} > d} b_\alpha^2 \int_{\R^n}  \abs{\nabla_x p_\alpha(x)}^2 e^{-\abs{x}^2/4} \; dx \\
&\qquad= \sum_{\abs{\alpha} > d} b_\alpha^2 \int_{\R^n} \dfrac{\abs{\alpha}}{2} p_\alpha(x)^2 e^{-\abs{x}^2/4} \; dx \ge \dfrac{d+1}{2}  \int_{\R^n} g(x)^2 e^{-\abs{x}^2/4} \; dx.
\end{split}\end{equation} Since $g\not\equiv 0$, \eqref{eqn:eigp} and \eqref{eqn:eigp2} are incompatible. Therefore, $k \le \binom{n+d}{n}$. \end{proof}

Note that Proposition \ref{prop:courant_hcp} and Theorem \ref{thm:courant} yield Theorem \ref{thm:max-asymptotics}.

\section{HCP in \texorpdfstring{$\R^{2+1}$}{2+1 dimensions}, Part I: lower bounds}\label{sec:dim_3_low}

To complete the proof of Theorem \ref{thm:main}, it remains to determine the minimum possible number of nodal domains for hcps in $\R^{2+1}$, where the story is more complicated than in $\R^{1+1}$ (see Corollary \ref{cor:dim_1}) and in high enough spatial dimensions (see Proposition \ref{prop:high_dim}). In this section, we aim to show that $m_{2,4k}\geq 3$ for all $k\geq 1$, i.e.~the number of nodal domains of an hcp in $\R^{2+1}$ of degree $d\geq 4$ with $d \equiv 0 \pmod 4$ is at least 3.

Towards our goal, we first lower bound the number of nodal domains of a continuous function in $\R^2$ with ``alternating nodal structure'' at the origin and at infinity. A \emph{chamber of $u$ in $V$} is a connected component of $V \cap \{ u \ne 0\}$ relative to $V$. A \emph{positive chamber} is a chamber on which $u>0$; a \emph{negative chamber} is a chamber on which $u<0$.

\begin{lemma}\label{lem:cont_comp}
Suppose that $u : \R^2 \ra \R$ is continuous, $u(0) = 0$, and $u$ has the following nodal structure near the origin and near infinity for some integers $\nin,\nout\geq 1$ with $\nin+\nout\geq 3$:
\begin{itemize}
\item There exists $\epsilon>0$ (small) such that in $B_\epsilon(0)$, the chambers of $u$ are sectors based at the origin, $u$ alternates signs on adjacent chambers, and $u$ is positive on $\nin$ of the chambers. (When $\nin=1$, we allow either $0$ or $1$ negative chambers.)
\item There exists $M>\epsilon$ (large) such that in $B_M(0)^c$, the chambers of $u$ are sectors extending off to infinity, $u$ alternates signs on adjacent chambers, and $u$ is positive on $\nout$ of the chambers. (When $\nout=1$, we allow either $0$ or $1$ negative chambers.)
\end{itemize}
The number of nodal domains of $u$ is at least $\nin + \nout$.
\end{lemma}

\begin{figure}
\begin{center}\includegraphics[width=.8\textwidth]{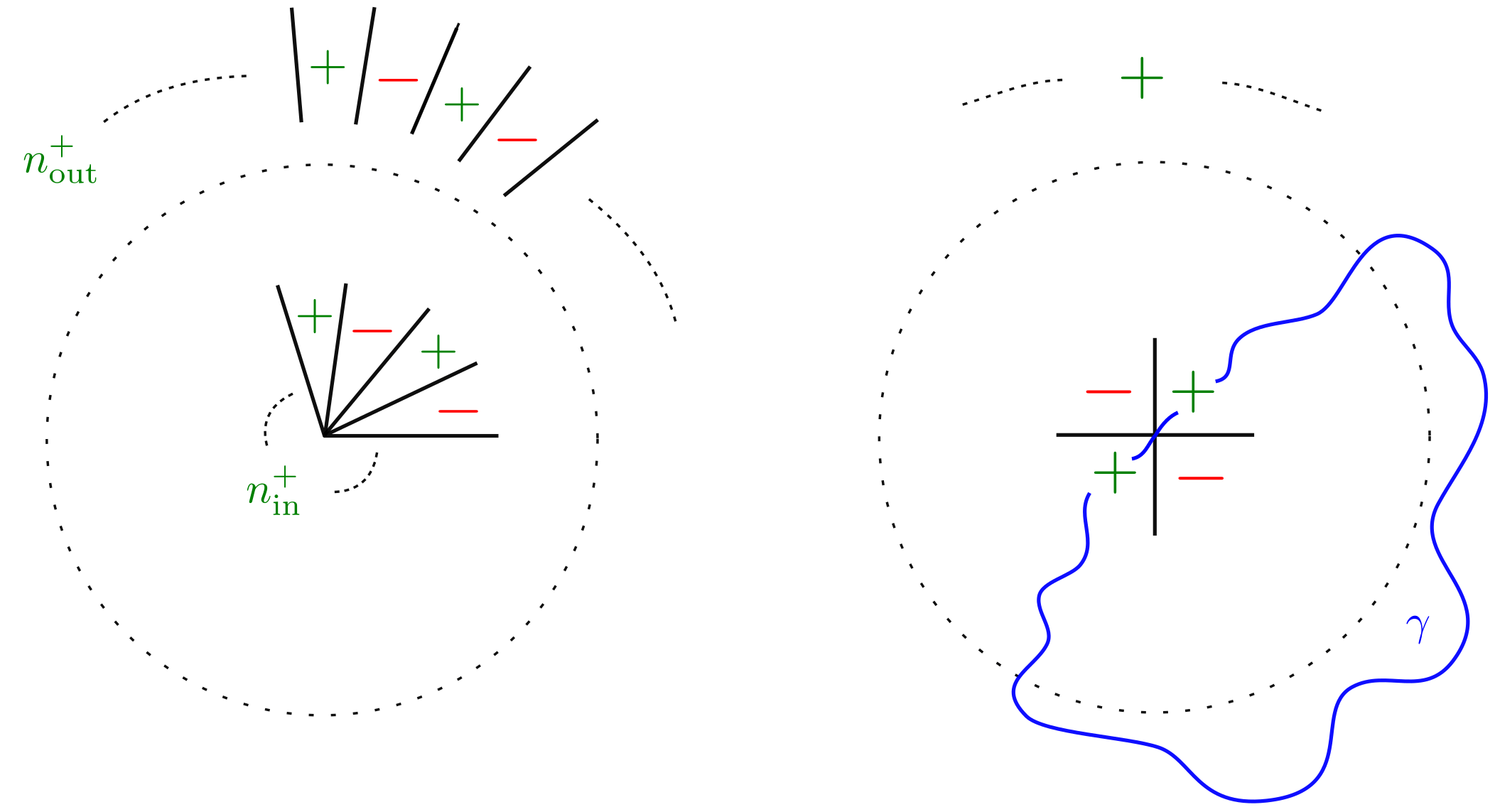}\end{center}
\caption{Parameters $\nin$ and $\nout$ count the number of positive chambers of $u$ near the origin and near infinity, respectively. On the right, we illustrate the base case $\nin = 2, \nout  =1$ (with zero negative chambers at infinity).}\label{fig:ind_b_2}
\end{figure}

\begin{proof} We adopt the phrase \emph{chamber at the origin} for chambers relative to $B_\epsilon(0)$ and the phrase \emph{chamber at infinity} for chambers relative to $\R^2\setminus B_M(0)$. The proof is by induction on $\nin + \nout$. For any fixed value of $\nin+\nout$, it suffices to establish either the case $\nin\geq\nout$ or the case $\nout\geq \nin$, as each case follows from the other case and inversion. By the alternating hypothesis, if $\nin\geq 2$, then $u$ also has $\nin$ negative chambers at the origin; if $\nout\geq 2$, then $u$ also has $\nout$ negative chambers at infinity.

For the base case, suppose that $\nin+\nout=3$, say $\nin=2$ and $\nout=1$. There are two alternatives. \emph{First alternative.} If the two positive chambers of $u$ at the origin belong to different nodal domains, then $u$ has at least three nodal domains, since $\{u>0\}$ has at least two connected components and $\{u<0\}$ is nonempty. \emph{Second alternative.} Suppose that the two positive chambers of $u$ at the origin belong to the same nodal domain. Then we can find a simple, closed curve $\gamma \subset \{ u >0\}\cup\{(0,0)\}$ that connects two points in distinct positive chambers of $B_\epsilon(0)$; see Figure \ref{fig:ind_b_2}. By the Jordan curve theorem, $\gamma$ disconnects the two negative chambers at the origin. Hence $u$ has at least three nodal domains, since $\{ u < 0\}$ has at least two connected components and $\{u>0\}$ is nonempty. The base case holds.

For the induction step, suppose that there exists $m\in\N$ such that the lemma holds whenever $3 \le \nin + \nout \le m$. Assume that $\nin+\nout=m+1$ and $\nin\geq \nout$. We remark that this implies $\nin\geq \lceil \frac{1}{2}(m+1)\rceil\geq 2$. We must prove that $u$ has at least $m+1$ nodal domains. There are two cases, one easy, one harder.

\emph{Easy case.} Suppose that no two chambers at the origin belong to the same nodal domain of $u$. Then $\mathcal{N}(u)\geq 2\nin\geq 2\lceil \frac{1}{2}(m+1)\rceil\geq m+1$ (since $\nin\geq 2$).

%Suppose that none of the positive chambers at infinity are in the same connected component of $ \{ u > 0\}$, and similarly none of the positive chambers of $u$ at the origin are in the same connected component of $\{u > 0\}$. Upon considering $-u$, we also assume that none of the negative chambers of $u$ at infinity and none of the negative chambers of $u$ at the origin are in the same connected component of $\{u < 0\}$. In this case, we clearly have that the number of nodal domains of $u$ is at least $2 \max\{ \nin, \nout \} \ge \nin + \nout$. This proves the induction step in this case.

\emph{Harder case.} Suppose that two distinct chambers $C$ and $D$ at the origin belong to the same nodal domain. Among all candidates, we choose $C$ and $D$ that have \begin{equation}\label{chamber-minimality} \text{the least number of (positive and negative) chambers between them}.\end{equation} Replacing $u$ by $-u$, if needed, we may assume without loss of generality that $C$ and $D$ are positive chambers. By the alternating condition, the minimum number of chambers strictly between $C$ and $D$ is at least one and is odd. Suppose there are $2k+1$ such chambers and enumerate them and $C$ and $D$ in order:  \begin{equation}\label{chamber-list} C,\ N_1,\ P_1, \dots,\ P_k,\ N_{k+1},\ D,\end{equation} where $N_1,\dots, N_{k+1}$ are negative and $P_{1},\dots, P_k$ are positive. (When $k=0$, the enumeration is $C, N_1, D$.) Pause and note that the $2k+2$ chambers $C$, $N_1$, $P_1$, \dots, $P_k$, $N_{k+1}$ lie in $2k+2$ disjoint nodal domains of $u$, otherwise we would violate \eqref{chamber-minimality}. Also note that $k\leq \frac{1}{2}\nin-1$, when $\nin$ is even, and $k\leq \frac{1}{2}\nin-\frac{3}{2}$ when $\nin$ is odd. (To see this, draw some examples for small values of $\nin$.) Either way, $\nin\geq 2k+2$.

To proceed, let $\sigma \subset \{u >0\}\cup\{(0,0)\}$ be a simple closed curve that connects a point in $C$ to a point in $D$, passes through the origin in $\overline{C}$ and $\overline{D}$, and encloses the intermediate chambers $N_1,P_1,\cdots,P_k,N_{k+1}$ in the bounded component $\Omega$ of $\R^2\setminus\sigma$. Next, let $\tilde u:\R^2\rightarrow\R$ be any continuous function such that $\tilde u|_{\R^2\setminus\Omega}=u$ and $\tilde u>0$ throughout $\Omega$. Effectively, this collapses the chambers in \eqref{chamber-list} into a single positive chamber. Since there are $k+2$ positive chambers in \eqref{chamber-list}, it follows that $\nin(\tilde u)=\nin(u)-(k+1)$ and $\nout(\tilde u)=\nout(u)$. Suppose to get a contradiction that $\nin(\tilde u)+\nout(\tilde u)\leq 2$. Then $$k+2=(2k+2)-(k+1)+1\leq \nin(u)-(k+1)+\nout(u)
=\nin(\tilde u)+\nout(\tilde u)\leq 2.$$ Hence $k=0$ and $4\leq m+1= \nin(u)+\nout(u)=\nin(\tilde u)+\nout(\tilde u)+1\leq 3$, which is absurd. Therefore, $3\leq \nin(\tilde u)+\nout(\tilde u)\leq \nin(u)+\nout(u)-1\leq m$.

By the induction hypothesis, $\nodal(\tilde u)\geq \nin(\tilde u)+\nout(\tilde u)=\nin(u)+\nout(u)-(k+1)$. Now, the nodal domains of $\tilde u$ are in one-to-one correspondence with the nodal domains of $u$ that intersect $\R^n\setminus\Omega$. Together with the additional $2k+1$ nodal domains of $u$ inside $\Omega$, corresponding to the chambers $N_1$, $P_1$, \dots, $P_k$, $N_{k+1}$, we conclude that in total  $\nodal(u)\geq \nin(u)+\nout(u)+k\geq \nin(u)+\nout(u)=m+1$. This completes the induction step.
\end{proof}

\begin{lemma}\label{l:alternate} If $u=u_d+f$ near the origin in $\R^2$, where $u_d$ is an hhp in $\R^2$ of degree $d\geq 1$, $f\in C^1$, and in polar coordinates, $|f(r,\theta)| + \abs{\partial_\theta f(r,\theta)} =o(r^d)$ as $r \ra 0$, then $u$ has $2d$ chambers with alternating signs in all sufficiently small neighborhoods of the origin and $\nin(u)=d$.\end{lemma}

\begin{proof}
 Up to a rotation and renormalization, $u_d$ is given in polar coordinates by $u_d(r,\theta) = r^d \sin (d\theta)$. Thus the nodal domains of $u_d$ consist of $2d$ sectors at the origin with opening $\pi/d$ with alternating signs. If we consider the function $g_r(\theta): \sphere^1 \ra \R$ defined by $g_r(\theta) = r^{-d} u(r,\theta)$, then zero set of $g_r$ is the same as that of $u$ intersected with $\partial B_r(0)$. Moreover, the assumptions on $f$ near the origin give us that
\begin{equation}\label{eqn:gr}
\begin{split}
g_r(\theta) = \sin(d\theta) + e_1(r,\theta), \\
\partial_\theta g_r(\theta) = d \cos(d\theta) + e_2(r,\theta),
\end{split}
\end{equation}
where $\abs{e_i (r,\theta)} \ra 0$ as $r \ra 0$ for $i=1,2$. Since $\abs{\cos(d\theta)} =1$ on whenever $\sin(d\theta)=0$, it is straight-forward to check from \eqref{eqn:gr} that the following holds. For each $\epsilon >0$, there is some $r_0 >0$ small so that for all $0 < r < r_0$, $g_r(\theta)$ has exactly $2d$ zeros in $\sphere^1$ (one in each of the sectors $\{\abs{\theta - j\pi/d} < \epsilon\}$ for $0 \le j \le 2d-1$), at which $g_r$ changes sign. If we order such zeros $\theta_0(r),  \theta_1(r) , \dotsc, \theta_{2d-1}(r)$, then the $\theta_j(r)$ are continuous in $r$ since $u$ is a continuous function, and $\theta_j(r) \ra j\pi/d$ as $r \da 0$. Recalling that the nodal set of $g_r$ coincides with $\nodal(u) \cap \partial B_r(0)$, one obtains the desired conclusion.
\end{proof}

As an application of Lemmas \ref{lem:cont_comp} and \ref{l:alternate}, we obtain a lower bound on the number of nodal domains of an hcp in $\R^{2+1}$ whose harmonic coefficient has degree at least 2. {Recall that the harmonic coefficient of an hcp $p$ is the leading spatial coefficient $p_m(x)$ when $p(x,t)$ is written as in \eqref{eqn:std_form}. Thus, the following proposition allows us to reduce to lower degree harmonic coefficients in the proof of $m_{2,d} = 3$ when $d \equiv 0 \pmod{4}$, which will be taken care of next in Proposition \ref{prop:low_bd_2}.}

\begin{prop}\label{prop:low_bd}
If $p(x,y,t)$ is an hcp in $\R^{2+1}$ of the form \eqref{eqn:std_form} and $\deg p_m \ge 2$, then $p$ has at least {$ 2 \deg p_m \ge 4$} nodal domains.
\end{prop}
\begin{proof} Put $d=\deg p_m$. By Remark \ref{r:sphere}, it suffices to prove that the nodal set of $p|_{\sphere^2}$ has at least $2d$ nodal domains. Since $d\geq 1$, $p(N)=p(S)=0$, where $N=(0,0,1)$ and $S=(0,0,-1)$ are the north and south poles of $\sphere^2$, respectively. Let $\pi_N$ denote the stereographic projection from the north pole of $\sphere^2$ onto $\widehat{\R}^2=\R^2\cup\{\infty\}$. Then $\pi_N(S)=0$ and $\pi_N(N)=\infty$. Thus, if we can show that the nodal domains of $u:=p\circ \pi_N^{-1}$ near the origin and near infinity are homeomorphic to $2d$ sectors with alternating signs, then $\nin(u)=\nout(u)=d$ and $\nodal(p)=\nodal(p|_{\sphere^2})=\nodal(u)\geq \nin(u)+\nout(u)=2d\geq 4$ by Lemma \ref{lem:cont_comp}.

The stereographic projection $\pi_N: \sphere^2 \setminus\{N\} \ra \R^2$ is given by
\begin{align*}
\pi_N(x,y,t) & = \left( \dfrac{x}{1-t}, \dfrac{y}{1-t}  \right), \quad \pi_N^{-1}(X,Y)  = \left( \dfrac{2X}{R^2+1}, \dfrac{2Y}{R^2 +1}, \dfrac{R^2-1}{R^2+1} \right),
\end{align*} where $R=R(X,Y)=(X^2+Y^2)^{1/2}$. Recalling \eqref{eqn:std_form} {for the definition of $p_j$}, we have
\begin{equation}\begin{split}
u(X,Y)  &= \sum_{j=0}^m \left( \dfrac{R^2-1}{R^2+1}  \right)^{j} p_j\left( \dfrac{2X}{R^2+1}, \dfrac{2Y}{R^2+1} \right)\\
&= \sum_{j=0}^m \left( \dfrac{R^2-1}{R^2+1}  \right)^{j} \left(\dfrac{2}{R^2+1}\right)^{d+2(m-j)} p_j(X,Y) = \sum_{j=0}^m a_j(R)p_j(X,Y),
\end{split}\end{equation} 
{where we assign}
\begin{align*}
	{a_j(R) \coloneqq \left( \dfrac{R^2-1}{R^2+1}  \right)^{j} \left(\dfrac{2}{R^2+1}\right)^{d+2(m-j)}.}
\end{align*}
Notice that each term $p_{j}$ is algebraically homogeneous of order $d + 2(m-j)$, the lowest order term $p_m$ is an hhp of degree $d$ (see e.g.~Figure \ref{fig:high_dim_paths} for the case $d=4$), and the coefficients $a_j(R)$ are radial, bounded real-analytic functions of $X$ and $Y$ with $\lim_{R\rightarrow 0} a_j(R) = (-1)^j2^{d+2(m-j)}$. {Note that} $a_m$ does not vanish in $\{0<R\leq 1/2\}$. It follows that $f:=u/a_m-p_m$ is real-analytic near $0$ and in polar coordinates, $|f(R, \theta)| + \abs{\partial_\theta f(R, \theta)} = o(R^d)$ as $R \ra 0^+$, since each of the terms $p_0$, \dots, $p_{m-1}$ is homogeneous with degree at least $d+2$. By Lemma \ref{l:alternate}, we conclude that $u/a_m=p_m+f$ has $2d$ chambers at the origin with alternating signs. Thus, since $a_m$ does not vanish near the origin, $u=a_m(u/a_m)$ also has $2d$ chambers at the origin with alternating signs and $\nin(u)=\nin(u/a_m)=d$.

Finally, let $v:=p\circ \pi_S^{-1}$, where $\pi_S$ denotes the stereographic projection from the south pole of $\sphere^2$ onto $\widehat{\R}^2$. Repeating the argument for $u$ shows that $\nin(v)=d$. Therefore, because $u$ and $v$ are related by inversion, we have $\nout(u)=\nin(v)=d$.\end{proof}

Next, we present a lower bound on the number of nodal domains for a certain class of hcps of degree $4k$, {which is the last class we need to consider to compute $m_{2,d}$ for $d \equiv 0 \pmod{4}$}.%parabolic degree congruent to $0 \pmod 4$ that is not covered by Proposition \ref{prop:low_bd}
\begin{prop} \label{prop:low_bd_2}
If $p(x,y,t)$ is an an hcp in $\R^{2+1}$ of the form
\begin{align}
p(x,y,t) & = t^{2k} + t^{2k-1} p_{2k-1}(x,y) + \cdots + p_0(x,y) \label{eqn:p_rep}
\end{align}
for some $k\geq 1$, then $p$ has at least $3$ nodal domains.
\end{prop}

\begin{figure}
\begin{center}\includegraphics[width=.35\textwidth]{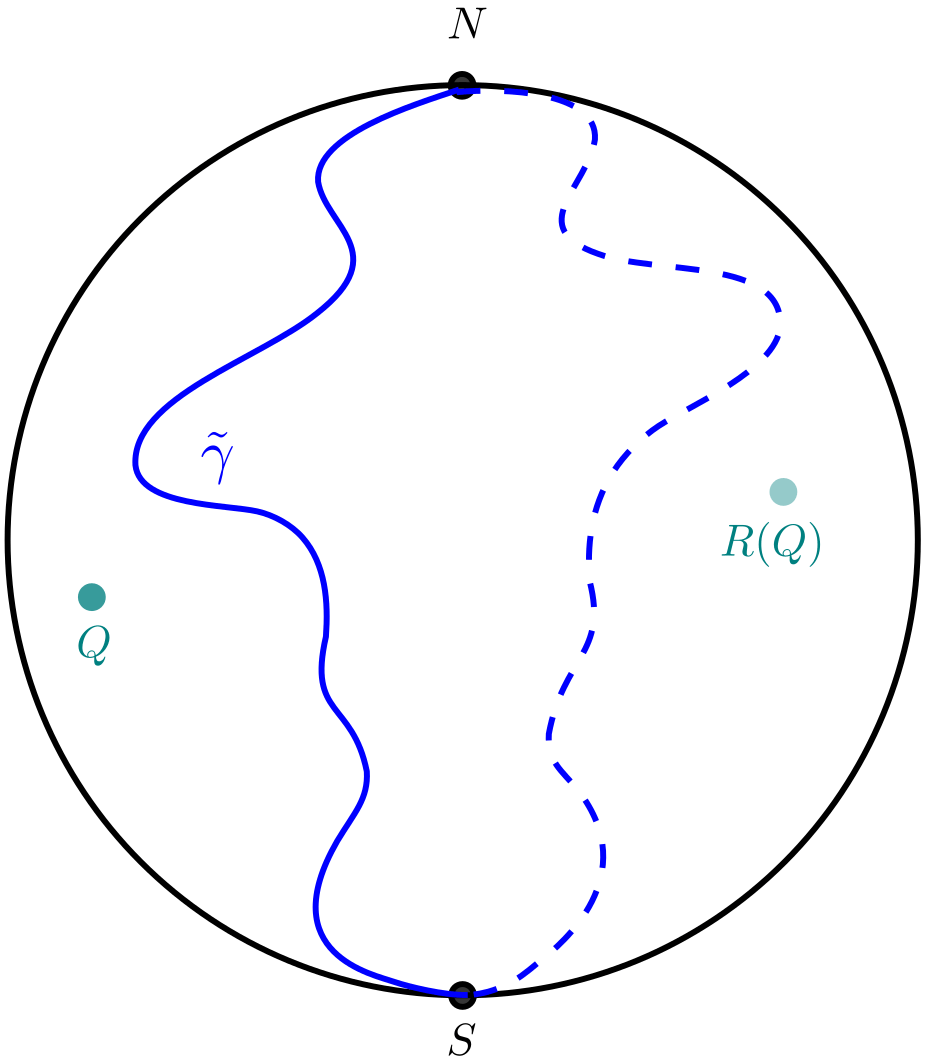}\end{center}%
\caption{Proof of Proposition \ref{prop:low_bd_2}. If the union $\tilde\gamma$ of the path $\gamma$ and its reflection $R\circ \gamma$ lie in the positivity set of $p|_{\sphere^2}$, then the negativity set of $p|_{\sphere^2}$ is\\ disconnected.}\label{fig:t_even_curve}
\end{figure}

\begin{proof} By Remark \ref{r:sphere}, it suffices to prove that the positivity set of $p|_{\sphere^2}$ is disconnected or the negativity set of $p|_{\sphere^2}$ is disconnected. Thus, if $\{p|_{\sphere^2}>0\}$ is disconnected, we are done.

Suppose that the positivity set of $p|_{\sphere^2}$ is connected. By \eqref{eqn:p_rep}, $p(0,0,1)=p(0,0,-1)=1$. Thus, the north pole $N=(0,0,1)$ and south pole $S=(0,0,-1)$ belong to the positivity set of $p|_{\sphere^2}$. Hence there exists a simple path $\gamma :[0,1] \ra \{p|_{\sphere^2} > 0\}$ such that $\gamma(0) = N$ and $\gamma(1) = S$. Now, when $t$ is fixed, $p(x,y,t)$ is an even polynomial in $x$ and $y$, since each coefficient $p_i(x,y)$ in \eqref{eqn:p_rep} has even degree. In particular, writing $R(x,y,t) = (-x,-y,t)$, we have $R\circ \gamma$ is a mirrored path from $N$ to $S$ contained in $\{p|_{\sphere^2} > 0\}$. Let $\tilde\gamma$ be the union of the traces of $\gamma$ and its reflection $R\circ \gamma$. See Figure \ref{fig:t_even_curve}. By the Jordan curve theorem, $\tilde\gamma$ separates $\sphere^2\setminus\tilde\gamma$ into two connected components. Let $Q\in\sphere^2$ be any point such that $p(Q)<0$. Then $p(R(Q))=p(Q)<0$, but $Q$ and $R(Q)$ lie in different connected components of $\sphere^2\setminus\tilde\gamma$. Therefore, the negativity set of $p|_{\sphere^2}$ is disconnected. % This completes the proof, since then $\sphere^2 \setminus \{ p < 0\}$ has at least two connected components in this case (when $\{p > 0 \} \cap \sphere^2$ is connected), and in the other case, we are done trivially.
\end{proof}

Combining the previous two results, we see that $m_{2,d} \ge 3$ when $d \equiv 0 \pmod 4$.
\begin{cor}[A lower bound for $d \equiv 0 \pmod 4$] \label{cor:0_mod_4}
If $d \equiv 0 \pmod 4$ and $d\geq 4$, then any time-dependent hcp $p$ of degree $d$ in $\R^{2+1}$ has at least $3$ nodal domains. Hence $m_{2,d}\geq 3$.
\end{cor}
\begin{proof}
Since $p$ has degree $d=4k\geq 4$ ({which is even}), its leading $t$ term is either $t^{2k}$ or of the form $t^m p_m(x,y)$ with $\deg p_m \ge 2$. In the first case, $\mathcal{N}(p)\geq 3$ by Proposition \ref{prop:low_bd_2}. In the second  case, $\mathcal{N}(p)\geq 2\deg p_m\geq 4$ by Proposition \ref{prop:low_bd}.
\end{proof}

\section{HCP in \texorpdfstring{$\R^{2+1}$}{2+1 dimensions}, Part II: constructions} \label{sec:dim_3_constr}

In this section, we construct examples of time-dependent hcps in $\R^{2+1}$ of degree $d\geq 2$ with two nodal domains when $d \not \equiv 0 \pmod 4$ and with three nodal domains when $d \equiv 0 \pmod 4$. By Remark \ref{r:sphere}, counting the nodal domains of an hcp $p$ is equivalent to counting the nodal domains of $p|_{\sphere^2}$. With this reduction, the general strategy in constructing examples is the same in all cases (and is the strategy introduced by \cite{Lewy77} and used by \cite{EJN07}, \cite{LTY15}, and \cite{BH16} in related contexts):
\begin{enumerate}[(a)]
\item Begin with an hcp $\phi_1$ of degree $d$ whose nodal set can be described explicitly.
\item Find another hcp $\phi_2$ of degree $d$ so that the nodal set of the perturbation $u=\phi_1 - \epsilon \phi_2$ in $\sphere^2$ is either one Jordan curve ($u$ has two nodal domains) or the nodal set of $u$ is the union of two disjoint Jordan curves ($u$ has three nodal domains).
\end{enumerate}
The key difficulty in this strategy is finding certain compatibility conditions between $\phi_1, \phi_2$. In general, the nodal set $\{ \phi_1|_{\sphere^2} =0 \}$ is the union of a relatively open smooth set,  where $|\nabla \phi_1|\neq 0$, and isolated singular points, where $|\nabla \phi_1|=0$. Understanding the picture near smooth portion of the nodal set is straightforward; see e.g.~\cite[Lemma 2]{Lewy77}. However, understanding how the nodal domains of $\phi_1$ change under perturbation near a singular point is quite delicate, requiring knowledge both of the local structure of $\phi_1$ and of the sign of $\phi_2$. This makes finding $\phi_1$ and $\phi_2$ challenging. After reviewing the main perturbation lemmas, we present our examples in order of increasing difficulty: $d\equiv 2\pmod 4$ in \S\ref{subsec:2mod4}, $d$ odd in \S\ref{subsec:1mod4}, and $d\equiv 0\pmod 4$ in \S\ref{subsec:0mod4}.

\begin{figure}\begin{center}\includegraphics[width=.7\textwidth]{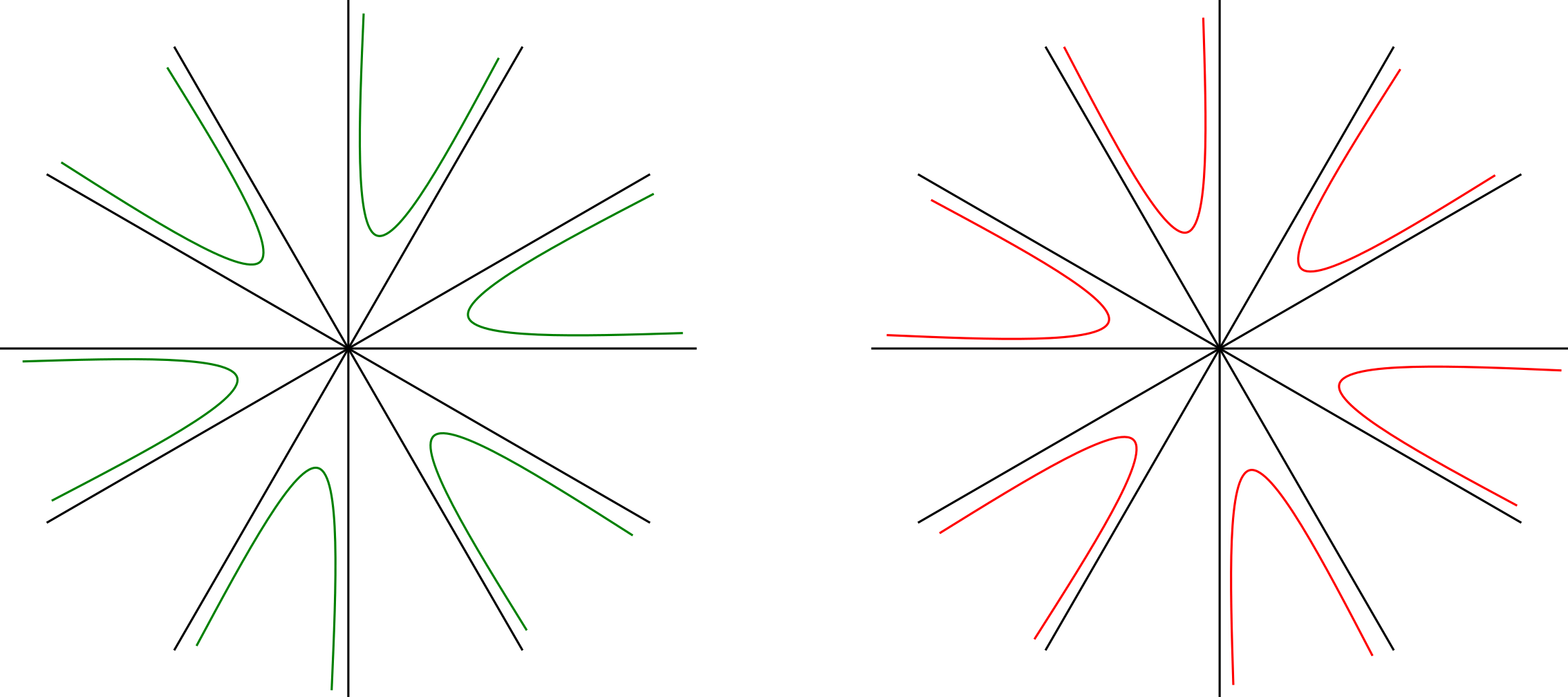}\end{center}\caption{Zero set of $\psi(x,y)=\Im((x+iy)^6)$ (in black) and its perturbation $\psi(x,y)-\epsilon f(x,y)$ near the origin when $f$ is $C^1$, $\epsilon$ is small, and $f(0,0)>0$ (left, in green) or $f(0,0)<0$ (right, in red). The green lines lie in the positivity set for $\psi$, whereas the red lines lie in the negativity set for $\psi$.}\label{fig:hhp6}\end{figure}

The first of two perturbation lemmas that we need, Lemma \ref{l:lewy}, is the consequence of \cite[Lemma 4]{Lewy77} recorded on the second paragraph on p.~1239 of Lewy's paper. In the original source, it is stated that $f$ should be real-analytic, but inspecting the proof shows that it suffices to assume $f$ is $C^1$. {Before stating our results, we remind the reader of the definition of the \textit{Hausdorff distance} between two non-empty sets $E, F \subset \R^n$
\begin{align}\label{eqn:hd_dist}
	\mathrm{HD}(E,F) \coloneqq \max \left \{ \sup_{x \in E} \dist(x, F) , \,   \sup_{x \in F} \dist(x, E) \right \}.
\end{align}} 

\begin{lemma}\label{l:lewy} Let $\psi(x,y)=\mathrm{Im}((x+iy)^d)$ for some $d\geq 2$. If $f:B_r(0)\rightarrow\R$ is $C^1$ and $f(0,0)>0$, then there exists $\tau\in(0,r)$ and $\epsilon_0>0$ such that for all $\epsilon\in(0,\epsilon_0)$, the nodal set of $\psi-\epsilon f$ in $B_\tau(0)$ consists of $d$ pairwise disjoint simple curves, with one curve inside each of the connected components of $\{\psi>0\}$, and $$\lim_{\epsilon\rightarrow 0} \mathrm{HD}\left(\{\psi- \epsilon f=0\}\cap B_\tau(0), \{\psi=0\}\cap B_\tau(0)\right)=0.$$ The same conclusion holds when $f(0,0)<0$ except that then the nodal set of the perturbation $\psi-\epsilon f$ lies in $\{\psi<0\}$. See Figure \ref{fig:hhp6}.
\end{lemma}

We also need a variant of Lemma \ref{l:lewy}, in which $\psi(x,y)$ is replaced by a function $G(x,y)$ whose nodal set is given locally by the union of $m$ graphs with simple intersection at a common point. We emphasize that the following lemma is inspired by \cite{Lewy77}.

\begin{lemma}\label{l:graph}
Suppose that $G:B_r(0) \subset \R^2 \ra \R$ takes the form
$G(x,y)  = \prod_{i=1}^m g_i(x,y)$ for some $m\geq 2$,
where $g_1,\dots,g_m:B_r(0) \ra \R$ are real-analytic functions satisfying \begin{itemize}
\item $g_i(0,0) = 0$ and $\partial_y g_i(0,0)\neq 0$ for all $i$,
\item $\{g_i =0 \} \cap \{g_j = 0\} = \{(0,0)\}$ for all $i\neq j$.\end{itemize}
If $F: B_r(0) \ra \R$ is $C^1$ and $F(0,0) > 0$, then there exists $\tau \in (0,r)$ and $\epsilon_0 > 0$ such that for all $\epsilon \in (0, \epsilon_0)$, the nodal set of $G - \epsilon F$ in $B_\tau(0)$ consists of $m$ pairwise disjoint simple curves, one inside each of the $m$ connected components of $\{ G > 0\}$, and $$\lim_{\epsilon\rightarrow 0} \mathrm{HD}\left(\{G- \epsilon F=0\}\cap B_\tau(0), \{G=0\}\cap B_\tau(0)\right)=0,$$
The same conclusion holds when $F(0,0)<0$ except that then the nodal set of the perturbation $G-\epsilon F$ lies in $\{G<0\}$. See Figure \ref{fig:graph-scheme}.
\end{lemma}

\begin{figure}\begin{center}%
\includegraphics[width=.31\textwidth]{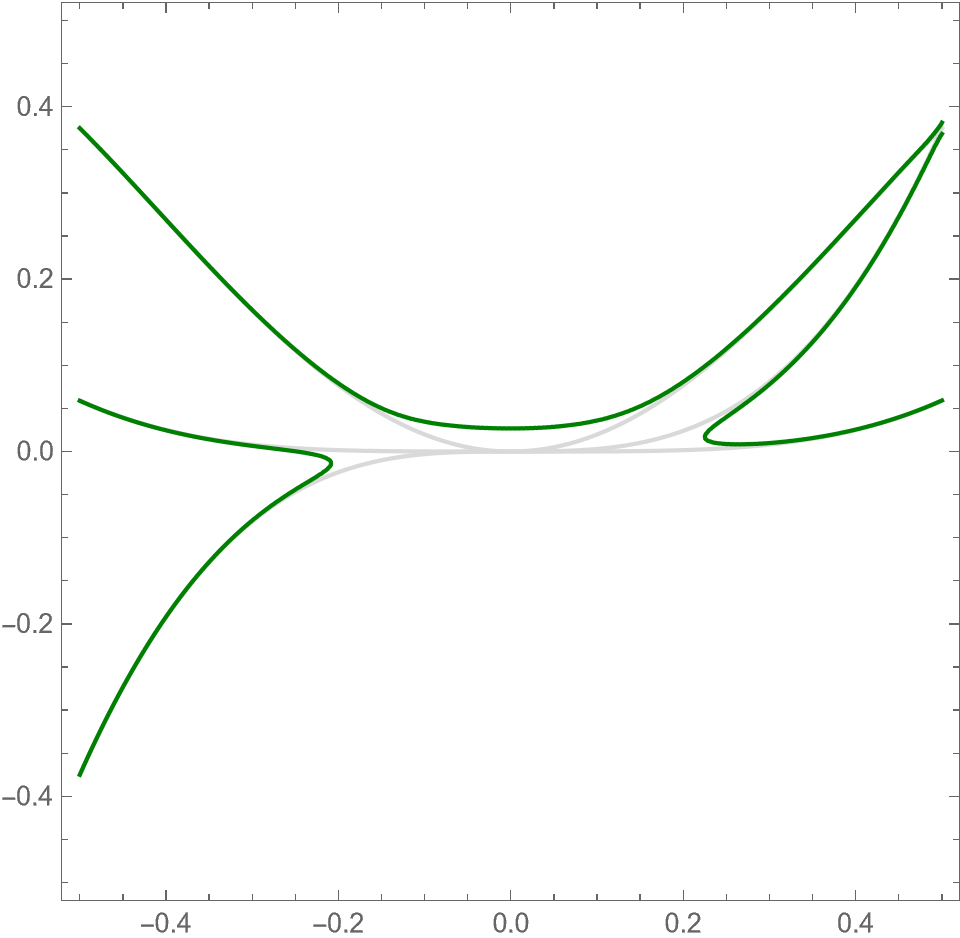}\hfill%
\includegraphics[width=.31\textwidth]{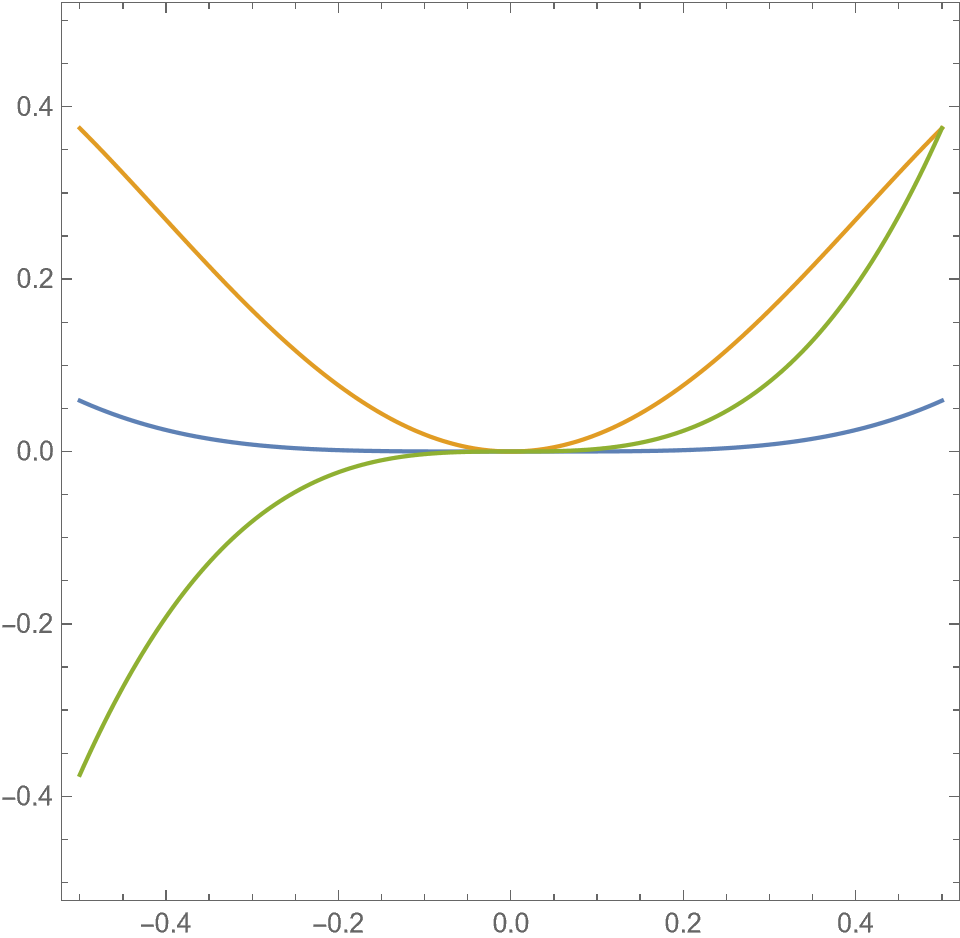}\hfill%
\includegraphics[width=.31\textwidth]{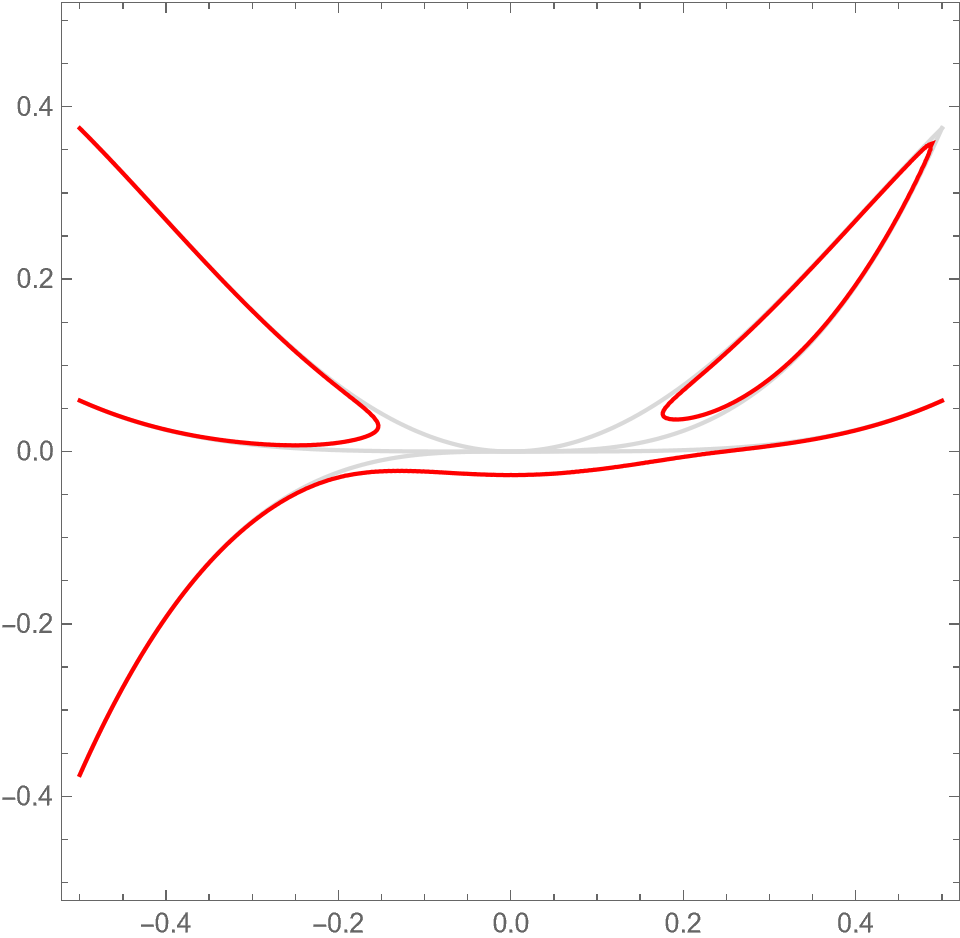}%
\end{center}\caption{Zero set of $G(x,y)=(x^4-y-y^2)(x^2(x^2-1)+\frac12y)(3x^3-y)$ and its perturbation $G-\epsilon F$: $\epsilon=10^{-5}$, $F(x,y)=1$ (left), $F(x,y)=-1$ (right).}\label{fig:graph-scheme}\end{figure}

\begin{rmk}\label{r:signs} When $F(0,0)>0$, the negativity set of $G-\epsilon F$ near the origin is connected. When $F(0,0)<0$, the positivity set of $G-\epsilon F$ near the origin is connected.\end{rmk}

We defer the proof of Lemma \ref{l:graph}, which may be considered a (somewhat long) exercise with the implicit function theorem, to \S\ref{s:appendix}. Note that by rotating coordinates, Lemma \ref{l:lewy} follows from Lemma \ref{l:graph}.

\subsection{Two nodal domains when \texorpdfstring{$d \equiv 2 \pmod 4$}{d is not congruent to 2 mod 4}} \label{subsec:2mod4} When $d=4k+2$ for some $k\geq 0$, the basic hcp $p_d(x,t)=t^{2k+1}+c_{2k}t^{2k}x^2+\cdots+c_0 x^{4k+2}$ in $\R^{1+1}$ (see Definition \ref{def:basic-p}) satisfies \begin{equation*}\label{pd-odd} p_d(0,1)>0\quad\text{and}\quad p_d(0,-1)<0.\end{equation*} This simple observation will let us build hcps in $\R^{2+1}$ of degree $d$ with two nodal domains by essentially copying Lewy's construction of odd degree hhps in $\R^3$ with two nodal domains.

\begin{thm}[{cf.~\cite[Theorem 1]{Lewy77}}]\label{thm:2mod4} Assume $d=4k+2$ for some $k\geq 0$. Let $\psi(x,y)=\Im((x+iy)^d)$ and let $p_d(x,t)$ be the basic hcp in $\R^{1+1}$. For all sufficiently small $\epsilon>0$,
\begin{equation} u_\epsilon(x,y,t) \coloneqq \psi(x,y) - \epsilon p_d(x,t) \label{eqn:lewy_u}
\end{equation} is a time-dependent hcp in $\R^{2+1}$ of degree $d$ and $u_\epsilon$ has two nodal domains.
\end{thm}
\begin{proof}It is clear that $u_\epsilon$ is a time-dependent hcp of degree $d$ whenever $\epsilon\neq 0$. To proceed, we argue that the nodal set of $u_\epsilon$ looks like the one in Figure \ref{fig:mathematica_graphics} (right) when $\epsilon>0$ is small.
Note that the nodal set of $\psi|_{\sphere^2}$ has only two singular points: the north and south pole.

We start with a description of the nodal set of $u_\epsilon|_{\sphere^2}$ near the north pole. Define $v_\epsilon (x,y):=\psi(x,y)-\epsilon p_d(x,1)$ so that by parabolic homogeneity, $$u_\epsilon(x,y,t)= t^{d/2} u_\epsilon(x/\sqrt{t},y/\sqrt{t},1)=t^{d/2}v_\epsilon(x/\sqrt{t},y/\sqrt{t})\quad\text{for all $t>0$}.$$ Hence the nodal set of $u_{\epsilon}|_{\sphere^2}$ in each time-slice $\{t=t_0\}$ (with $0<t_0<1$) agrees with the zeros of $v_{\epsilon}$ on the circle $\{x^2 + y^2 = (1-t_0^2)/t_0 \}$. Thus, the nodal set $u_\epsilon$ on a small spherical cap at the north pole is homeomorphic to the nodal set of $v_\epsilon$ in a small disk at the origin (see Figure \ref{fig:hhp6}). Recall that $p_d(0,1)>0$. Therefore, we can use Lemma \ref{l:lewy} to conclude that when $\epsilon>0$ is sufficiently small, the nodal set of $u_\epsilon$ in a small spherical cap at the north pole consists of $d$ ``southward-opening U-shaped'' curves lying in every other longitudinal sector in $\sphere^2$ of angle $\vartheta:=\frac{\pi}{d}$, starting with $\{0 < \theta < \vartheta\}$ and ending with $\{(2d-2)\vartheta<\theta<(2d-1)\vartheta\}$.

For all $t<0$, parabolic homogeneity instead yields $$u_\epsilon(x,y,t)=|t|^{d/2}u_\epsilon(x/\sqrt{|t|},y/\sqrt{|t|},-1)=|t|^{d/2}\big(\psi(x/\sqrt{|t|},y/\sqrt{|t|})-\epsilon p_d(x/\sqrt{|t|},-1)\big).$$ Since $p_d(0,-1)<0$, we can use Lemma \ref{l:lewy} to conclude that when $\epsilon>0$ is sufficiently small, the nodal set of $u_\epsilon$ in a small spherical cap at the south pole consists of $d$ ``northward-opening U-shaped'' curves lying in every other longitudinal sector in $\sphere^2$ of angle $\vartheta=\frac{\pi}{d}$, starting with $\{\vartheta < \theta < 2\vartheta\}$ and ending with $\{(2d-1)\vartheta<\theta<2d\vartheta\}$.

Outside of the polar regions, i.e.~in the complement of the union of fixed spherical caps at the north and south pole, the nodal set of $\psi|_{\sphere^2}(x,y)$ consists of $d$ disjoint smooth arcs, along which $|\nabla \psi(x,y)|\geq c>0$ for some constant $c$ (depending on the size of the caps). Hence the same is true for $u_\epsilon|_{\sphere^2}$ when $\epsilon$ is sufficiently small by the implicit function theorem.

In the end, since the chambers occupied at the north and south pole alternate, we see that when $\epsilon$ is sufficiently small, the nodal set of $u_\epsilon|_{\sphere^2}$ is a single closed, smooth, Jordan curve. Therefore, $u_\epsilon$ has two nodal domains.\end{proof}

\subsection{Two nodal domains when \texorpdfstring{$d$}{d} is odd} \label{subsec:1mod4} See Figure \ref{fig:mathematica_graphics} (middle) for an illustration of the example in Theorem \ref{thm:d_odd} when $d=5$. {Recall that the basic hcp $p_d(x,t)$ is given in Definition \ref{def:basic-p}}.

\begin{thm}\label{thm:d_odd}
Assume $d\geq 3$ is odd. For all sufficiently small $\epsilon>0$ and $\alpha>0$, \begin{align} \label{eqn:d_odd_u}
u_{\epsilon,\alpha}(x,y,t) := y p_{d-1}(x,t) - \epsilon p_d(x\cos\alpha-y\sin\alpha,t)
\end{align} is a time-dependent hcp in $\R^{2+1}$ of degree $d$ and $u_{\epsilon,\alpha}$ has two nodal domains.
\end{thm}

\begin{proof} Let $d=2k+1$ for some $k\geq 1$. By Theorem \ref{lem:hcp_dim_1}, we can write the basic hcps $p_{d-1}$ and $p_{d}$ in $\R^{1+1}$ from Definition \ref{def:basic-p} as \begin{equation*}p_{d-1}(x,t)=(t+b_1x^2)\cdots(t+b_kx^2)\quad\text{and}\quad p_d(x,t)=x(t+c_1x^2)\cdots(t+c_kx^2)\end{equation*}
for some numbers $0<c_1<b_1<c_2<b_2<\cdots<b_{k-1}<c_k<b_k$. Note that the expression $p_d(x\cos\alpha-y\sin\alpha,t)$ is just the composition of $p_d(x,t)$ with a rotation in the $x$ and $y$ coordinates and the Laplacian is rotationally-invariant. Thus, $p_d(x\cos\alpha-y\sin\alpha,t)$ and $u_{\epsilon,\alpha}(x,y,t)$ are time-dependent hcps in $\R^{2+1}$ for all $\epsilon$ and $\alpha$. To proceed, fix $\varepsilon>0$ and $\alpha>0$ (small) and write $u=u_{\epsilon,\alpha}$, $p=p_{d-1}$, $q=p_d$, and $q_\alpha(x,y,t)=p_d(x\cos\alpha-y\sin\alpha,t)$. Our goal is to show that when $\epsilon$ and $\alpha$ are small enough that $\{u=0\}\cap\sphere^2$ is a Jordan curve, whence $\nodal(u)=\nodal(u|_{\sphere^2})=2$.

\begin{figure}[p]\begin{center}\includegraphics[width=0.795\textwidth]{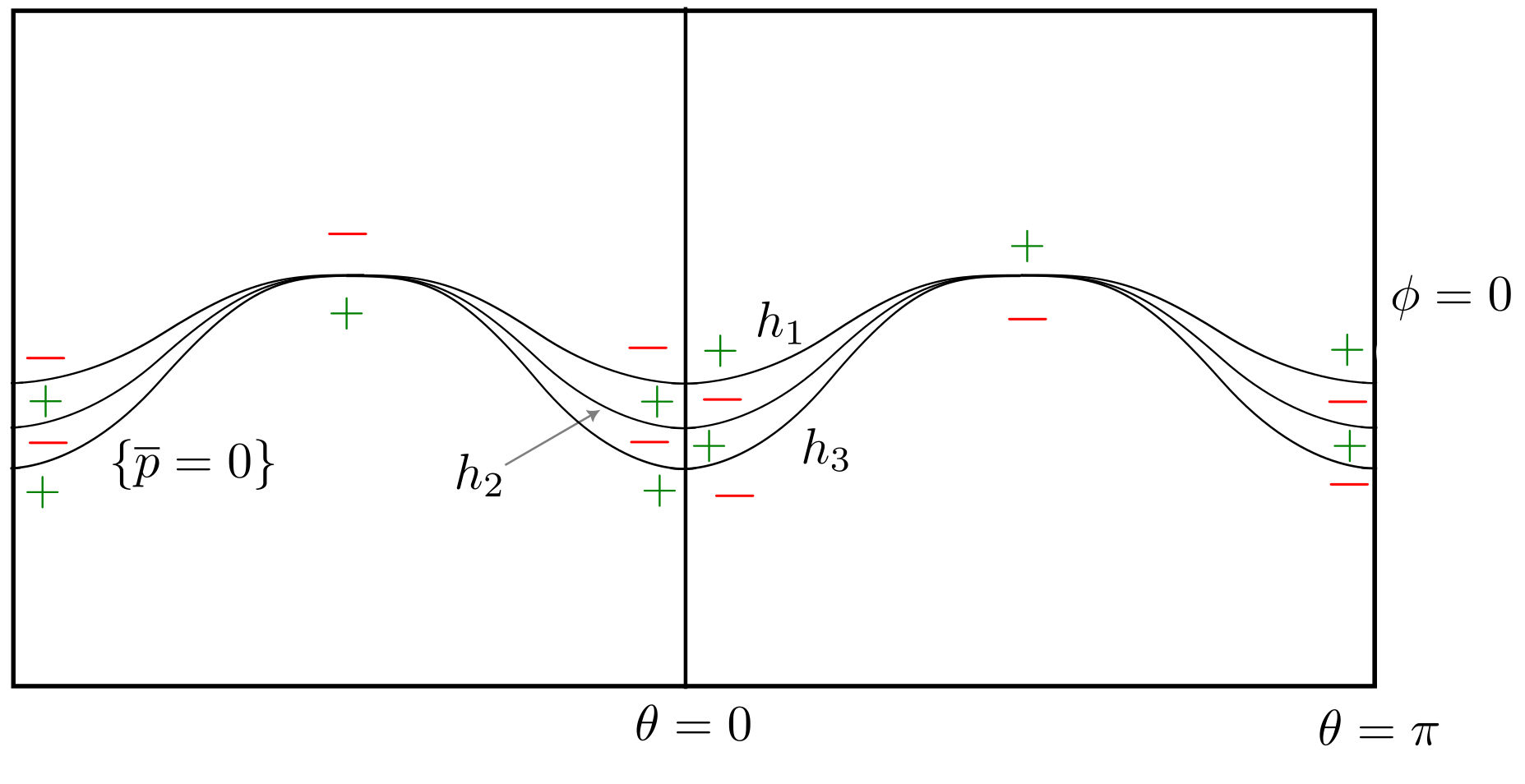}\end{center}
\begin{center}\includegraphics[width=0.795\textwidth]{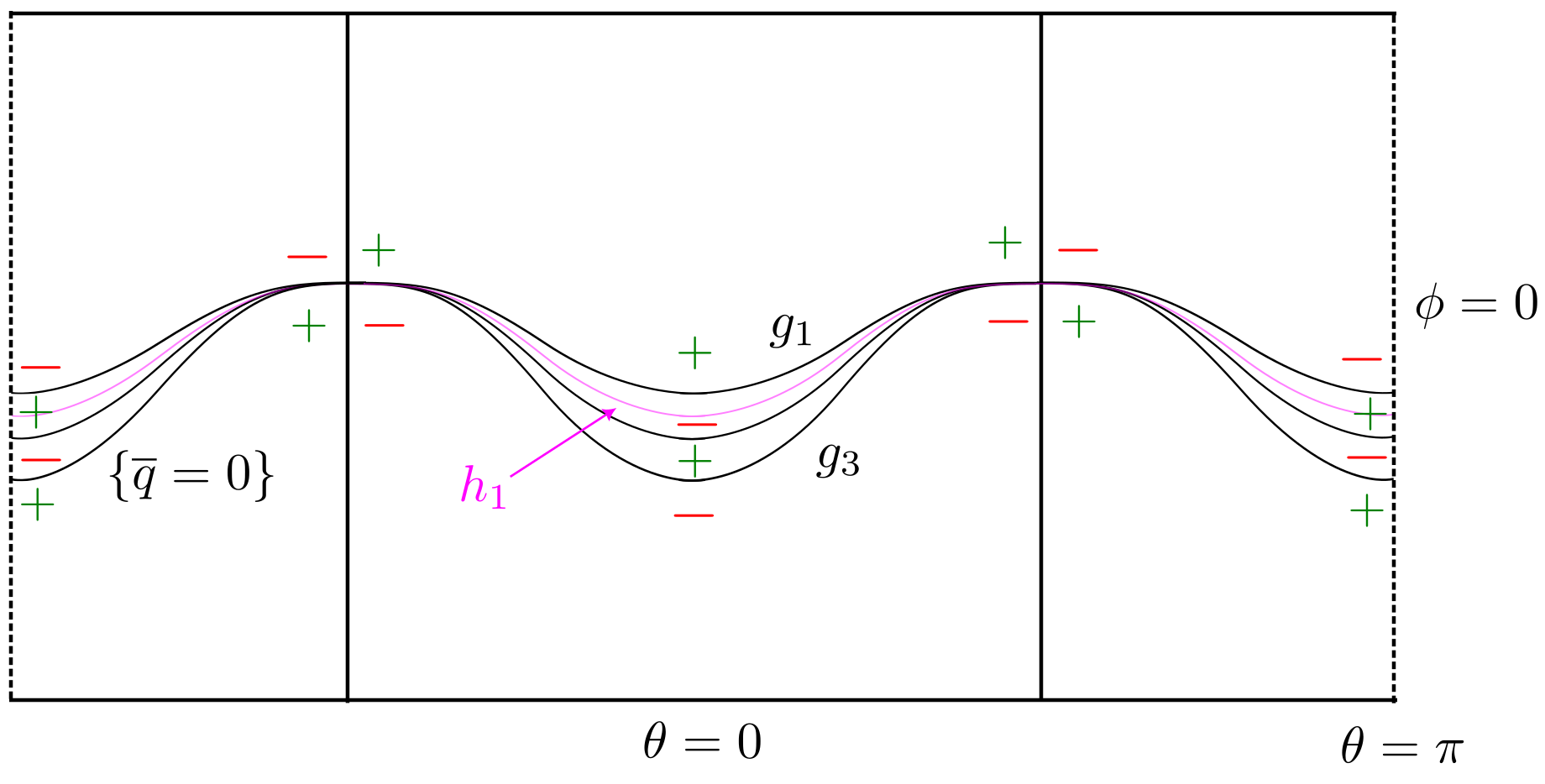}\end{center}
\begin{center}\includegraphics[width=0.795\textwidth]{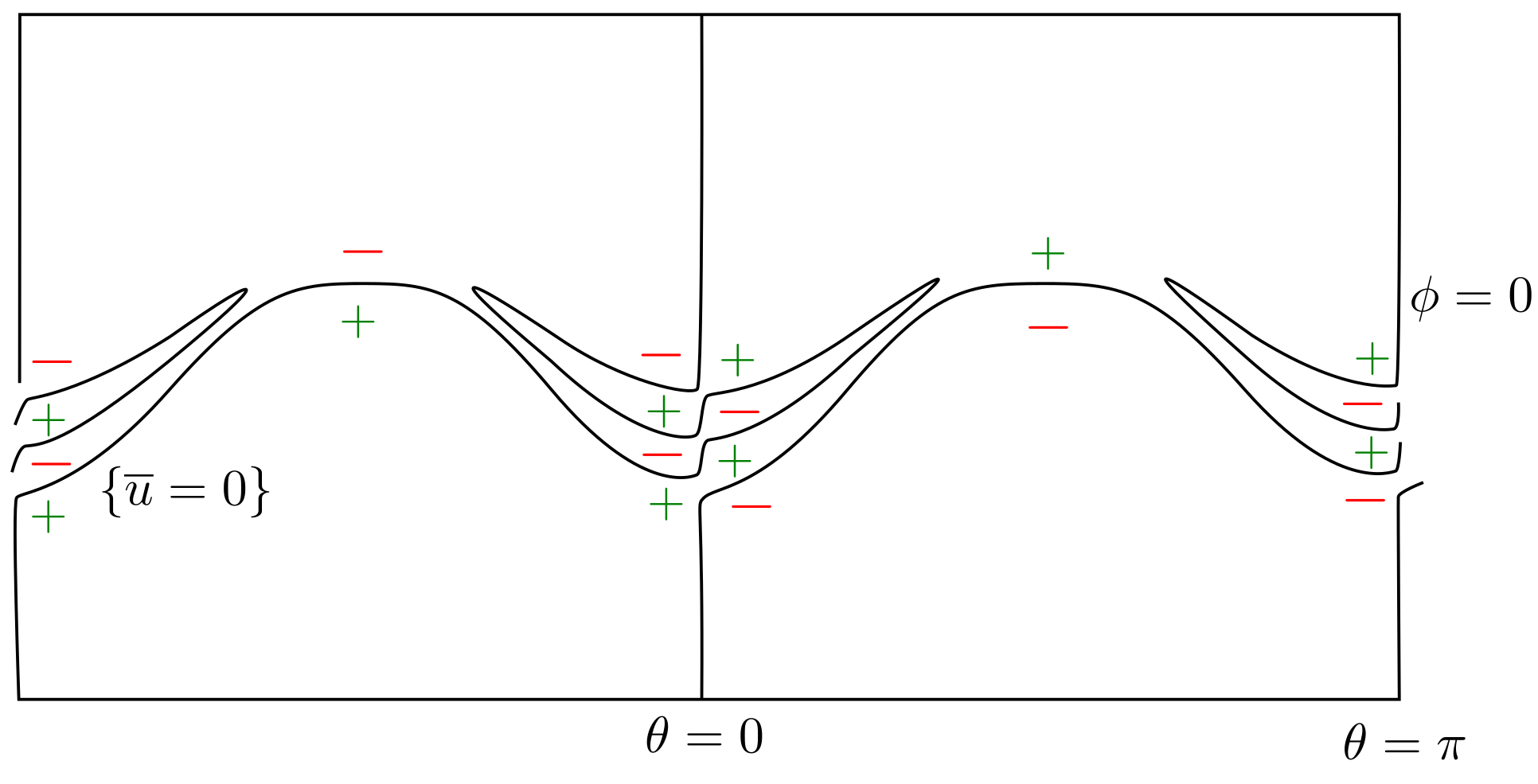}\end{center}
\caption{Proof of Theorem \ref{thm:d_odd} (1/2): Nodal set of $\ol{p}$ (top), $\ol{q}$ (middle), and $\ol{u}$ (bottom) when $k = 3$ and $\epsilon$ and $\alpha$ are sufficiently small.} \label{fig:olp_nodal}
\end{figure}

Consider the standard spherical coordinates on $\sphere^2$ given by
\begin{equation}\label{spherical-coordinates}
x  = \cos \theta \cos \phi, \; y  = \sin \theta \cos \phi, \; t  = \sin \phi,\qquad  -\pi < \theta \le \pi,\; -\pi/2 \le \phi \le \pi/2\end{equation} and write $\ol{p}$, $\ol{q}$, $\ol{q}_\alpha$, and $\ol{u}$ for the functions corresponding to $yp_d(x,t)$, $q(x,t)$, $q_\alpha(x,y,t)$, and $u_{\epsilon,\alpha}(x,y,t)$ on $\sphere^2$ written in spherical coordinates. Hence
\begin{equation}\begin{split}
\ol{p}(\theta,\phi) &= \sin\theta\cos\phi\prod_{i=1}^k\left(\sin\phi + b_i\cos^2\theta\cos^2\phi \right),\\
\ol{q}(\theta,\phi) &= \cos\theta\cos\phi\prod_{i=1}^k\left(\sin\phi + c_i\cos^2\theta\cos^2\phi \right),\\
\ol{q}_\alpha(\theta,\phi) &= \ol{q}(\theta+\alpha,\phi),\quad \ol{u}(\theta,\phi) = \ol{p}(\theta,\phi)-\epsilon \ol{q}_\alpha(\theta,\phi).
\label{ol-defs}\end{split}\end{equation} As an aid for the reader, in Figure \ref{fig:olp_nodal}, we depict nodal domains of $\ol{p}$, $\ol{q}$, and $\ol{u}$ in the $\theta\phi$-plane when $k=3$. When $\alpha>0$ is small, the picture for $\ol{q}_\alpha$ is obtained by translating the picture for $\ol{q}$ slightly to the left. This is crucial for the construction. Observe that the positivity and negativity sets for $\ol{u}$ in the figure are connected. We must explain why this is so.

The nodal set of $\ol{p}$ is comprised of the great circle $\{\theta=0\text{ or }\pi \}$ (including the north and south poles $\{\phi=\pm\frac{\pi}{2}\}$) and the graphs $\{(\theta,h_i(\theta)):-\pi< \theta\leq \pi\}$ of the $k$ functions
\begin{equation} h_i(\theta):=\sin^{-1}\left(\frac{1-\sqrt{1+4b_i^2\cos^4\theta}}{2b_i\cos^2\theta}\right)\qquad(i=1,\dots,k). \label{eqn:yp_soln_i}
\end{equation} To find the formula for $h_i$, expand $\cos^2\phi=1-\sin^2\phi$ in the equation $\sin\phi+b_i\cos^2\theta\cos^2\phi=0$ and use the quadratic formula to solve for $\sin\phi$ (recalling the restriction that $|\sin\phi|\leq 1$). The functions $h_i$ are $\pi$-periodic, $h_i(\pm\pi/2)=0$, and $h_i(\theta)>h_{i+1}(\theta)$ for all $1\leq i\leq k-1$ and $\theta\neq\pm \pi/2$, since $b_1<\cdots<b_k$. Moreover, from the definition of  $\ol{p}$, we observe that $\ol{p}$ takes opposite signs in adjacent nodal domains as indicated in Figure \ref{fig:olp_nodal}.

To understand the nodal structure of $\ol{u}$ when $\epsilon$ is small using Lemma \ref{l:graph} and Remark \ref{r:signs}, we must determine the signs of $\ol{q}_\alpha$ at the singular points in the nodal set of $\ol{p}$ (i.e.~the points where two or more nodal lines of $\ol{p}$ intersect). See Figure \ref{fig:pert_1}.

\begin{figure}
\begin{center}\includegraphics[width=.8\textwidth]{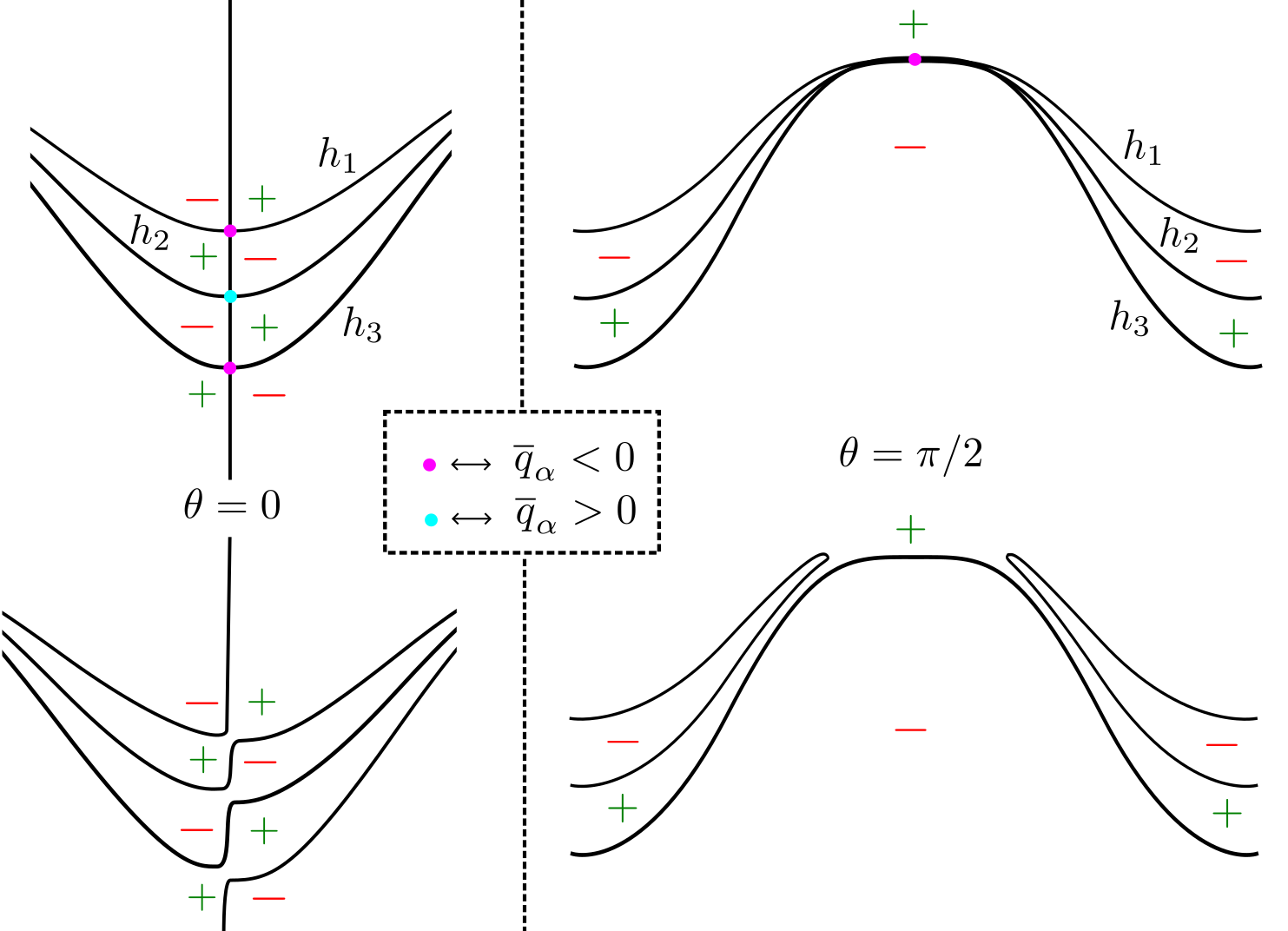} \end{center}
\caption{Proof of Theorem \ref{thm:d_odd} (2/2): Nodal sets of $\ol{p}$ (top) and $\ol{u}$ (bottom) near $\theta = 0$ (left) and $\theta=\pi/2$ (right) when $k=3$. The sign of $\ol{q}_\alpha$ at singular points in the nodal set of $\ol{p}$ determines the local configuration of nodal domains of $\ol{u}$ (see Lemma \ref{l:graph} and Remark \ref{r:signs}).}\label{fig:pert_1}
\end{figure}

The nodal set of $\ol{q}$ is the union of the great circle $\{\theta=\pm\pi/2\}$ and the graphs $\{(\theta,g_i(\theta)):-\pi<\theta\leq \pi\}$ of the $k$ functions \begin{equation} \label{eqn:q_soln_i} g_i(\theta):=\sin^{-1}\left(\frac{1-\sqrt{1+4c_i^2\cos^4\theta}}{2c_i\cos^2\theta}\right)\qquad(i=1,\dots,k). \end{equation} Because of the interlacing property $c_1<b_1<\cdots<c_k<b_k$, the order of the graphs of the functions $g_i$ and $h_i$ alternate: $g_1\geq h_1\geq \cdots \geq g_k\geq h_k$ with strict inequality when $\theta\neq\pm\pi/2$. In particular, along the lines $\theta=0$ and $\theta=\pi$, the sign of the function $\ol{q}$ at the singular point in the nodal set of $\ol{p}$ corresponding to $h_i$ is $(-1)^i$ when $\theta=0$ and $(-1)^{i+1}$ when $\theta=\pi$. By continuity, the same alternating sign pattern persists for $\ol{q}_\alpha$ if $\alpha$ is sufficiently small. At the two remaining singular points $(-\pi/2,0)$ and $(\pi/2,0)$ in the nodal set of $\ol{p}$, the function $\ol{q}$ is zero.\footnote{Exceptionally, when $d=3$ and $k=1$, the nodal set for $\ol{p}$ is regular at $(\pm\pi/2,0)$.}  This is why we introduce the parameter $\alpha$. Taking $\alpha>0$ (and small), we get $\ol{q}_\alpha(-\pi/2,0)=\ol{q}((-\pi/2)+\alpha,0)>0$ and $\ol{q}_\alpha(\pi/2,0)=\ol{q}((\pi/2)+\alpha)<0$.

By the perturbation lemma, the nodal lines for $\ol{p}$ transform into the nodal lines for $\ol{u}$ as indicated in the figures provided that $\epsilon$ is sufficiently small (with $\alpha$ fixed before $\epsilon$). Outside of a neighborhood of the singular points, the nodal lines for $\ol{u}$ are homeomorphic to the nodal lines for $\ol{p}$ by the implicit function theorem; cf.~proof of Theorem \ref{thm:2mod4}. (We have purposely ignored the singularities of $\ol{p}$ on the lines $\phi=\pm\pi/2$, because the nodal lines of $p|_{\sphere^2}$ are regular at the north and south poles.) A careful stitching of the local structure of $\{\ol{u} =0\}$ yields that $\{u=0\} \cap \sphere^2$ is a single Jordan curve.
%Although our pictures have been drawn for $k$ odd, the same exact construction holds for $k$ even as well. Notice that there is a qualitative difference in $k$ even or odd though, since the number of graphs $h_i$ determined whether $\overline{p}$ has the same (or different) sign above and below all of the $h_i$ at $\theta = \pi/2$.
\end{proof}

\subsection{Three nodal domains when \texorpdfstring{$d \equiv 0 \pmod 4$}{d is not congruent to 0 mod 4}}\label{subsec:0mod4} See Figure \ref{fig:mathematica_graphics} (left) for an illustration of the example in Theorem \ref{thm:0mod4} when $d=4$. {Recall that the basic hcp $p_d(x,t)$ is given in Definition \ref{def:basic-p}.}

\begin{thm}\label{thm:0mod4}
Assume $d=4k$ for some $k\geq 1$. For all sufficiently small $\epsilon > 0$ and $\alpha >0$,
\begin{align}\label{eqn:u_0_mod_4}
u_{\epsilon,\alpha}(x,y,t) := p_{2k}(x,t)p_{2k}(y,t) + \epsilon p_{2k+1}(x\cos\alpha-y\sin\alpha,t)p_{2k-1}(x\sin\alpha+y\cos\alpha,t)
\end{align}
is a time-dependent hcp in $\R^{2+1}$ of degree $d$ and $u_{\epsilon,\alpha}$ has three nodal domains.
\end{thm}
\begin{proof} The shell of the proof is the same as for proof of Theorem \ref{thm:d_odd}, but the construction is more intricate.
Assign $p(x,y,t) := p_{2k}(x,t)p_{2k}(y,t)$ and $q(x,y,t):=p_{2k+1}(x,t)p_{2k-1}(y,t)$. By Theorem \ref{lem:hcp_dim_1}, we may express
\begin{align*}
p(x,y,t) & = (t + b_1 x^2) \cdots (t + b_k x^2) (t+ b_1 y^2) \cdots (t + b_ky^2),\\
q(x,y,t) & = xy(t + c_1 x^2)\cdots (t + c_{k} x^2)(t+a_1 y^2) \cdots (t + a_{k-1} y^2),
\end{align*} where $a_1,\dots,a_{k-1}$, $b_1,\dots,b_k$, and $c_1,\dots, c_k$ are positive constants satisfying \eqref{pd-interlace}. Fix small parameters $\epsilon>0$ and $\alpha>0$ and assign $q_\alpha(x,y,t) := q(x \cos \alpha - y \sin  \alpha, x \sin \alpha + y \cos \alpha, t)$ and $u:=u_{\epsilon,\alpha}=p+\epsilon q_\alpha$.  Since the Laplacian is rotationally invariant, we conclude that $q_\alpha$ and $u$ are time-dependent hcps of degree $d$. Our goal is to prove that the nodal set of $u|_{\sphere^2}$ is the union of two disjoint, closed Jordan curves, which implies that $u$ has 3 nodal domains.

\begin{figure}[p]
\begin{center}
\includegraphics[width=0.8\textwidth]{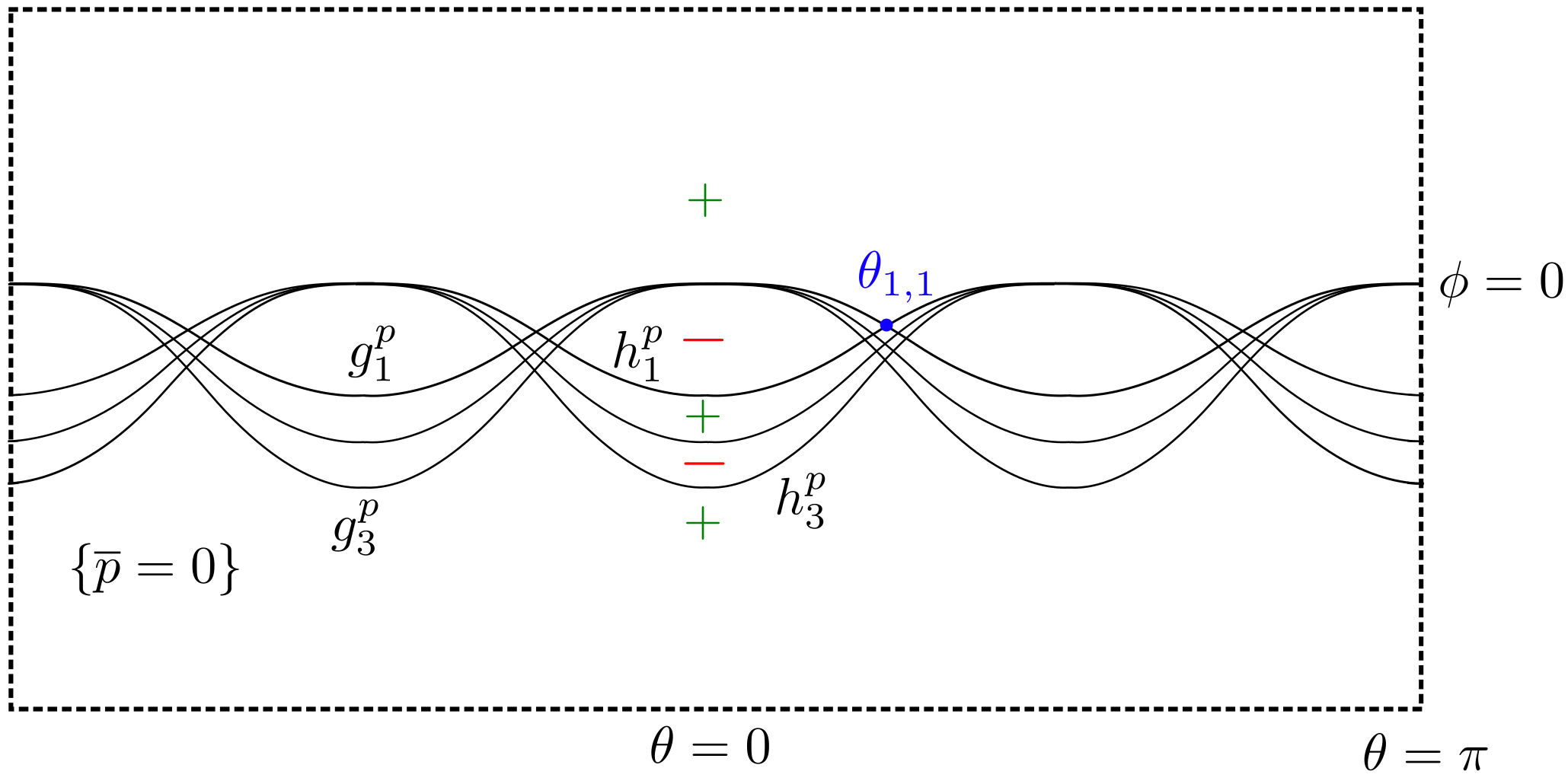}\end{center}
\begin{center}\includegraphics[width=0.8\textwidth]{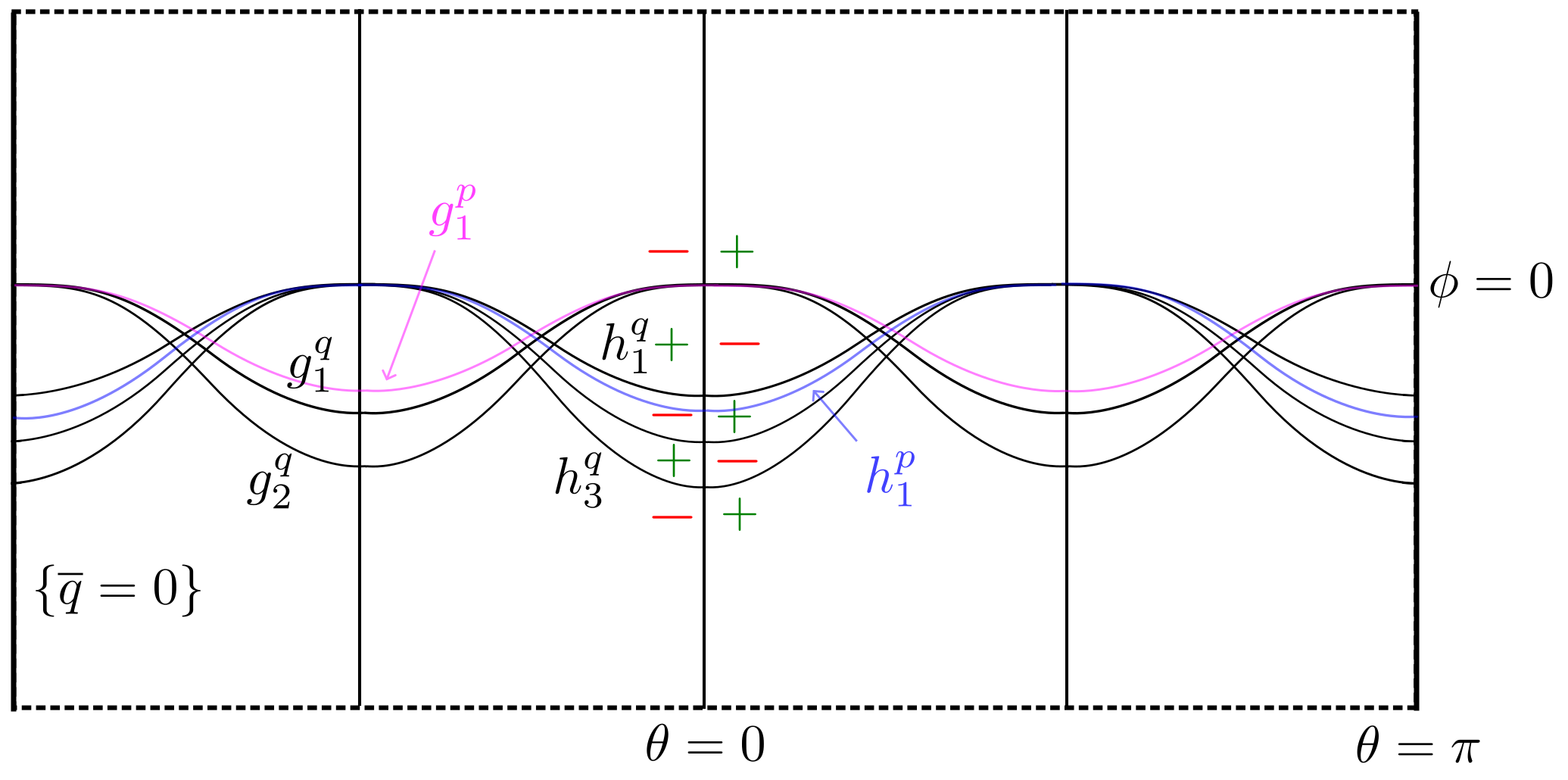}\end{center}
\begin{center}\includegraphics[width=0.8\textwidth]{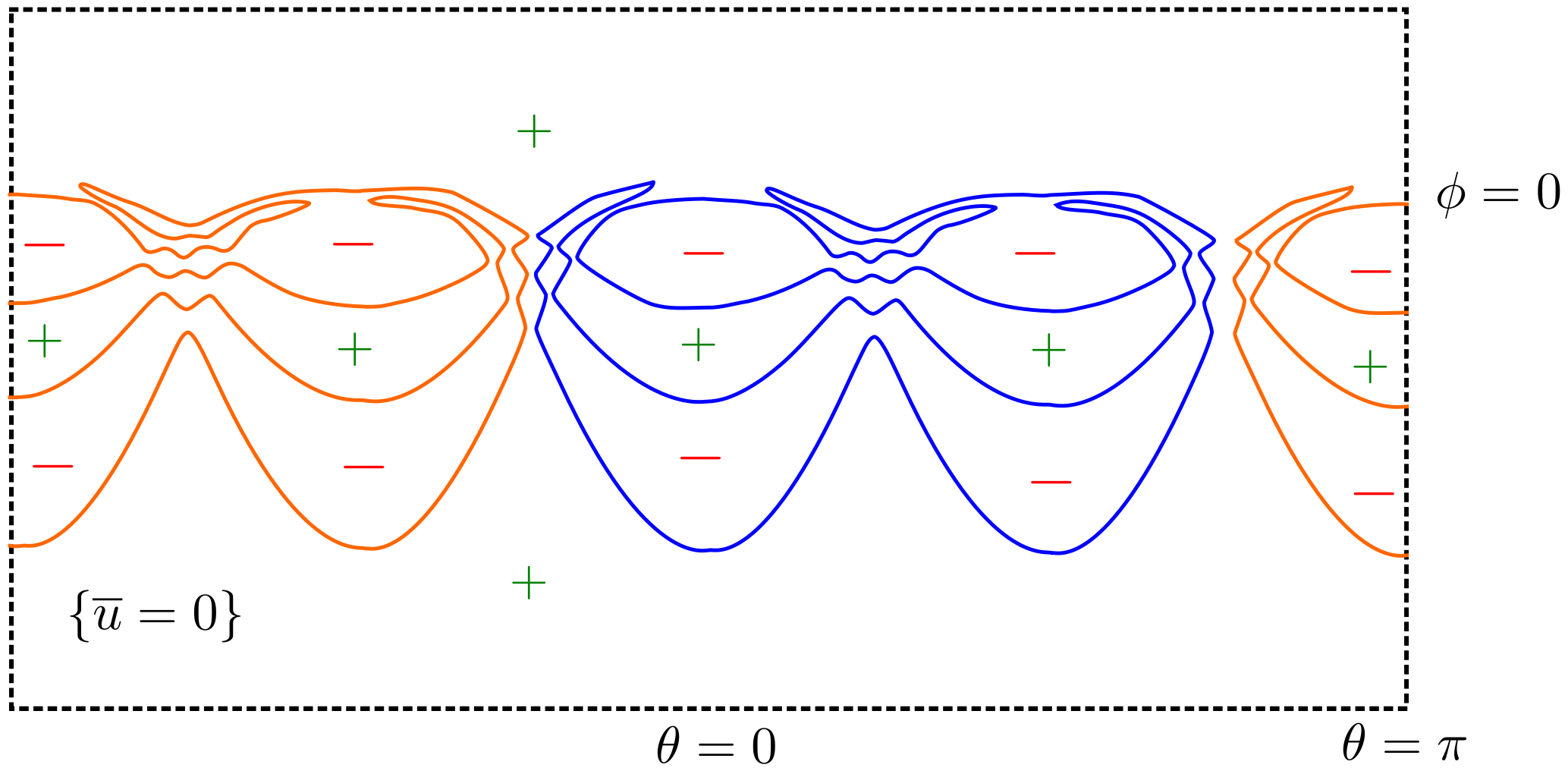}\end{center}
\caption{Proof of Theorem \ref{thm:0mod4} (1/2): The nodal set of $\ol{p}$ (top), $\ol{q}$ (middle), and $\ol{u}$ when $k = 3$ and $\epsilon$ and $\alpha$ are sufficiently small. The two Jordan curves which form $\{\overline{u}=0\}$ are depicted in blue and orange.}\label{fig:0m4_p_nodal}
\end{figure}

Once again, we write $\ol{p}$, $\ol{q}$, $\ol{q}_\alpha$, and $\ol{u}$ to denote the functions $p$, $q$, $q_\alpha$, and $u$ expressed in the standard spherical coordinates \eqref{spherical-coordinates}. Thus,
\begin{equation}
\begin{split}
\ol{p}(\theta, \phi) & = \prod_{i=1}^k (\sin\phi + b_i \cos^2\theta \cos^2\phi)\prod_{i=1}^k(\sin\phi+b_i\sin^2 \theta \cos^2 \phi), \\
\ol{q}(\theta, \phi) & = \cos \theta \sin \theta \cos^2\phi \prod_{i=1}^{k} (\sin\phi+c_i\cos^2\theta\cos^2 \phi)  \prod_{i=1}^{k-1} (\sin\phi+a_i\sin^2\theta\cos^2 \phi).
\end{split}
\end{equation} See Figure \ref{fig:0m4_p_nodal} for an illustration of the nodal domains of $\ol{p}$, $\ol{q}$, and $\ol{u}$ when $k=3$. The remainder of the proof is devoted to showing that $\ol{u}$ has two disjoint, closed nodal curves, bounding one positive component and two negative components.

The nodal set of $\ol{p}$ is the union of the $2k$ graphs of the functions $h^p_i(\theta)$ defined by \eqref{eqn:yp_soln_i} and the functions $g^p_i(\theta):=h^p_i(\theta-\pi/2)$ for all $1\leq i\leq k$. (Here we used the elementary fact that the positive $y$-axis is the rotation of the positive $x$-axis by $\pi/2$ radians.) Recall that the functions $h^p_i$ are $\pi$-periodic, $h^p_i(\pm\pi/2)=0$, and $h^p_i(\theta)>h^p_{i+1}(\theta)$ for all $1\leq i\leq k-1$ and $\theta\neq\pm \pi/2$. Similarly, the functions $g^p_i$ are $\pi$-periodic, $g^p_i(0)=g^p_i(\pi)=0$, and $g^p_i(\theta)>g^p_{i+1}(\theta)$ for all $1\leq i\leq k -1$ and $\theta\neq 0,\pi$. The sign of $\ol{p}$ changes across adjacent nodal domains.

\begin{figure}
\begin{center}\includegraphics[width=.4\textwidth]{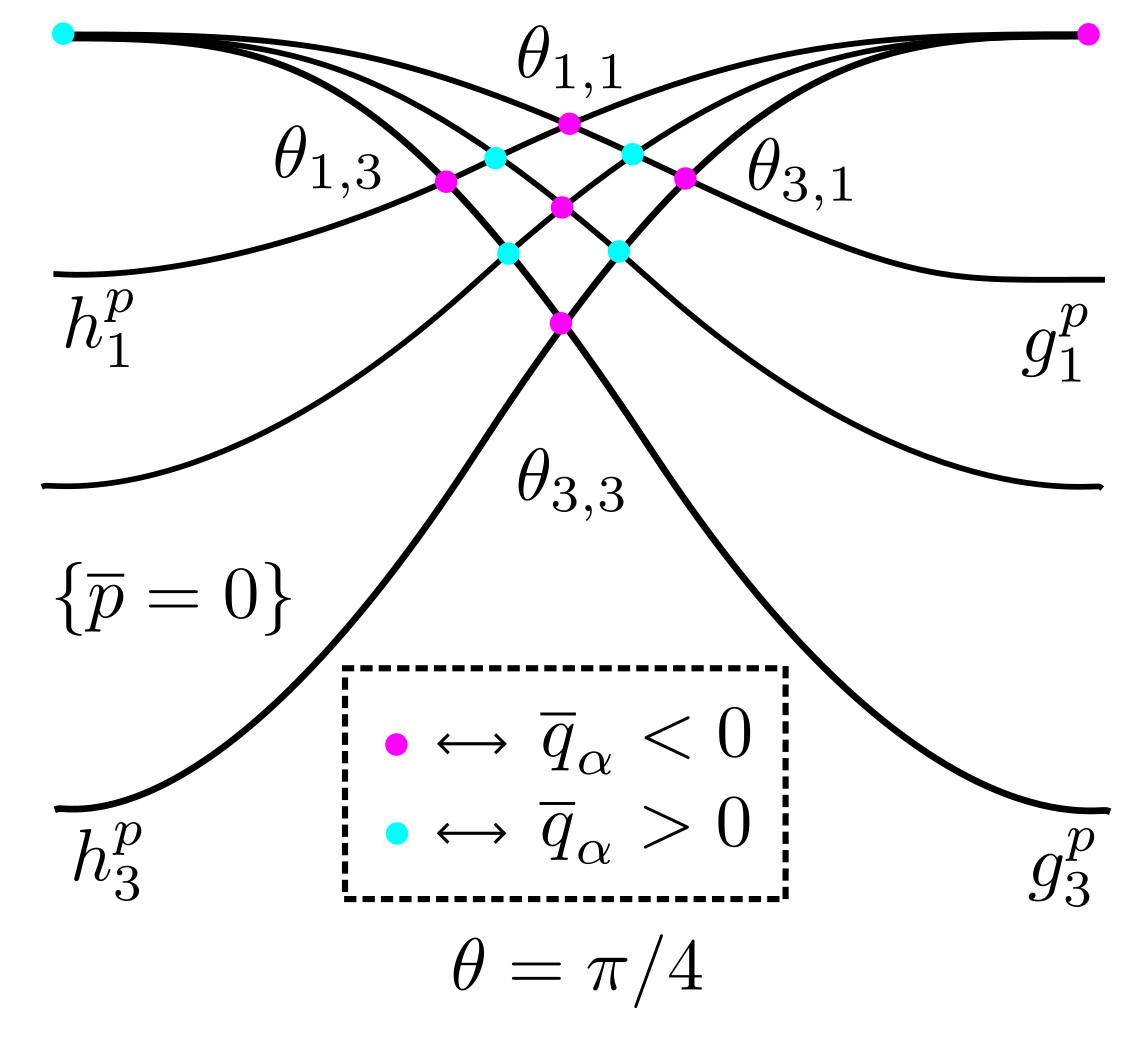}\end{center}%
\caption{Proof of Theorem \ref{thm:0mod4} (2/2): Sign pattern for $\ol{q}_\alpha$ at singular points in the nodal set of $\ol{p}$ near $\theta = \pi/4$ when $k = 3$.}\label{fig:0m4_theta}
\end{figure}

As usual, we must next identify the singular points in the nodal set of $\ol{p}$. The points $(-\pi/2,0)$ and $(\pi/2,0)$ common to the graph of each $h^p_i$ and the points $(0,0)$ and $(\pi,0)$ common to the graph of each $g^p_i$ are singular points.\footnote{When $d=4$ and $k=1$, the nodal set for $\ol{p}$ is regular at $(-\pi/2,0)$, $(0,0)$, $(\pi/2,0)$, and $(\pi,0)$.} In addition, the nodal set of $\ol{p}$ has $4k^2$ singular points where the graph of an $h^p_i$ intersects the graph of a $g^p_j$. For each $1\leq i,j\leq k$, there is a unique angle $\theta_{i,j}\in(0,\pi/2)$ so that the graphs of $h^p_i$ and $g^p_j$ intersect precisely at
\begin{equation}\label{hg-crossings}
(\theta_{i,j}, h_i^p(\theta_{i,j})),\quad (\theta_{i,j}-\pi, h_i^p(\theta_{i,j})),\quad (-\theta_{i,j}, h_i^p(\theta_{i,j})),\quad (-\theta_{i,j}+\pi, h_i^p(\theta_{i,j})).\end{equation}
%Here the $\theta_{i,j}$ for $1 \le i, j \le k$ are the unique angles $\theta \in (0,\pi/2)$ where $h_i^p(\theta) = g_j^p(\theta)$ (the uniqueness of such an angle follows from the fact that the $h_i^p$ are strictly increasing on $(0, \pi/2)$, while the $g_i^p$ are strictly decreasing there). Notice that by periodicity of $h_i^p$ and $g_j^p$ (as well as evenness about the angles $-\pi/2, 0, \pi/2, \pi$), the graphs of $h_i^p$ and $g_j^p$ intersect $3$ more times in the rectangle $R$: at the angles $ \theta_{i_,j} - \pi, -\theta_{i,j}$, and $-\theta_{i,j} + \pi$, which gives the description of $S$ above.

The nodal set of $\ol{q}$ includes two great circles $\{(\theta, \phi) \; : \; \theta = -\pi/2, 0, \pi/2, \pi\}$ and the $2k-1$ graphs of functions $h_i^q$ defined by \begin{equation} h^q_i(\theta):=\sin^{-1}\left(\frac{1-\sqrt{1+4c_i^2\cos^4\theta}}{2c_i\cos^2\theta}\right)\qquad(i=1,\dots,k)
\end{equation} and the functions $g^q_i$ defined by
\begin{equation}
g_i^q(\theta) \coloneqq \sin^{-1} \left( \frac{1-\sqrt{1+4a_i^2\sin^4\theta}}{2a_i\sin^2\theta} \right) \qquad (i=1,\dots, k-1).
\end{equation}
The sign of $\ol{q}$ alternates on adjacent nodal domains. In view of \eqref{pd-interlace}, we have
\begin{align}
h_1^q & \ge h_1^p \ge h_2^q \ge h_2^p \ge \cdots \ge h_{k}^q \ge h_k^p, \label{eqn:order1} \\
g_1^p & \ge g_1^q \ge g_2^p \ge g_2^q \ge \cdots \ge g_{k-1}^q \ge g_k^p, \label{eqn:order2}
\end{align} with strict inequalities in \eqref{eqn:order1} unless $\theta=\pm\pi/2$ and strict inequalities in \eqref{eqn:order2} unless $\theta = 0, \pi$. From the strict inequalities, periodicity, and evenness, one can deduce the following sign pattern for $\ol{q}$ at the singular points in the nodal set of $\overline{p}$:
\begin{equation}\begin{split}\label{crazy-signs}
\mathrm{sgn}(\ol{q}(\theta_{i,j}, h_i^p(\theta_{i,j})) & = \mathrm{sgn}(\ol{q}(\theta_{i,j} - \pi, h_i^p(\theta_{i,j})) = (-1)^{i+j+1}, \\
\mathrm{sgn}(\ol{q}(-\theta_{i,j}, h_i^p(\theta_{i,j})) & = \mathrm{sgn}(\ol{q}(-\theta_{i,j} + \pi, h_i^p(\theta_{i,j})) = (-1)^{i+j}.
\end{split}\end{equation} See Figure \ref{fig:0m4_theta}. By continuity, \eqref{crazy-signs} persists for $\ol{q}_\alpha$ if $\alpha>0$ is small enough. In addition,
\begin{equation}\label{less-crazy-signs}
\ol{q}_\alpha(0,0) = \ol{q}_\alpha(\pi, 0)  > 0,\quad \ol{q}_\alpha(-\pi/2,0) = \ol{q}_\alpha(\pi/2, 0)  <0.
\end{equation}

Near the singular points of $\ol{p}$, we again apply Lemma \ref{l:graph} / Remark \ref{r:signs} (in a rotated coordinate system, as necessary) to see that locally the nodal set of $\ol{u}$ is given by finitely many simple curves, which are contained either in the positive or negative components of $\{\ol{p} \ne 0\}$, if $\ol{q}_\alpha$ is negative or positive there, respectively. Moreover, the curves approach the nodal set of $\ol{p}$ when $\epsilon \da 0$. (Note that $\ol{u}=\ol{p}+\epsilon\ol{q}_\alpha = \ol{p}-\epsilon(-\ol{q}_\alpha)$.) Careful piecing of all of this information together, one deduces that for $\alpha >0$ and $\epsilon >0$ small, the nodal domain of $\ol{u}$ consists of two disjoint simple curves that separate $\R^{2}$ into three connected components (one positive connected component and two negative connected components).

Indeed away from the singular points of $\ol{p}$, the nodal set of $\ol{u}$ consists of smooth curves tending to the nodal set of $\ol{p}$ as $\epsilon \da 0$. The key observation is then that at the singular points corresponding to the $\theta_{i,j}$ and $\theta_{i,j} - \pi$, connectivity is gained in the horizontal direction. That is to say, at the points $(\theta_{i,j}, h_i^p(\theta_{i,j}))$ and $(\theta_{i+1, j+1}, h_{i+1}^p(\theta_{i+1, j+1}))$, $\ol{q}_\alpha$ has the same sign, which is equal to the sign of $\ol{p}$ at $(\theta, h_i^p(\theta_{i,j})) $ for $\theta \in (\theta_{i,j} - \delta, \theta_{i,j} + \delta) \setminus \{\theta_{i,j}\}$ and $\delta >0 $ sufficiently small. Thus for the singular points $(\theta_{i,j}, h_i^p(\theta_{i,j}))$, horizontally-adjacent chambers of $\ol{p}$ become connected in the nodal domain of $\ol{u}$ by Lemma \ref{l:graph}. On the other hand, connectivity is gained in the vertical direction at the singular points corresponding to the $-\theta_{i,j}, -\theta_{i,j} + \pi$ in similar manner. Along with the fact that all positive chambers of $\ol{p}$ meeting $(0, 0)$ and $(0, \pi)$ become connected in the nodal domain of $\ol{u}$, and all the negative chambers of $\ol{p}$ meeting $(\pm \pi/2, 0)$ become connected in the nodal domains of $\ol{q}$, one can conclude that the positivity set of $\ol{u}$ is connected, and $\ol{u}$ has only $2$ negative chambers. See Figure \ref{fig:0m4_p_nodal}.
\end{proof}

\section{Proof of Lemma \ref{l:graph}} \label{s:appendix}

Towards the proof of Lemma \ref{l:graph}, we start with an easy variation of \cite[Lemma 3]{Lewy77}, which assumed that $g(x)=x^k$ for some integer $k\geq 2$ and $f$ is real-analytic.

\begin{lemma}\label{lem:z_pert_1} Suppose that $f,g:[0,t]\rightarrow\R$ are Lipschitz functions such that $f(0)=1$ and for some numbers $k>1$, $a>0$, and $C>0$, \begin{equation}|g(x)-a x^k|\leq C x^{k+1}\quad\text{for all $x\in[0,1]$,}\end{equation} \begin{equation}|g'(x)-ak x^{k-1}|\leq C x^{k}\quad\text{and}\quad |f'(x)|\leq C\quad\text{for a.e.~$x\in[0,1]$.}\end{equation} There exist $\epsilon_0=\epsilon_0(C,a,k)>0$ such that for all $\epsilon \in (0, \epsilon_0)$, the function $F_\epsilon:[0,1]\rightarrow\R$, \begin{equation}F_\epsilon(x) = g(x) - \epsilon f(x)\quad\text{for all $x\in[0,1]$},\end{equation} has a unique root $x_0\in(0,\tau)$, where $\tau=\min\{1,\frac{1}{2}C^{-1}ak\}$. Moreover, $x_0\in (\epsilon^{1/(k-\frac12)},\epsilon^{1/(k+1)})$, $F_\epsilon(x)<0$ for all $0\leq x<x_0$, $F_\epsilon(x)>0$ for all $x_0<x\leq \tau$, and $F_\epsilon$ is strictly increasing on $[\epsilon^{1/(k-\frac12)},\tau]$.
\end{lemma}

\begin{proof}Let $0<\epsilon\leq \epsilon_0<1$. Write $j=k-\frac{1}{2}$ and $\ell=k+1$, so that $k-1<j<k<\ell$. For all $0\leq x\leq \epsilon^{1/j}$, we have \begin{equation}\begin{split} \label{A1}
F_\epsilon(x) &\leq ax^{k}+Cx^{\ell}-\epsilon(1-Cx)\\
     &\leq a\epsilon^{k/j}+C\epsilon^{\ell/j}+\epsilon (C\epsilon^{1/j}-1)= (a\epsilon^{1/(2j)}+C\epsilon^{3/(2j)}+C\epsilon^{1/j}-1)\epsilon\leq -\frac{1}{2}\epsilon
\end{split}\end{equation} provided that $\epsilon_0$ (hence $\epsilon$) is sufficiently small depending only on $C$, $k$, and $a$. Similarly, for all $\epsilon^{1/\ell}\leq x \leq \min\{1,\frac{1}{2}C^{-1}a\}$, \begin{equation}\begin{split} F_\epsilon(x) \geq ax^{k}-Cx^{k+1}-\epsilon(1+C)&= (a-Cx)x^k-\epsilon(1+C)\\ &\geq \tfrac{1}{2}a \epsilon^{k/\ell}-\epsilon(1+C)\geq (\tfrac12a \epsilon^{-1/\ell}-1-C)\epsilon\geq \epsilon
\end{split}\end{equation} provided that $\epsilon_0$ is sufficiently small depending only on $C$, $k$, and $a$. Since $F$ is continuous, it follows that $F$ has at least one root $x_0$ in the interval $(\epsilon^{1/j},\epsilon^{1/\ell})$. Now, for any $\epsilon^{1/j}\leq x\leq \min\{1,\frac{1}{2}C^{-1}ak\}=:\tau$ at which $F$ is differentiable, \begin{equation}\begin{split} F_\epsilon'(x) \geq ak x^{k-1}-Cx^k-C\epsilon&=(ak-Cx)x^{k-1}-C\epsilon\\& \geq \tfrac{1}{2}ak \epsilon^{(k-1)/j}-C\epsilon\geq (\tfrac{1}{2}ak\epsilon^{-1/(2j)}-C)\epsilon\geq\epsilon
\end{split}\end{equation} provided that $\epsilon_0$ is sufficiently small depending only on $C$, $k$, and $a$. It follows that $F$ is strictly increasing on $[\epsilon^{1/j},\tau]$. Choosing $\epsilon_0$ to be sufficiently small guarantees that $\epsilon^{1/j}<x_0<\epsilon^{1/\ell}\leq \tau$. In conjunction with \eqref{A1}, this shows that $F_\epsilon(x)<0$ for all $0\leq x<x_0$ and $F_\epsilon(x)>0$ for all $x_0<x\leq \tau$. Therefore, $x_0$ is the unique root of $F_\epsilon$ in $(0,\tau)$.\end{proof}

\begin{rmk}\label{rmk:zero}
The proof of Lemma \ref{lem:z_pert_1} shows that in fact, $\epsilon_0$ and $\tau$ can be chosen to depend only on $C$, $k$ and $a_0 >1$ provided that $a \in [a_0^{-1}, a_0]$.
\end{rmk}

Using Lemma \ref{lem:z_pert_1}, we prove Lemma \ref{l:graph}.

%\begin{lemma} Suppose that $h^\pm:[0,1)\rightarrow\R$ are real-analytic functions such that $h^\pm(0)=0$ and $h^-(x)<h^+(x)$ for all $x\in(0,1)$. For all $0<x_0\leq 1$, let \begin{equation}U_{x_0}:=\{(x,y)\in\R^2:0<x<x_0,\; h^-(x)<y<h^+(x)\}\end{equation} be the open region between the graphs of $h^\pm|_{(0,x_0)}$, denoted by \begin{equation}\Gamma^\pm_{x_0}:=\{(x,h^\pm(x)):0<x<x_0\}.\end{equation} If $G$ is real-analytic on $\overline{U_1}$, $G>0$ on $U_1$, $G=0$ on $\Gamma^\pm_1$, $\nabla G(0,0)=0$, $F\in C^1(\overline{U_1})$, and $F(0,0)=1$, then there exist $0<\tau<1$ and $\epsilon_0>0$ such that for all $0<\epsilon\leq \epsilon_0$, the nodal set of $G-\epsilon F=0$ in $U_\tau$ is a $C^1$ Jordan arc $\gamma_\epsilon$ and $\lim_{\epsilon\rightarrow 0} \gamma_\epsilon = \Gamma^+_{\tau}\cup\Gamma^-_\tau$ in the Hausdorff distance. A similar conclusion holds if $h^+\equiv +\infty$ (resp.~$h^-\equiv -\infty$) provided that we interpret $\Gamma^+_{x_0}=\emptyset$ (resp.~$\Gamma^-_{x_0}=\emptyset$) and $U_{x_0}$ to be the area above $\Gamma^-_{x_0}$ (resp.~the area below $\Gamma^+_{x_0}$).\end{lemma}
%

\begin{proof}[Proof of Lemma \ref{l:graph}]
Let $F$ and $G=\prod_{i=1}^m g_i$ be given as in the statement of the lemma and let $\epsilon>0$ be small. Because $G-\epsilon F$ and $-G-\epsilon(-F)=-(G-\epsilon F)$ have the same zero set, we may assume without loss of generality that $F(0,0)>0$. Now, since $\partial_y g_i(0,0) \ne 0$, the real-analytic implicit function theorem (see e.g.~\cite[Theorem 2.3.5]{primer-real-analytic}) implies that in a neighborhood of the origin, the zero set of $g_i$ is given by the graph of real analytic function of one variable, $h_i(x)$, so $g_i(x, h_i(x)) = 0$. Each $h_i$ satisfies $h_i(0) =0$ since $g_i$ vanishes at the origin. Let us continue by proving the result for the nodal set of $G-\epsilon F$ in $\{x \ge 0\}$, since similar reasoning applies to the nodal set in $\{x \le 0\}$.

Note that since the $g_i$ only share a common zero at the origin, the functions $h_i(x)$ are distinct real-analytic functions. In particular, for $\tau$ chosen small enough, then the $h_i$ must be ordered, and thus we may as well assume that $h_1(x) < h_2(x) < \cdots < h_m(x)$ for $ 0 < x < \tau$. For each $i$, we write $h_i(x)  = \sum_{k = k_i}^\infty a_{i,k} x^k$,
where $a_{i,k_i} \ne 0$ is the first non-zero in the expansion of $h_i$.
%\footnote{It does little harm to assume all the $k_i$ are the same here. The idea of the proof does not change, but more technical details are added. Unfortunately though, we do need the result in the case when they are not identical, so we provide a proof of the more general version.}
Of course $k_i \ge 1$ since $h_i(0) = 0$. It is straightforward to see from the fact that the nodal set of $g_i$ is given by the graph of $h_i$, that the Taylor series of $g_i$ takes the form
\begin{align}
g_i(x,y) & = y \left( b_i^* + \sum_{\alpha = (\alpha_1, \alpha_2), \abs{\alpha} \ge 1} b_\alpha x^{\alpha_1} y^{\alpha_2}   \right) + \sum_{k = k_i}^\infty b_{i, k} x^k, \label{eqn:gi_ser}
\end{align}
where $b_i^* = \partial_y g_i(0,0) \ne 0$, and $b_{i, k_i} \ne 0$. In particular, the first nonzero pure $x^k$ term in the expansion for $g_i$ has the same power as that of $h_i$ (and in fact, $a_{i,k_i} = -b_{i, k_i}/b_i^*$).

We remark that since $\partial_y g(0,0) \ne 0$, $g_i$ changes sign about its nodal set and $G$ changes sign about the graphs $h_i$. Note that for $\epsilon_0$ and $\tau$ small, the zero set of $G - \epsilon F$ in $B_\tau(0)$ is contained in the positivity set of $G$, since $F (0,0) > 0$. Hence we consider a chamber \[U_{i_0} \coloneqq B_\tau(0) \cap \{(x,y) \; : \; x > 0, \;  h_{i_0}(x) < y < h_{i_0+1}(x)\}.  \] (the case when one of these chambers is just $\{y > h_m\}$ or $\{ y < h_1\}$ is similar) and prove the conclusion of the lemma inside this chamber. First, we need some estimates on $G$ and $\nabla G$, especially near the graphs of $h_{i_0}$ and $h_{i_0+1}$, so define the region $U_{i_0}^\delta$ for $\delta >0$ small by
\begin{align*}
U_{i_0}^\delta \coloneqq \left \{(x,y) \in U_{i_0} \; : \; y = t h_{i_0}(x) + (1-t) h_{i_0+1}(x) \text{ for some } t \in [0, \delta) \cup (1-\delta, 1] \right \}.
\end{align*}
See Figure \ref{fig:ui0}.

\begin{figure}
\begin{center}\includegraphics[width=0.55\textwidth]{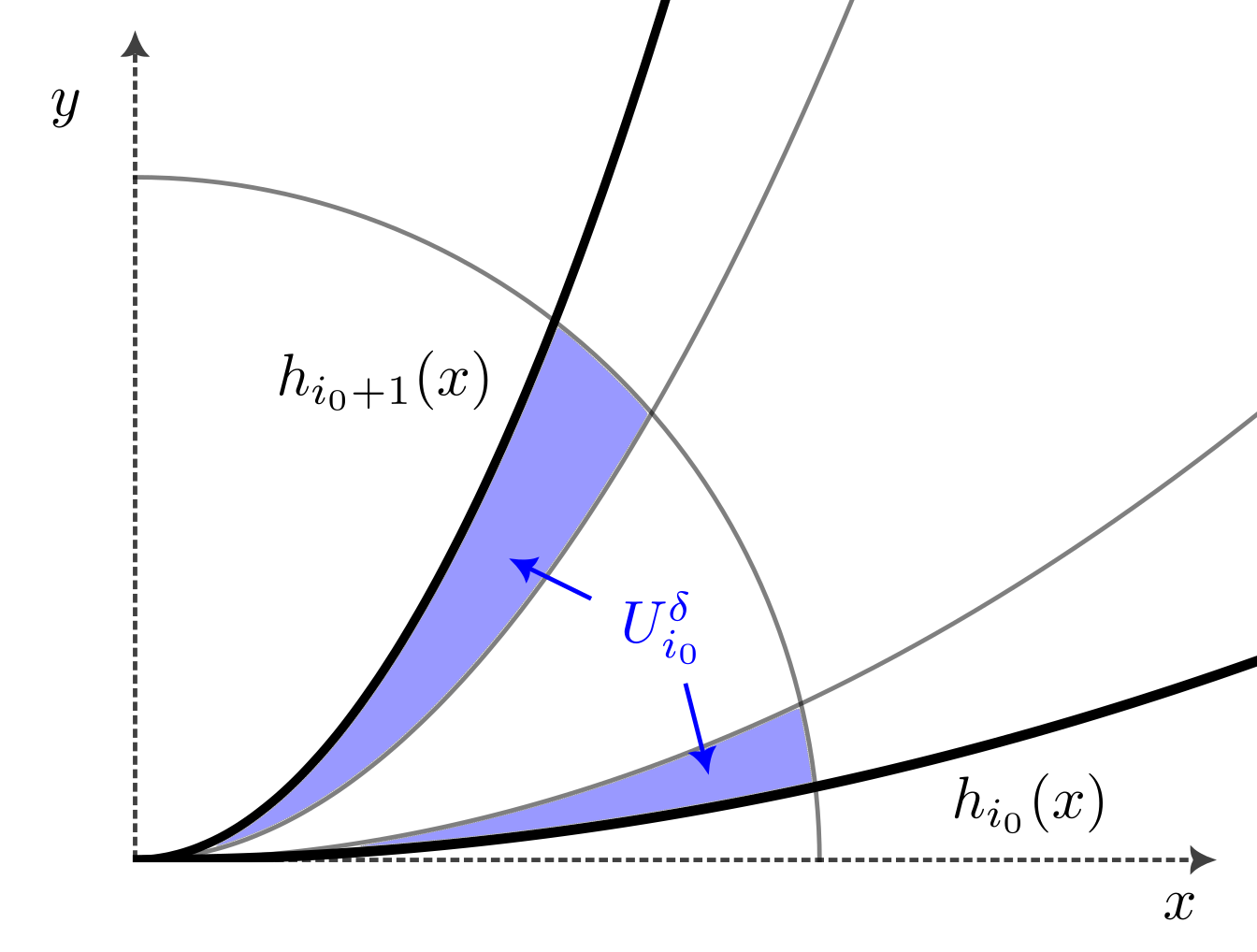}\end{center}
\caption{The region $U_{i_0}^\delta$ when the boundary curves $h_{i_0}(x)$ and $h_{i_0+1}(x)$ are quadratic polynomials.}\label{fig:ui0}
\end{figure}

Note that for $\ell_{i_0} \coloneqq \min \{ k_{i_0}, k_{i_0+1} \}$ and $(x,y) \in U_{i_0}$, we have $\abs{y} \lesssim \abs{x}^{\ell_{i_0}}$, i.e.~$\abs{y}\leq C \abs{x}^{\ell_{i_0}}$ for some constant $C>0$. In addition, from \eqref{eqn:gi_ser}, we see that for $i \ne i_0$ and such $(x,y)$,
\begin{align}
\abs{g_i(x,y)} \lesssim \abs{x}^{\min\{ k_i, \ell_{i_0}\}} \label{eqn:gi_ubd}
\end{align}
and moreover,
\begin{equation}\label{eqn:G_ubd}
\abs{g_{i_0}(x,y)}  \lesssim \abs{x}^{\ell_{i_0}}, \; \abs{G(x,y)} \lesssim \abs{x}^{\sum_{i=1}^m \min\{k_i, \ell_{i_0}\}}.
\end{equation}
We can compute directly
\begin{align*}
\partial_y G & = \sum_{i=1}^m \partial_y g_i \prod_{j \ne i} g_j, \quad \partial_y^2 G  = \sum_{i=1}^m \partial_y^2 g_i \prod_{j \ne i} g_j + \sum_{i=1}^m \sum_{j \ne i} \partial_y g_i \partial_y g_j \prod_{k \ne i, j} g_j.
\end{align*}
Since $\nabla G$ and $\nabla^2 G$ are bounded in a neighborhood of $(0,0)$, it follows that for all $(x,y) \in U_{i_0}$,
\begin{align}
\abs{ \partial_y G (x,y)}  \lesssim \abs{x}^{\sum_{i=1}^m \min\{k_i, \ell_{i_0}\} -  \ell_{i_0}}\quad\text{and}\quad \abs{ \partial_y^2 G (x,y)} & \lesssim \abs{x}^{\sum_{i=1}^m \min\{k_i, \ell_{i_0}\} - 2 \ell_{i_0}} \label{eqn:g_yy}
\end{align}
by \eqref{eqn:gi_ubd} and the fact that $\min \{ k_i, \ell_{i_0}\} \le \ell_{i_0}$. Of course, similar estimates hold for $\partial_x \partial_y G$ and $\partial_x^2 G$, so we obtain $\abs{\nabla^2 G(x,y)} \lesssim \abs{x}^{\sum_{i=1}^m \min \{k_i, \ell_{i_0}\} - 2 \ell_{i_0}}$ in $U_{i_0}$.

Next we estimate $\abs{\nabla G(x,y)}$ from below in $U_{i_0}^\delta$. From our computation of $\partial_y G$,
\begin{align*}
\partial_y G(x, h_{i_0}(x)) & = \partial_y g_{i_0}(x, h_{i_0}(x)) \prod_{i \ne i_0} g_i(x, h_{i_0}(x)).
\end{align*}
Recalling that $\partial_y g_i(0,0) \ne 0$, then for $ \tau$ chosen sufficiently small, we have
\begin{align}
\abs{ \partial_y G(x, h_{i_0}(x))} \gtrsim \prod_{i \ne i_0} \abs{g_i(x, h_{i_0}(x))}. \label{eqn:grad_G_lbd}
\end{align}
We may estimate from below for $i \ne i_0$,
\begin{equation}\label{eqn:gi_lbd}
\begin{split}
\abs{g_i(x, h_{i_0}(x))} & = \abs{\int_0^1 \dfrac{d}{ds} \left( g_i(x, sh_{i_0}(x) + (1-s) h_i(x)) \right) \; ds }  \\
& = \abs{\left(h_{i_0}(x) - h_i(x) \right)  \int_0^1 \partial_y g_i(x, sh_{i_0}(x) + (1-s) h_i(x)) \; ds}  \\
& \gtrsim \abs{h_{i_0}(x) - h_i(x)}  \gtrsim \abs{x}^{\min\{ k_i, \ell_{i_0}\}},
\end{split}
\end{equation}
where in the second to last inequality, we again used that $\partial_y g_i(0,0) \ne 0$ (and take $\tau$ small), and in the last, we are simply using the definition of the $k_i$ and the expansions of the $h_i$. In conjunction with \eqref{eqn:grad_G_lbd}, we see that
\begin{align*}
\abs{\partial_y G(x, h_{i_0}(x))} \gtrsim \abs{x}^{\sum_{i \ne i_0} \min \{k_i, \ell_{i_0} \}},
\end{align*}
as long as $\tau$ is taken sufficiently small, and similarly for $\abs{\partial_y G(x, h_{i_0+1}(x))}$. Along with \eqref{eqn:g_yy}, this gives by the mean value theorem that
\begin{equation}\label{eqn:grad_G_lbd2}
\begin{split}
\abs{\nabla G(x, y)} \ge \abs{\partial_y G(x,y)} & \gtrsim \abs{x}^{\sum_{i =1}^m \min \{k_i, \ell_{i_0} \} - \ell_{i_0}}   =  \abs{x}^{\sum_{i \ne i_0}^m \min \{k_i, \ell_{i_0} \}}
\end{split}
\end{equation}
when $(x,y) \in U_{i_0}^\delta$ provided that $\delta$ is chosen small enough.

Set $\ell^* = \sum_{i=1}^m \min\{k_i, \ell_{i_0}\}$ for convenience. By \eqref{eqn:G_ubd}, we have that, as long as $\lambda >0$ is chosen sufficiently small (but not depending on $\epsilon$), $G - \epsilon F \ne 0$ in $B_{(\lambda \epsilon)^{1/\ell^*}}(0)$, since $F \gtrsim 1$ there while $\abs{G} \lesssim \lambda \epsilon$. Now if $(x,y) \in U_{i_0}^\delta$, and $(x,y) \not \in B_{(\lambda \epsilon)^{1/\ell^*}}(0)$, then by estimate \eqref{eqn:grad_G_lbd2},
\begin{align*}
\abs{\partial_y G(x,y)}  \gtrsim \abs{x}^{\ell^* - \ell_{i_0}}   \gtrsim (\lambda \epsilon)^{1 - \ell_{i_0}/\ell^*},
\end{align*}
since $(x,y) \in U_{i_0}$ implies $\abs{y} \lesssim \abs{x}^{\ell_{i_0}}$ with $\ell_{i_0} \ge 1$, so $\abs{x} \simeq \abs{(x,y)}$ there. Since $F$ is $C^1$ in a neighborhood of the origin, then as long as $\tau$ and $\epsilon_0$ are chosen sufficiently small, then $\epsilon \norm{\nabla F}_{L^\infty(B_{r/2}(0))} \le \abs{ \partial_y G(x,y)}/8$ for such points, and thus in $U_{i_0}^\delta \setminus B_{(\lambda \epsilon)^{1/\ell^*}}(0) $,
\begin{align}
\abs{\partial_y( G - \epsilon  F)}  \ge \abs{\partial_y G} - \epsilon \abs{\nabla F}  \ge (7/8) \abs{\partial_y G} > 0. \label{eqn:grad_H_lbd}
\end{align}
In particular, the implicit function theorem implies that locally near any zero $(x, y)$ of $G - \epsilon F$ in $U_{i_0}^\delta$, the set $\{G - \epsilon F = 0\}$ is the graph of a $C^1$ function over the $x$-axis.

For $(x,y) \in U_{i_0} \setminus U_{i_0}^{\delta/2}$, we apply Lemma \ref{lem:z_pert_1}. In particular, for $t \in [\delta/2, 1- \delta/2]$, write $y_t(x) \coloneqq t h_{i_0}(x) + (1-t) h_{i_0+1}(x)$,
\begin{align*}
g_t(x)  \coloneqq G(x, y_t(x)) \quad\text{and}\quad f_t(x)  \coloneqq F(x, y_t(x)).
\end{align*}
Since $h_{i_0}$, $h_{i_0+1}$, and $G$ are real-analytic in a neighborhood of the origin, so is $g_t$, and similarly, $f_t$ is Lipschitz since $F$ is (with Lipschitz constant independent of $t$). The fact that $G(0,0) = h_{i_0}(0) = h_{i_0+1}(0) = 0$ and $\nabla G(0,0) =0$ forces the Taylor expansion of $g_t$ at the origin to have leading term $a(t) x^{k(t)}$ for some $k(t) \ge 2$ and $a(t) >0$, since $g_t$ is a non-trivial, non-negative analytic function that is positive in $\{x >0\}$. (Recall that $U_{i_0}$ is a positive chamber for $G$.) In fact, $k(t) \equiv \ell^*$ for all such $t$, as can be seen from the estimates \eqref{eqn:G_ubd} and an estimate similar to \eqref{eqn:gi_lbd}, from which one deduces that $\abs{G(x, y_t(x))} \simeq_t x^{\ell^*}$ for all $x$ sufficiently small, which forces $k(t) = \ell^*$. Thus, Lemma \ref{lem:z_pert_1} gives for each $t \in [\delta/2, 1 -\delta/2]$, some $\epsilon_0(t)$ and $\tau(t)$ so that the function $g_t - \epsilon f_t$ has exactly one root in $(0, \tau(t))$ for all $\epsilon < \epsilon_0(t)$. Since $t$ lives in a compact interval and $a(t)$ is continuous and strictly positive, then it attains its positive minimum and finite maximum in $[\delta/2, 1-\delta/2]$. In particular, $\epsilon(t)$ and $\tau(t)$ can be chosen to depend only on the maximum and minimum of $a(t)$. By Remark \ref{rmk:zero}, we can choose $\epsilon_0$ and $\tau$ small independently of $t$ so that satisfies the conclusion of Lemma \ref{lem:z_pert_1} for all $t \in [\delta/2, 1-\delta/2]$. Of course, $\epsilon_0$ and $\tau$ will depend on $\delta$, but this is a small, fixed quantity depending on $F$ and $G$.

We are in a position to conclude the proof. Our work so far shows that at any zero $z \in \nodal(G- \epsilon F) \cap U_{i_0}$, there is a neighborhood $V_z$ of $z$ such that $ \nodal(G - \epsilon F)\cap V_z $ coincides with a curve passing through $z$: in $U_{i_0}^\delta$ this is by the implicit function theorem and in $U_{i_0} \setminus U_{i_0}^{\delta/2}$ from the fact that the (unique) zeros of $G - \epsilon F$ along the trajectories $x \ra (x, y_t(x)) \in U_{i_0}$ indexed by $t \in [\delta/2, 1-\delta/2]$ form a continuous curve in $t$. Let $\gamma$ denote some connected component of $\nodal(G - \epsilon F) \cap \overline{U_{i_0}}$. If $\gamma$ meets $U_{i_0}^\delta$, define
\begin{align*}
x_{\mathrm{min}} = \inf \{ x >0 \; : \; (x,y) \in U_{i_0}^\delta \cap \gamma \text{ for some } y \},
\end{align*}
which exists since $G-\epsilon F \ne 0$ in $B_{(\lambda \epsilon)^{1/\ell^*}}(0)$. Choose a minimizing sequence $x_k \ra x_{\mathrm{min}}$ with corresponding points $y_k$ such that $(x_k, y_k) \in U_{i_0}^\delta \cap \gamma$, and note that we may assume (up to a subsequence) that $(x_k, y_k) \ra (x_{\mathrm{min}}, y_{\mathrm{min}})$ for some $y_{\mathrm{min}}$. Note that $(x_{\min}, y_{\min})$ is a zero of $G - \epsilon F$, and moreover, $y_{\min} = y_t(x_{\min})$ for either $t = \delta$ or $t = 1-\delta$. Indeed $G - \epsilon F$ is nonzero on the graphs of $h_{i_0}$ and $h_{i_0+1}$, and \eqref{eqn:grad_H_lbd} says that if $(x_{\min}, y_{\min}) \in \mathrm{int}( U_{i_0}^\delta)$, then locally near this point, $\nodal(G - \epsilon F)$ coincides with the graph of a $C^1$ function over the $x$-axis, which would contradict the definition of $x_{\min}$. By the remarks above, there is a neighborhood $V$ of $(x_{\min}, y_{\min})$ in which $\nodal{(G - \epsilon F)}$ coincides with a curve passing through $(x_{\min}, y_{\min})$, which therefore must coincide with $\gamma \cap V$. Hence $(x_{\min}, y_{\min}) \in \gamma \, \cap (U_{i_0} \setminus U_{i_0}^{\delta/2})$. We have proved that every connected component of $\nodal(G - \epsilon F) \cap \overline{U_{i_0}}$ meets $U_{i_0} \setminus U_{i_0}^{\delta/2}$.

Recall that $\nodal(G - \epsilon F) \cap (U_{i_0} \setminus U_{i_0}^{\delta/2})$ has exactly one connected component, given by the curve of zeros of $G - \epsilon F$ along the trajectories $x \ra (x, y_t(x))$. Since any connected component of $\nodal(G - \epsilon F) \cap \overline{U_{i_0}}$ meets this region, it follows that $\nodal(G - \epsilon F) \cap \overline{U_{i_0}}$ has exactly one connected component and we know that this component is a simple curve. We leave it to the reader to verify that the curve tends to $\{G = 0\}$ near the origin as $\epsilon \da 0$.
\end{proof}

\renewcommand{\baselinestretch}{1}

\bibliographystyle{amsbeta}
\bibliography{bibl}

\end{document}